\documentclass[11pt]{article}
\usepackage{amsmath,amsthm,amssymb,amscd,fancybox,ifthen,float,epsfig,subcaption}
\usepackage[all]{xy}
\usepackage{times}
\usepackage{hyperref}

\newcommand{\mar}[1]{{\marginpar{\sffamily{\scriptsize #1}}}}
\newcommand\cle[1]{\mar{{\bf CLE}: #1}}

\floatstyle{plain}
\restylefloat{figure}

\usepackage{color}

\definecolor{gr}{rgb}   {0.,   0.8,   0. }
\definecolor{bl}{rgb}   {0.,   0.5,   1. }
\definecolor{mg}{rgb}   {0.7,  0.,    0.7}

\newcommand{\Bk}{\color{black}}

\newcommand{\Bl}{\color{black}}

\newcommand\mf{\operatorname{mf}}
\newcommand\Sc{\operatorname{sc}}
\newcommand\cf{\operatorname{cf}}
\newcommand\df{\operatorname{df}}
\newcommand\scH{{}^{\Sc}H}
\newcommand\ff{\operatorname{ff}}

\newcommand\WG{\operatorname{WG}}

\newcommand\Diff{\operatorname{Diff}}

\newcommand\pp{\operatorname{pp}}

\newcommand\sign{\operatorname{sig}}

\newcommand\WF{\operatorname{WF}}

\newcommand\In{\operatorname{in}}
\newcommand\Int{\operatorname{int}}

\newcommand\out{\operatorname{out}}

\newcommand\cT{\mathcal T}
\newcommand\obbRd{\overline{\bbR^d}}
\newcommand\obbR[1]{\overline{\bbR^{#1}}}
\newcommand\cM{\mathcal M}
\newcommand\IM{\mathcal I_M}

\newcommand\cH{\mathcal H}
\newcommand\cI{\mathcal I}

\newcommand\cR{\mathcal R}

\newcommand\bm{\boldsymbol m}
\newcommand\bM{\boldsymbol M}

\newcommand\fd{\mathfrak d}
\newcommand\ellk{\mathfrak l}

\newcommand\cF{\mathcal F}
\newcommand\cE{\mathcal E}

\newcommand\tH{\widetilde{H}}

\newcommand\tQ{\widetilde{Q}}
\newcommand\hQ{\widehat{Q}}
\newcommand\txi{\tilde{\xi}}

\newcommand\teta{\tilde{\eta}}

\newcommand\cxi{\check{\xi}}
\newcommand\cuxi{\breve{\xi}}
\newcommand\cutheta{\breve{\theta}}

\newcommand\ceta{\check{\eta}}

\newcommand\hx{\widehat{x}}

\newcommand\hN{\widehat{N}}
\newcommand\hpsi{\widehat{\psi}}

\newcommand\cA{\mathcal{A}}
\newcommand\cC{\mathcal{C}}
\newcommand\cL{\mathcal{L}}
\newcommand\fL{\mathfrak{L}}
\newcommand\cS{\mathcal{S}}

\newcommand\cD{\mathcal{D}}

\newcommand\hf{\hat{f}}

\newcommand\hH{\hat{H}}

\renewcommand\Im{\operatorname{Im}}

\newcommand\bbR{\mathbb R}

\newcommand\bbS{\mathbb S}

\newcommand\pa{\partial}

\newcommand\restrictedto{\upharpoonright}

\newcommand\supp{\operatorname{supp}}

\newcommand\CI{{\mathcal C}^{\infty}}

\newcommand\Id{\operatorname{Id}}

\newpage

\newcommand{\RR}{\mathbb{R}}

\newcommand\half{{\frac{1}{2}}}
\newcommand\ep{\epsilon}
\newcommand\epv{\varepsilon}

\newcommand{\Cinf}{C^\infty}

\newcommand\Range{\operatorname{Ran}}

\newcommand\sci{{}^{\mathrm{sc}}}
\newcommand\sct{\sci T^*}

\newcommand\scl{{\mathrm{sc}}}
\newcommand\Scl{{\mathrm{Sc}}}

\newcommand\DiffSc{\Diff_\Scl}

\newcommand\PsiSc{\Psi_\Scl}

\newcommand\Psisc{\Psi_{\scl}}

\newcommand\Hsc{H_{\scl}}

\newcommand\phit{\tilde\phi}

\newcommand\Xb{\bar X}

\newtheorem{theorem}{Theorem}[section]
\newtheorem{proposition}[theorem]{Proposition}
\newtheorem{corollary}[theorem]{Corollary}
\newtheorem{lemma}[theorem]{Lemma}

\theoremstyle{definition}
\newtheorem{definition}[theorem]{Definition}

\theoremstyle{remark}
\newtheorem{remark}[theorem]{Remark}

\numberwithin{equation}{section}

\begin{document}

\title{Radiation Conditions and Uniqueness Theorems\\ for Open Wave-guide Networks}

\author{Charles L. Epstein\footnote{ Center for Computational
  Mathematics, Flatiron Institute, 162 Fifth Avenue, New York, NY
  10010. E-mail: {cepstein@flatironinstitute.org}.},\, Rafe
  Mazzeo\footnote{Dept. of Mathematics, Stanford University, Stanford,
  CA. E-mail: rmazzeo@stanford.edu},\, Solomon Quinn\footnote{ Center
  for Computational Mathematics, Flatiron Institute, 162 Fifth Avenue,
  New York, NY 10010. E-mail: {squinn@flatironinstitute.org}.}, and
  Andr\'as Vasy\footnote{Dept. of Mathematics, Stanford University,
  Stanford, CA. E-mail: vasy@stanford.edu} } \date{October 14, 2026}

\maketitle
\begin{abstract}
We state radiation conditions and prove a variety of uniqueness
theorem for models of wave-guide networks, including acoustic
wave-guide networks, and electromagnetic wave-guide networks described by the time
harmonic Maxwell equations. To analyze Maxwell wave-guide networks we
need to establish generalizations of the Mourre estimates
from~\cite{VasyAsterisque,VasyExpDecay}. 
\end{abstract}

\centerline{\bf This paper is dedicated to the memory of Peter Lax
  (1926-2025),}
  \centerline{\bf who
  was a great inspiration to us all.}

\section{Introduction}
In a recent series of papers {\Bl a subset of the authors has} \Bk studied analytic and numerical problems
connected to scattering from open wave-guide networks. In these papers, an open
wave-guide network is modeled as a perturbation of the Laplacian on $\bbR^d$ of the
form\footnote{In this paper the Laplace operators, $\Delta$ or $\Delta_g,$ are
understood to be  non-negative.}
\begin{equation}\label{eqn1}
  H_q=\Delta-k^2(1+q(x)),
\end{equation}
where $q(x)$ is a potential whose support, outside of a ball, is
contained in \emph{channels}\footnote{This usage differs from that in
the study of $N$-body Schr\"odinger equations, where these are called
collision planes, {\Bl while channels there are given by bound states of
particular subsystems, i.e.\ wave guide modes in our setting.} \Bk}, which are cylindrical neighborhoods of a finite
collection of rays, $\cR=\{\ell_1,\dots,\ell_m\}.$ In $2d$ this could
be a model for the TM-mode in an non-magnetic (i.e. $d\mu=0$),
translation invariant, electromagnetic wave-guide, where
$k^2(1+q(x))=\omega^2\mu \epsilon(x).$ 
It is assumed that the time dependence is
harmonic, that is, the solution takes the form
\begin{equation}
  U(x,t)=e^{-i\omega t}u(x),
\end{equation}
see \cite{ChandlerZhang}.  The operator $H_q$ also gives a standard
model for an acoustic wave-guide network where $k^2(1+q(x))$ is the
reciprocal of the squared sound speed.

An operator, like that in~\eqref{eqn1}, defines a \emph{bi-infinite}
wave-guide if it is invariant under translation in some direction, say
$x\mapsto x+t\theta,$ for $t\in\bbR.$ In this case $q$ is a function of
$\pi(x)=x-\langle x,\theta\rangle\theta.$ We use the term
\emph{semi-infinite} wave-guide to refer to an operator like that
in~\eqref{eqn1}, which takes this form for $\langle x,\theta\rangle$
sufficiently large.

Let $\cC=\{\theta_1,\dots, \theta_m\}\subset\bbR^d$ be distinct
unit vectors such that there are positive numbers $\{r_j\},$ for
which:
\begin{equation}
  \ell_j=\{\alpha \theta_j:\:\alpha>r_j\}.
\end{equation}
We call these rays the \emph{channel ends}.
If we let $\obbRd$ denote the radial compactification of $\bbR^d,$
then the set $\cC,$ is identified with points on
$\pa\obbRd.$ Let $\pi_j(x)=x-\langle x,\theta_j\rangle \theta_j;$ in
the cylindrical neighborhood of $\ell_j,$ $q(x)$ is independent of the
\emph{channel radial variable} $\langle x,\theta_j\rangle$ and is only
a function of the transverse variables, that is, the image of
$\pi_j(x){\Bl \in\theta_j^\perp.} \Bk$ That is, the channel ends are semi-infinite wave-guides.

The main results of~\cite{EpMaWG3}, which are adaptations of results
in~\cite{VasyAsterisque}, give outgoing radiation conditions that
imply uniqueness for solutions of $H_qu=f,$ for $f$ with sufficiently
rapid decay. The details of the geometry underlying these sorts of
operators is described in Section~\ref{wg_geom}. Generally speaking we
call such a collection of rays, along with a compact interaction zone,
a \emph{wave-guide network geometry}, or \emph{wave-guide geometry},
and an operator whose coefficients have the dependency on variables
described above a \emph{wave-guide network operator}.

There is a great deal of mathematical, physical and engineering
literature on the subject of wave-guides,
see~\cite{Christiansen,CiraoloMagnanini2008,
  ChandlerZhang,DebevierePravica1, DebevierePravica2, JolyPoirier,
  KNH_SIAM_2005,Marcuse,MagnaniniSantosa, Nosich1994}. Most of these
papers are concerned with the bi-infinite case and perturbations of
this case. To the best of our knowledge~\cite{EpMaWG3} is the only
work that addresses radiation conditions and uniqueness questions for
anything more complicated. We employ positive commutator techniques (Mourre
estimates), which are most closely related to the study of uniqueness
results for $N$-body Schr\"odinger operators, see~\cite{CFKS,
  FroeseHerbst1, FroeseHerbst2,
  Isozaki94,Melrose94,VasyAsterisque}. As far as we are aware there
are no previous results on radiation conditions or uniqueness results
for wave-guide networks modeled by Maxwell's equations.

In this paper we give extensions of the results of~\cite{EpMaWG3} to
other physically interesting classes of equations, which are naturally
defined in a wave-guide geometry.  The results of our previous paper
rely heavily on the analysis in~\cite{VasyAsterisque} of propagation
of singularities ``at infinity'' for these sorts of equations. In that
paper and other similar works, several assumptions are made about the
form of the perturbations, which limit the applicability of these
results, excluding several physically interesting cases, in particular
Maxwell's equations. We assume that the reader is somewhat familiar
with the scattering calculus and 3-scattering calculus, at least at
the level of detail given in~\cite{EpMaWG3}. A more detailed
understanding is needed to read Appendices~\ref{MourreEstGenCase}
and~\ref{prop_est}.

 As a warm-up we obtain radiation conditions and uniqueness theorems
 for a model of acoustic wave-guide networks. Using the 3-scattering
 calculus, we can define radiations conditions that imply uniqueness
 for operators of the form
\begin{equation}
  \tH_{c,\rho}u=c(x)^2\rho(x)\nabla\cdot \rho(x)^{-1}\nabla u,
\end{equation}
where $c(x)$ is the sound-speed and $\rho(x)$ is the density of the
medium. These are assumed to take on constant positive values outside
of a wave-guide network geometry. A similar class of operators, which
are perturbations of bi-infinite waveguides is considered
in~\cite{DebevierePravica1, DebevierePravica2}.

The principal goal of this paper is to give radiation conditions for
electromagnetic wave-guide networks, which are modeled using the
time-harmonic Maxwell equations with spatially varying permittivity
and permeability. In section~\ref{s.maxwell} we consider this model:
\begin{equation}
    d\xi=i\omega\mu(x)\eta,\quad d^*\eta=-i\omega\epsilon(x)\xi,
\end{equation}
the electric field is represented by a 1-form, $\xi,$ and the magnetic
field by a 2-form, $\eta.$ The permeability, $\mu(x)$ and
permittivity, $\epsilon(x)$ are scalar functions that assume positive
constant values outside of a neighborhood of a wave-guide network
geometry, and, within the channel ends, are independent of the channel
radial variables $\{\langle x,\theta_j\rangle\}.$

To analyze this case requires a substantial generalization of known
results.  We need to allow more general first order perturbations, for
which the indicial operators in the channel ends have a term linear in
the rescaled conormal direction, i.e. the $\tau$-direction. These
terms arise from first order derivatives in the channel radial
directions.  The presence of such terms entails significant
modifications to the proof of the Mourre estimate.

Our approach to this problem has been only partially successful.  To
extend the positive commutator estimates to this case we have to
assume that the group velocities and phase velocities of the wave
guide modes have the same signs. The product
$\kappa(x)=\epsilon(x)\mu(x)$ is the reciprocal of the squared light
speed; let $\kappa_0=\sqrt{\epsilon_0\mu_0},$ be the reciprocal of the
free-space wave light-speed, and let $\kappa_j(x)$ be $\kappa(x)$
restricted to the end of the $j$th channel. The condition on the group velocities
holds if either
\begin{equation}
  \sqrt{2}\|d\log \kappa_j\|_{\infty}<\omega\kappa_0,\text{ or }
  \frac{\max\kappa_j(x)}{\kappa_0}< \sqrt{2},
\end{equation}
for all $j.$ These are, respectively, a lower bound
on the frequency, or an upper bound on the contrast ratios, which is
independent of the frequency.

This may point to limitations in our method or it may indicate new
physical phenomenon; to wit:
\begin{enumerate}
\item For small $\omega,$ or large contrast there may exist wave-guide
  modes of the form
  $$e^{i(\tau x_1-\omega t)}[\cxi(x_2,x_3),\ceta(x_2,x_3)],\quad
  \omega,\tau>0$$ which appear to be macroscopically outgoing, but are
  microlocally (at infinity) incoming, i.e. the group velocity and
  phase velocity point in opposite directions.
  \item There may exist a discrete set of frequencies, $\{\omega_j\},$
    for which the wave-guide network supports $L^2$-eigenmodes with
    positive energy.
\end{enumerate}
Of course, neither situation occurs in the scalar case, and either, or
both could be an artifact of the method used to prove uniqueness. In
the closely related context of elastic wave-guides the group velocity
and phase velocity can have opposing signs. The question of
physically correct radiation conditions then does not have a clear-cut
answer. There are many papers addressing these questions primarily in
the Russian literature, see for example~\cite{Nazarov2013}.

\subsection{Contents of the Paper}
In Section~\ref{wg_geom} we discuss the geometry that underlies
wave-guide networks and the scattering and 3-scattering
pseudodifferential operator calculi that are used to state the
radiation conditions and prove the uniqueness results. In
Section~\ref{acous_wg} we give radiation conditions for a model of
acoustic wave-guides. Perturbations of the bi-infinite case were
analyzed in~\cite{Christiansen,DebevierePravica1,
  DebevierePravica2}. In Section~\ref{s.maxwell} we begin our study of
Maxwell wave-guides. We represent the electric field as a 1-form,
$\xi$ and the magnetic field as a 2-form, $\eta.$ In this section we
introduce an elliptic system of second order equations, which
generalizes the vector Helmholtz equation, satisfied by solutions to
the time-harmonic Maxwell equations. In these equations the 1-form and
2-form terms are decoupled; these equations are described by a complex
of first order differential operators:
\begin{equation}
\fd_j: \cC^{\infty}(\bbR^3;\Lambda^j)\to
\cC^{\infty}(\bbR^3;\Lambda^{j+1}),
\, j=0,1,2,
\end{equation}
which satisfy $\fd_{j+1}\fd_j=0.$ The equations then take the form
\begin{equation}\label{eqn7.339}
  \begin{split}
    [\fd_0\fd_1^*+\fd_2^*\fd_1]\xi&=\omega^2\xi\\
     [\fd_1\fd_2^*+\fd_3^*\fd_2]\eta&=\omega^2\eta.
  \end{split}
\end{equation}

In Section~\ref{Max_rad_cond} we present the radiation conditions for
Maxwell wave-guide networks which are described in terms of the
3-scattering wave-front set. Away from the wave-guide channels the
conditions are equivalent to the Silver-M\"uller radiation
conditions. In the channels they distinguish between incoming and
outgoing wave-guide modes. To prove the uniqueness theorems we first
prove a Mourre-type estimate for the operators in~\eqref{eqn7.339},
following the microlocal method used in~\cite{VasyAsterisque}. To that
end we first need to analyze the spectral theory of  bi-infinite
wave-guide operators, and in particular the wave-guide modes they
supports. Similar results are given in~\cite{JolyPoirier}. This leads
to definitions of characteristic varieties and radial sets in the
3-scattering sense. The radiation conditions we give are phrased in
terms of these sets. The section concludes with a
proof of the radiation condition.

In Section~\ref{ss.mourre} we state the generalized Mourre estimate
needed for the uniqueness argument. This section uses the notation and
terminology from~\cite{VasyAsterisque}.  As noted above, our proof of
uniqueness only works under an assumption on the group velocities of
the relevant wave-guide modes, see Proposition~\ref{prop4}. This
condition holds if $\sqrt{2}\|d\log\kappa\|_{\infty}<\omega\kappa_0,$
or if the contrast ratios are less than $\sqrt{2},$ independently of
the frequency. We are also not able to exclude the possibility of a
discrete set of frequencies for which finite dimensional spaces of
positive energy $L^2$-eigenmodes might exist.

The paper concludes with several appendices containing proofs of the
theorems. The first appendix gives a relatively explicit, simple proof
of the Mourre estimate for wave-guide networks. The second appendix
gives a proof of the Mourre estimate for more general $N$-body
Schr\"odinger operators. Appendix~\ref{prop_est} describes
modifications of the proofs of various propagation of singularities
results from~\cite{VasyAsterisque} that are needed to handle the more
general class of operators considered here. Finally,
Appendix~\ref{app1} contains useful identities from Riemannian
geometry.

\section{Wave-guide geometry}\label{wg_geom}
While most of our applications are described for operators defined on
Euclidean space, the results hold in much greater generality. In this
section we describe the general geometric setting, taken
from~\cite{VasyAsterisque}, where our results hold, {\Bl in part as the
geometric perspective leads one to focus on the key analytic and
geometric, rather than algebraic, aspects.} \Bk

Let $X$ be a compact smooth manifold with boundary. If $s$ is a
defining function for $\pa X,$ and $(y_1,\dots,y_{d-1})$ coordinates
on a open subset $U\subset \pa X,$ then a scattering metric is one
which takes the form
\begin{equation}\label{metric0}
  g_0=\frac{ds^2}{s^4}+\frac{1}{s^2}\sum_{i,j}h_{ij}dy_idy_j,
\end{equation}
on $(0,\epsilon)\times U$ and the sum is positive definite metric on
$\pa X=\{s=0\}.$ We call such coordinates \emph{boundary adapted coordinates}.
The volume of the ball of radius $R$ grows like
$R^d,$ that is, like that of $\bbR^d.$ Indeed, the radial
compactification of $\bbR^d$ with the flat metric, which we denote by
$\obbRd,$ satisfies these hypotheses. {\Bl Recall that the radial
compactification simply glues a sphere at infinity to $\RR^d_w$ using
the product structure $(0,\infty)_r\times\bbS^{d-1}_\omega$ in polar
`coordinates', identifying $(0,1)_s\times\bbS^{d-1}_\omega$ in
$[0,1)_s\times\bbS^{d-1}_\omega$ (with $s=0$ the boundary of the
compactification) with $(1,\infty)_r\times\bbS^{d-1}_\omega$ via
$(s,\omega)\mapsto(s^{-1},\omega)$. A similar construction works at
the `large' end of a cone; thus metrics as in \eqref{metric0} are also
called {\em asymptotically conic}.} \Bk

{\Bl The locus of the phase space analysis is the scattering cotangent
bundle $T^*_{\Sc}X$,} \Bk which is trivialized near the
boundary by the approximately unit length covectors
$\left\{\frac{ds}{s^2},\frac{dy_1}{s},\dots,\frac{dy_{d-1}}{s}\right\}.$
Locally a section takes the form
\begin{equation}
\tau(s,y)\frac{ds}{s^2}+\sum_{j}\mu_j(s,y)\frac{dy_j}{s};
\end{equation}
we let $\zeta=(\tau,\mu)$ denote fiber coordinates for $T^*_{\Sc}(X)$
in a boundary adapted coordinate system. Observe that
\begin{equation}
  |\zeta|^2=\tau^2+\sum_{i,j}h_{ij}\mu_i\mu_j.
\end{equation}
We call $\tau$ the conormal direction along $\pa X.$ {\Bl Note that in
  the case of $\obbRd$, this trivialization corresponds again to
  `polar coordinates', i.e.\ with respect the decomposition
  $(0,\infty)_r\times\bbS^{d-1}$, $\frac{ds}{s^2}=-dr$, and with
  coordinates $y_j$ on the sphere, $\frac{dy_j}{s}=r^{-1}\,dy_j$. Note
also that, as familiar from writing the Euclidean metric in polar
coordinates, the standard coordinates $w_j$ give equivalently another
trivialization, in this case globally, via $dw_j$. Thus,
$T^*_{\Sc}\obbRd=\obbRd\times(\bbR^d)^*$ simply.} \Bk

On this space we define the calculi of scattering differential and 
pseudodifferential operator, $\Diff_{\Sc}^m(X)\subset \Psi^{m,l}_{\Sc}(X).$
The differential operators are easiest to describe: they are
represented locally as polynomials in the vector fields, dual to the
covectors above, $\{s^2\pa_s,s\pa_{y_1},\dots,s\pa_{y_{d-1}}\}.$ A differential
operator of order $m$ has the representation in an open neighborhood
$U$ of $\pa X,$
\begin{equation}\label{scat_diff}
  P(x,D)\restrictedto_{U}=\sum_{j=0}^m\sum_{\{\alpha:|\alpha|\leq
    m-j\}}a_{j,\alpha}(s,y)[s^2\pa_s]^j[s\pa_y]^{\alpha},
\end{equation}
where the coefficients $a_{j,\alpha}\in\cC^{\infty}(U).$ {\Bl
  Corresponding to the above description of $T^*_{\Sc}\obbRd$, which
  implies that $T_{\Sc}\obbRd=\obbRd\times\bbR^d$, with
  trivialization given by the $\pa_{w_j}$, equivalently these are vector fields,
  resp.\ differential operators, built from $\pa_{w_j}$ with
  coefficients that are smooth on the compactified space $\obbRd$.} \Bk

For the sake of simplicity we only define the scattering calculus on
$\obbRd.$ The general case follows by coordinate invariance, from the
$\obbRd$-case. This is a bi-filtered algebra, $\Psi^{m,l}_{\Sc}(\obbRd),$
where $m$ is the order of the operator, which agrees with usual order
of a pseudodifferential operator in $\bbR^d.$ The second index $l$ is
the order of vanishing (or growth) at the $\pa\obbRd.$ Let $w$ 
denote linear coordinates on $\bbR^d$ and $\xi$ dual coordinates on
the fiber of $T^*\bbR^d,$ {\Bl a function
$a(w,\xi)\in\cC^{\infty}(\bbR^d\times\bbR^d)$ defines the symbol of an
operator of order $(0,0),$ provided that, for all pairs of $d$-multi-indices,
$(\alpha,\beta),$ we have the symbolic estimates
\begin{equation}
 | \pa_w^{\alpha}\pa_{\xi}^{\beta} a(w,\xi)|\leq
 \frac{C_{\alpha\beta}}{(1+|w|)^{|\alpha|}(1+|\xi|)^{|\beta|}}.
\end{equation}
We denote the set of such functions by $S^{0,0}_{\Sc}(\obbRd).$
The space $\cC^{\infty}(\obbRd\times\obbRd)$ is a subspace of
$S^{0,0}_{\Sc}(\obbRd)$, called the space of classical symbols;
elements of this
have a one step polyhomogeneous expansion corresponding to their
Taylor series at the boundary; equivalently or more traditionally this expansion can be
phrased in powers of $r^{-1}$ with smooth coefficients on the sphere.} \Bk

The
action of the operator is defined for $f\in\cS(\bbR^d)$ by
\begin{equation}
  A(X,D)f(w)=\frac{1}{(2\pi)^d}\int_{\bbR^d}a(w,\xi)e^{i\xi\cdot w}\hf(\xi)d\xi.
\end{equation}
An operator in $\Psi^{m,l}_{\Sc}(\obbRd)$ is similarly defined by a symbol 
\begin{equation}
  a(w,\xi)=(1+|\xi|^2)^{\frac m2}(1+|w|^2)^{-\frac{l}{2}}a'(w,\xi)\text{ where }a'\in S^{0,0}_{\Sc}(\obbRd).
\end{equation}
We denote these spaces of symbols by $S^{m,l}_{\Sc}(\obbRd).$

{\Bl These operators have a principal symbol at
  $\pa(\obbRd\times\obbRd)$, i.e.\ at infinity either in position or
  in momentum, defined by
\begin{equation}
 \sigma_{m,l}(A)(w;\xi)=a(w,\xi)\mod S^{m-1,l+1};
\end{equation}
for classical symbols of order $(0,0)$ this can be identified with
restriction to $\pa(\obbRd\times\obbRd)$, and the same holds for
general classical symbols after renormalization. It is sometimes
convenient to break this up into two
principal symbols: an interior symbol
$$
\sigma_{m}(A)(w;\xi)=a(w,\xi)\mod S^{m-1,l},
$$
and a normal symbol, $\hN_l(A)$ which for classical $a$ is a function on
$T^*_{\Sc}(\obbRd)\restrictedto_{\pa\obbRd}.$ Let $(s,y)$ denote
boundary adapted coordinates. The normal symbol of a classical operator
$A\in\Psi^{m,l}_{\Sc}(\obbRd)$ is defined by restricting
$s^{-l}a(s,y;\tau,\mu)$ to $\pa\obbRd=\{s=0\}.$ }

The symbol $\sigma_m(A)$ captures the leading differential order  of the
operator; for a operator
$P\in\Diff^{m}_{\Sc}(\obbRd)$ we have, in boundary adapted coordinates,
\begin{equation}
  \sigma_m(P)(s,y;\tau,\mu)=\sum_{\{(j,\alpha):j+|\alpha|=m\}}a_j(s,y)(i\tau)^j[i\mu]^{\alpha}.
\end{equation}
The normal symbol involves lower order terms as well {\Bl but only at
  the boundary of $\obbRd,$ as it is more a
measure of order of vanishing there.} \Bk If
$P\in\Psi^{m,l}_{\Sc}(\obbRd)$ is a differential operator, then
  \begin{equation}
    \hN_l(P)(y;\tau,\mu)=\sum_{j=0}^m
    \sum_{\{\alpha:|\alpha|\leq m-j\}}[s^{-l}a_j(s,y)]\restrictedto_{s=0}(i\tau)^j[i\mu]^{\alpha}.
  \end{equation}
  For example, if $q\in\cC^{\infty}(\obbRd),$ then
  $\Delta_g+q-k^2\in\Diff^2_{\Sc}(\obbRd),$ with
  \begin{equation}
    \sigma_2(\Delta_g+q-k^2)=|\xi|_g^2\text{ and }
    \hN_0(\Delta_g+q-k^2)=|\zeta|^2+q\restrictedto_{\pa\obbRd}-k^2.
  \end{equation}

 For $m\in\bbR,$ the space $H_{\Sc}^{m,0}(\bbR^d)$ is the usual order $m$
 $L^2$-Sobolev space. Let $\rho\in\cC^{\infty}(\obbRd)$ be a global defining function for
 $\pa\obbRd;$ for $l\in\bbR$ set
\begin{equation}
  H_{\Sc}^{m,l}(\bbR^d)=\rho^lH_{\Sc}^{m,0}(\bbR^d).
\end{equation}
For all $(m,l)$ the space $H_{\Sc}^{m,l}(\bbR^d)\subset\cS'(\bbR^d).$
The Fourier transform, defined on $\cS(\bbR^d)$ by
\begin{equation}
  \cF u(\xi)=\int_{\bbR^d}e^{-i\xi\cdot w}u(w)dw
\end{equation}
extends by duality to $\cS'(\bbR^d).$ It defines isomorphisms
\begin{equation}
  \cF: H_{\Sc}^{m,l}(\bbR^d)\to H_{\Sc}^{l,m}(\bbR^d).
\end{equation}
An operator $A\in\Psi^{m,l}_{\Sc}(\obbRd)$ extends to define bounded maps
\begin{equation}
  A:H_{\Sc}^{p,q}(\bbR^d)\to H_{\Sc}^{p-m,q+l}(\bbR^d)\text{ for all }p,q\in\bbR.
\end{equation}
With these definitions in hand we can now define the scattering wave-front set,
which is used to describe the radiation conditions for Helmholtz-like operators
with short range potentials. To define radiation conditions for wave-guide operators
requires a further blow-up and a somewhat more complicated class of
pseudodifferential operators. This is because a wave-guide potential
does not extend to be a smooth function on the radial compactification $\obbRd.$

{\Bl Since we will be interested with tempered distributions $u$ in the null-space of
  elliptic operators, elliptic in the sense of the interior symbol $\sigma_m$,
  elliptic regularity implies that they are smooth
in $\bbR^d$, thus have no wave front set there. Moreover,
ellipticity at the corner $\pa(\obbRd\times\obbRd)$, still captured by
$\sigma_m$, implies the
absence of wave front set there as well. The \emph{scattering wave-front set,} $\WF_{\Sc}(u)\subset
T^*_{\Sc}(\obbRd)\restrictedto_{\pa\obbRd},$ thus gives a frequency localized
description of the growth rate of $u$ at $\pa\obbRd.$} \Bk If $\omega\in\pa\obbRd$
and there is a function $\psi\in\cC^{\infty}(\obbRd)$ with $\psi(\omega)=1,$
such that $\psi u\in\cS(\bbR^d),$ then the fiber of $T^*_{\Sc}(\obbRd)$ over
$\omega$ is not in $\WF_{\Sc}(u).$ To localize in frequency, we suppose that
there is an operator $A\in\Psi^{0,0}_{\Sc}(\obbRd)$ such that
$\hN_0(A)(\omega;\tau,\mu)=1$ so that $Au\in\cS(\bbR^d).$ In this case the point
$(\omega;\tau,\mu)\notin\WF_{\Sc}(u).$ In fact we can further refine the wave front set
using the spaces $H_{\Sc}^{m,l}(\bbR^d).$ If 
\begin{equation}
  Au\in H_{\Sc}^{m,l}(\bbR^d)\text{ then }(\omega;\tau,\mu)\notin\WF^{m,l}_{\Sc}(u).
\end{equation}
For our applications, the data, $u,$ is smooth, so the differential order,
$m,$ does not play any role. {\Bl A typical example of such an $A$
  would be the product of a Fourier (i.e.\ momentum) localizer
  $\chi=\chi(\xi)$, $\chi\in\cC^\infty_c(\bbR^d)$ and the position
  localizer $\psi=\psi(w)$, $\psi\in\cC^{\infty}(\obbRd)$ as
  above\footnote{Note that $\phi$ localizes to a compact set of
    momenta (without a compactification), while $\psi$ localizes near
    infinity of $\bbR^d$, corresponding to an asymptotically conic
    position localizer. In traditional microlocal analysis on $\bbR^d$
  the roles of $\phi$ and $\psi$ are reversed.}: $A=\cF^{-1}\chi\cF\psi$ corresponding to
  the amplitude $a(w,\xi)=\psi(w)\chi(\xi)$; in fact it suffices to test $u$ with
  pseudodifferential operators of such form as $\psi$ and $\chi$ are
  allowed to vary (with smaller support). We discuss this briefly
  in the 3-scattering context below.} \Bk

The operator $\Delta-k^2$ is classically elliptic, but its normal symbol,
$\hN_0(\Delta+k^2)=\tau^2+|\mu|^2-k^2,$ vanishes on a subset
\begin{equation}\label{eq:free-char-var}
  {\Bl \Sigma^0_{k^2}} \Bk \subset
  T^*_{\Sc}(\obbRd)\restrictedto_{\pa\obbRd},
\end{equation}
which is its
\emph{scattering characteristic variety}. In~\cite{Melrose94} it is shown that if
$u\in\cS'(\bbR^d),$ satisfies $(\Delta-k^2)u\in\cS(\bbR^d),$ then
\begin{equation}
  \WF_{\Sc}(u)\subset\Sigma^0_{k^2}=\{(\omega;\tau,\mu):\:\tau^2+|\mu|^2=k^2\}.
\end{equation}
The symbol of the Laplace operator defines a {\Bl (rescaled)} \Bk Hamiltonian vector field
$\scH_g$ on $T^*_{\Sc}(\obbRd)\restrictedto_{\pa\obbRd};$ the
scattering wave-front set is proved to be a union of complete
trajectories of this vector field, {\Bl which are called
  bi-characteristics.} \Bk The vector field vanishes at two
points in each fiber {\Bl of
  $T^*_{\Sc}(\obbRd)\restrictedto_{\pa\obbRd}$.} \Bk The unions of these points are submanifolds of
$\Sigma^0_{k^2}$ called the $+$ and $-$ radial sets:
\begin{equation}
  R^{\pm}_{k^2}=\{(\omega;\pm k,0):\omega\in\pa\obbRd\}.
\end{equation}
A complete trajectory of $\scH_g$ consists either of a closed arc in
$T^*_{\Sc}(\obbRd)\restrictedto_{\pa\obbRd}$ joining $(\omega;k,0)$ to
$(-\omega;-k,0)$ or simply a point in $R^{+}_{k^2}$ or $R^{-}_{k^2},$
see~\cite{MelroseZworski}. {\Bl The closed arcs project to (open)
  geodesic segments of length $\pi$ explaining the antipodal relation
  above; this description persists even in the geometric
  (asymptotically conic)
  generalization.} \Bk

In~\cite{Melrose94} Melrose shows that the Sommerfeld radiation
conditions are equivalent to the assumptions that either
\begin{equation}
  \WF^{*,-\frac 12}_{\Sc}(u)\cap R^{+}_{k^2}=\emptyset,\text{ or }
  \WF^{*,-\frac 12}_{\Sc}(u)\cap R^{-}_{k^2}=\emptyset.
\end{equation}
In this case $\WF(u)_{\Sc }\subset R^{-}_{k^2},$ ($\WF(u)_{\Sc }\subset
R^{+}_{k^2}$) and the solution is outgoing (incoming) in the classical
sense, {\Bl and $R^{-}_{k^2}$ is called the outgoing radial set and
$R^{+}_{k^2}$ the incoming one. Here the normalization of ``outgoing'' corresponds to the
  limiting absorption principle resolvent at spectral parameter
  (energy) $k^2+i0$, while ``incoming'' to $k^2-i0$. If
the potential $q$ is a symbol of negative order, or if it decays
sufficiently rapidly,} \Bk
then the same results apply to
$\Delta+q(w)-k^2.$ The operators  $\pa_r\pm
ik\in\Psi^{1,0}_{\Sc}(\obbRd)$ with 
\begin{equation}
  \hN_0(\pa_r\pm ik)(\omega;\tau,\mu)=-i(\tau\mp k);
\end{equation}
these symbols are non-zero except where $\tau=\pm k.$ {\Bl Here we
  take $k>0$, so $k^2+i0=(k+i0)^2$ in the spectral parameter.} The usual
outgoing Sommerfeld radiation condition is that
\begin{equation}
  (\pa_r-ik)u(r\omega)=O\left(r^{-\left(\frac{d-1}{2}+\delta\right)}\right),
\end{equation}
for a $\delta>0.$ {\Bl As $\hN_0(\pa_r- ik)(\tau)\neq 0$ for
  $\tau>0,$} \Bk we
see that this implies that
\begin{equation}
  \WF^{*,\eta-\frac 12}_{\Sc}(u)\cap R^{+}_{k^2}=\emptyset,
\end{equation}
for any $0<\eta<\delta,$ which then implies that $\WF _{\Sc }
(u)\subset R^{-}_{k^2}.$ {\Bl In fact, a more precise oscillatory
  description of $u$ is then possible, namely
  $e^{ikr}$ times a symbol, of order $-\frac{d-1}{2}$, which is
  classical if $q$ is a classical symbol of order $-2$. Note that
  $e^{i(k+i\ep)r}$ is exponentially decaying as $r\to\infty$ for
  $\ep>0$, in accordance with the limiting resolvent description above.}

{\Bl While, to minimize complexity, we do not discuss a geometrically
  generalized version of wave guide networks, we do indicate how our
  definition in
  $\bbR^d$ is placed in this geometric setting since the analytic
  framework of \cite{VasyAsterisque} is guided by the geometry. Thus,
  the wave guide network corresponds to a
finite
set of points $\cC=\{\theta_1,\dots,\theta_m\}\subset \pa X,$, via
considering $\pa X=\pa\bbR^d$ as the sphere at infinity in the radial
compactification picture,} \Bk and a collection of
  geodesic rays $\cR=\{\ell_1,\dots,\ell_m\}$ for the metric $g_0,$ with
  \begin{equation}
    \ell_j\cap\pa X=\theta_j.
  \end{equation}
{\Bl Within the compact space $X$, the support of wave-guide
  potentials $q$ only meets $\pa X$ at the $\theta_j$, and thus the
  $q$ are singular on this compactification. Thus, we need to resolve
  $\theta_j$ in order to assure that $q$ becomes smooth on this new
  space.} \Bk

 The 3-scattering {\Bl space} \Bk introduced in~\cite{VasyAsterisque} is adequate for
 our applications. To analyze perturbations that define wave guides we need to
 blow up the ends of channels, i.e. the points in $\cC.$ We denote this space
 and the blow-down map by
  $$\beta: X_{\WG}\longrightarrow X.$$
  Let
  $$\ff_j=\beta^{-1}(\theta_j)\text{ for }j=1,\dots,m,$$ denote the new front
  faces, i.e.  boundary components introduced by the blow-up, and let
  \begin{equation}
   \mf=\overline{\beta^{-1}(\pa X\setminus\cC)}, 
  \end{equation}
  denote the closure of the lift of the
  $\pa X\setminus \cC.$ In our case, the analysis is greatly simplified by the
  fact that $\cC$ consists of isolated points, and so the ``fibration'' of the
  each front-face, $\ff_j,$ consists of a single fiber isomorphic to a closed
  hemisphere {\Bl (or disk)} \Bk $S^{d-1}_+\simeq \obbR{d-1}$.

  {\Bl For the analysis near $\theta_j$, introduce coordinates
    $(x_1,x')$ on $\bbR^d$ so that in an asymptotically conic region $U_j$ where $x_1$ is sufficiently large, and
    $|x'|/|x_1|$ is sufficiently small, $q$ is independent of $x_1$,
    i.e.\ $q=q(x')$; in these coordinates $\theta_j$ corresponds to
    $(1,0,\ldots,0)$ on $\bbS^{d-1}$. Projective coordinates on $\obbRd$
    near $\theta_j$ are then given by
    \begin{equation*}
      \rho=|x_1|^{-1},\ z_i=x'_i/|x_1|,\ i=2,\ldots,d,
    \end{equation*}
    with $\rho$ the boundary defining one of these coordinates, and
    $\theta_j$ is $\{\rho=0,\ z=0\}$.
    The blow-up then
    introduces projective coordinates around $\rho=0$, $z=0$; thus, in
    any region where $z/\rho$ is bounded, one can take
    \begin{equation}\label{proj_coord}
      \rho=|x_1|^{-1},\ z_{i-1}/\rho=(x'_i/|x_1|)/|x_1|^{-1}=x'_i,\ i=2,\ldots,d.
    \end{equation}
    Note that within $U_j$ thus
the support of $q$ lies in this $z/\rho$ bounded region, i.e.\ in a
cylindrical neighborhood of the ray $\ell_j$ of the form
$[0,1)_{\rho}\times B_r(0)(z)$. Also, $s$ and $\rho$ are locally
equivalent boundary defining functions, i.e.\ they are both smooth
boundary defining functions of $X$, with their ratio a smooth positive
function (thus bounded from above and below by positive constants).
This ratio is $1$ on $\ell_j,$ but near $\ff_j$
it is
$$
\frac{(|x'|^2+x_1^2)^{-1/2}}{|x_1|^{-1}}=(|x'|^2/x_1^2+1)^{-1/2}=
(|z|^2+1)^{-1/2}=(1+\rho^2|x'|^2)^{-1/2},
$$
which is not 1. Correspondingly, {\em at} $\ff_j$, in terms of the
radial compactification coordinates discussed after \eqref{metric0}
via the polar decomposition, with standard spherical local
coordinates $y_i=x'_i/|x|$ on $\bbS^{d-1}_\omega$, we also have
\begin{equation}\label{proj_coord_2}
Y_i=\frac{y_i}{s}=z_i\ \text{on}\ \ff_j.
\end{equation}
In general, the approach of \cite{VasyAsterisque}
only uses the structure at $\ff_j$, where the ratio of these
defining functions is $1$,  but one needs to be careful if one wishes
to talk about a full neighborhood of even the interior of $\ff_j$. In
any case, for the operator in~\eqref{eqn1}, in $U_j$,
\begin{equation}
  H_q=D_{x_1}^2+\Delta_{x'}-k^2(1+q(x')),
\end{equation}
while for a more standard 3-body operator (with the potential
independent of $k$), in $U_j$
\begin{equation}\label{eq:conic-model-op}
  D_{x_1}^2+\Delta_{x'}+q(x')-k^2.
  \end{equation}
  } \Bk

The differential operators in the 3-scattering calculus take the form
in~\eqref{scat_diff}, but with coefficients that belong to
$\cC^{\infty}(\obbRd_{\WG}).$ The main point is to extend the normal symbols to
the front faces. These turn out to be operator-valued symbols, taking values in
$\Psi^{m,0}_{\Sc}(\ff_j),$ which are functions of $\tau,$ the conormal variable.
For $A\in\Psi_{3\Sc}^{m,l}(\obbRd_{\WG}),$ the normal symbol
$\hN_l(A)(\theta_j;\tau)\in\Psi^{m,0}_{\Sc}(\ff_j)$ is elliptic if this operator
has an inverse in $\Psi^{-m,0}_{\Sc}(\ff_j).$ Unlike points on $\mf,$ this is
not a pointwise condition, but rather a global condition on each front-face. We
{\Bl can} \Bk only microlocalize in the $\tau$-direction. {\Bl In the
non-degenerate situations we are interested in $\tau=0$ will not be a
locus of singularities of distributions which allows us to split the relevant part of the extension of the
scattering cotangent bundle over $\ff$ into two half spaces, $\{\tau>0\},
\{\tau<0\},$ corresponding to incoming and outgoing directions, as we
shortly explain.} \Bk

In terms of the {\Bl projective coordinates defined
  in~\eqref{proj_coord}-\eqref{proj_coord_2},} \Bk the normal symbol of
$\Delta+q-k^2$ on $\ff_j$  is the family of \emph{indicial operators}
\begin{equation}
  \hN_0(\Delta+q-k^2)(\ff_j;\tau)=\Delta_Y+q(Y)-k^2+\tau^2,
\end{equation}
where $\Delta_Y$ is the Laplace operator on $\ff_j$ defined by
$h_{ij}(0);$ {\Bl this agrees with \eqref{eq:conic-model-op} in view of the
comments preceding it.} \Bk We
extend the notions of characteristic variety and radial set to the
front faces; {\Bl note that over $\Int\mf$ we simply have
$$
\Sigma_{k^2}\restrictedto_{\Int\mf}=\Sigma^0_{k^2},
$$
with $\Sigma^0_{k^2}$ as in \eqref{eq:free-char-var}, and similarly over $\Int\mf$
bi-characteristics are the integral curves of the (rescaled) Hamiltonian vector
field $\scH_g$ of the Euclidean metric, of which the non-constant ones
(i.e.\ ones not contained in $R^0_{k^2}$ over $\Int \mf$) then project to geodesic
segments on $\Int \mf$, tending either to radial points there or to
one of the front faces $\ff_j$.
Over $\ff_j$,} \Bk the characteristic set is
\begin{equation}
  \Sigma_{k^2}\restrictedto_{\ff_j}=\{\tau:
\hN_0(\Delta+q-k^2)(\ff_j;\tau){\Bl \text{ is not invertible as an element
  of}\ \Psi^{2,0}_{\Sc}(\ff_j)} \Bk\}.
\end{equation}
{\Bl By the lack of ellipticity of $\Delta_Y+q(Y)-k^2+\tau^2$ as
  $|Y|\to\infty$ when $k^2\geq\tau^2$,
  $\hN_0(\Delta+q-k^2)(\ff_j;\tau)$ is not invertible in
  $\Psi^{2,0}_{\Sc}(\ff_j)$ in this case. Although for $k^2<\tau^2$
  the operator
  $\hN_0(\Delta+q-k^2)(\ff_j;\tau)$ is elliptic, thus has a pseudodifferential parametrix
  and is Fredholm (as a map between
  appropriate Sobolev spaces), it is not invertible
  if $k^2-\tau^2$ is a (negative) eigenvalue of $\Delta_Y+q(Y)$.
Thus,} \Bk $\Sigma_{k^2}\restrictedto_{\ff_j}$ consists of $\{\tau:\: |\tau|\leq k\}\cup\{\pm\tau_l:\: l=1,\dots,
N_j\},$ where $\{k^2-\tau_l^2\}$ are the negative eigenvalues (if there are
any) of $\Delta_Y + q(Y).$ {\Bl Note that $\{\tau:\: |\tau|\leq k\}$
  is the projection to $\bbR_\tau$ of the free characteristic set
  $\Sigma^0_{k^2}$ at $\theta_j$.} \Bk Assuming that all  $0<\tau_l,$ the radial sets over
the $\ff_j$ are
\begin{equation}
  R_{k^2}^{\pm}\restrictedto_{\ff_j}=\{\pm k{\Bl \}\cup\{} \Bk\pm\tau_{l}: l=1,\dots, N_j\};
\end{equation}
{\Bl $\{\pm k\}$ is the projection of the free radial set $R_{k^2}$ to
  $\bbR_\tau$ at $\theta_j$.} \Bk

{\Bl As described in \cite{VasyAsterisque}, maximally extended broken bi-characteristics\footnote{In the general 3-body
  setting more complicated phenomena can occur since the $\{\theta_j\}$ are
replaced by submanifolds along which propagation is possible; these
are called generalized broken bi-characteristics and are discussed more
in Section~\ref{prop_est}.} (which are not constant, i.e.\ not in
  the radial sets) are then concatenations of maximal bi-characteristic
  segments over $\Int \mf$ with the property that the projection of
  the concatenation to $\pa\obbRd$ extends to be continuous at the
  $\theta_j$ and such that for two consecutive bi-characteristic
  segments $\gamma_{i}$, $\gamma_{i+1}$, defined on $(t_{i-1},t_{i})$,
  resp.\ $(t_ i, t_{i+1})$ ending to $\theta_j$ in the forward, resp.\ backward
  direction, the $\tau$ components $\tau\circ\gamma_i$,
  $\tau\circ\gamma_{i+1}$ extend to be continuous: $\lim_{t\to
    t_i-}\tau(\gamma_i(t))=\lim_{t\to
    t_i+}\tau(\gamma_{i+1}(t))$. This corresponds to the law of
  reflection for the wave equation at boundaries or edges: tangent
  components of the momenta, which here are tangent to the $\ell_j$,
  i.e.\ just $\tau$, are conserved, while the normal component can
  change upon hitting a boundary or edge (but the total momentum needs
  to stay in the characteristic variety). Note that along the first
  segment $\gamma_0$ in the backward direction and along the last
  segment $\gamma_N$ in the forward direction $\tau$ can tend to $k$,
  resp.\ $-k$, i.e.\ the incoming, resp.\ outgoing radial sets, both
  over $\Int\mf$ and over $\cup_j\ff_j$. It is convenient to extend
  the concatenation to
  a curve $\gamma$ defined at the concatenation points $t_i$, where (with $\pi$ the bundle projection) 
  its value is $(\theta_j=\lim_{t\to t_i-}\pi(\gamma_i), \lim_{t\to
    t_i-}\tau(\gamma_i(t)))$, which agrees with the analogous limit
  from above.\footnote{It is best to define $\gamma$ directly in a
    compressed phase space as in \cite{VasyAsterisque}, but the in
    our case equivalent
    presentation we adopted here is simpler.}} \Bk

Suppose that $A\in\Psi_{3\Sc}^{0,0}(\obbRd),$ with its indicial operator, $
\hN_0(A)(\ff_j;\tau_0),$ on $\ff_j$ at $\tau_0$ invertible. If
\begin{equation}
  Au\in H_{\Sc}^{m,l}(\bbR^d), \text{ then }(\ff_j,\tau_0)\notin \WF^{m,l}_{3\Sc}(u).
\end{equation}
A formal solution, {\Bl namely} \Bk $u,$ with $(\Delta+q(x)-k^2)u\in\cS(\bbR^d)$, is outgoing if
\begin{equation}\label{eqn37.101-pre}
  \WF_{3\Sc}(u)\subset R_{k^2}^{-},
\end{equation}
and incoming if its 3-scattering wave-front set is contained in
$R_{k^2}^+.$ In fact,
to show that the solution is outgoing it suffices to show that
\begin{equation}\label{eqn37.101}
  \WF^{m,l}_{3\Sc}(u)\cap R^+_{k^2}=\emptyset,
\end{equation}
for an $l>-\frac 12,$ with an analogous result for incoming solutions.
{\Bl Indeed, this follows from the propagation of singularities
  \cite{Melrose94,VasyAsterisque}: first using that this limited
  regularity gives $\WF_{3\Sc}(u)\cap R^+_{k^2}=\emptyset$, and then
  using that any other non-radial point in the characteristic variety
  (outside which we have elliptic estimates) is on a broken
  bi-characteristic which tends to $R^+_{k^2}$ in the backward
  direction.}

If a formal solution is known to be outgoing or incoming over $\Int\mf,$ which
we assume in the sequel, then this considerably constrains the possibilities for
$\WF^{m,l}_{3\Sc}(u)\restrictedto_{\ff}.$ As noted above, the propagation
properties for $\WF^{m,l}_{3\Sc}(u)$ under the Hamiltonian vector field $\scH_g$
are well understood, see~\cite{Melrose94,VasyAsterisque}.  From this
analysis  it follows that 
\begin{equation}\label{eqn32.100}
  \WF_{3\Sc}(u)\restrictedto_{\Int\mf}\subset
R^{-}_{k^2}\text{ or }R^+_{k^2},
\end{equation}
implies that
\begin{equation}
  \WF_{3\Sc}(u)\restrictedto_{\ff}\subset [R^+_{k^2}\cup R^-_{k^2}]\restrictedto_{\ff}.
\end{equation}
{\Bl  Indeed, any other point in the characteristic set over $\ff$
 necessarily has $\tau^2<k^2$, hence there are bi-characteristic
 segments tending to it in the forward/backward direction, at least
 one of each kind must carry $\WF_{3\Sc}(u)$: see Proposition 14.1
 in~\cite{VasyAsterisque} and Section~\ref{prop_est} here for what is
 needed in the more general context of Maxwell's equations studied
 later in this paper.} \Bk  This means that to show that
 the solution is incoming, or outgoing, with this assumption, only requires a
 proof that $ \WF_{3\Sc}(u)\restrictedto_{\ff}$ is a subset of either
 $\{\tau>0\}$ or $\{\tau<0\}.$

 A simple explicit condition can be given to demonstrate such a claim. This
 condition is essentially that given\footnote{ {\Bl We remark
 that the version presented in \cite{VasyAsterisque} takes advantage
 of writing $u=\psi(H)u+(\Id-\psi(H))u$, with the second term Schwartz, for $u$ a formal solution where $\psi\equiv
 1$ near $k^2$ and compactly supported; incorporating
 $\psi(H)\in\Psi_{3\Sc}^{-\infty,0}(\obbRd)$ into  $A_{\omega, \pm}$
 makes the composite an element of
 $\Psi_{3\Sc}^{-\infty,0}(\obbRd)$; in general  $A_{\omega, \pm}$ is {\em not}
 pseudodifferential on its own as it is not symbolic in $\xi$. One adds to
 $A_{\omega,\pm}\phi(H)$ a suitable multiple of $\Id-\phi(H)$ to
 obtain a pseudodifferential operator with invertible indicial
 operator at the point in question. See also \cite[Section~5]{Vasy:Propagation-Many}.}} \Bk  in~\cite{Isozaki94}. To state this
 condition, we first define conical cut-offs. Let $\omega\in S^{d-1}$ and choose
 $R\gg 1$ and $\delta\ll 1.$ Let $\varphi_0\in\cC^{\infty}_c((-\delta,\delta)),$
 with $\varphi_0(t)=1$ for $2|t|<\delta,$ and $\varphi_1\in\cC^\infty(\bbR)$
 with $\varphi_1(t)=0,$ for $t<1,$ and $\varphi_1(t)=1$ for $t>2.$ We define the
 conic cut-off near to $\omega\in\obbRd$ setting
 \begin{equation}
   \Psi_{\omega,\delta,R}(x)=\varphi_1\left(\frac{|x|}{R}\right)\varphi_0\left(\left|\frac{x}{|x|}-\omega\right|\right)
 \end{equation}

 Now choose $0<\epsilon\ll k$ and $\chi\in\cC^{\infty}(\bbR),$ so that $\chi(t)=0$ for
 $t<-2\epsilon,$ and $\chi(t)=1$ for $t>-\epsilon.$
 For $u\in\cS'(\bbR^d)$ a
 formal solution to $(\Delta+q(x)-k^2)u\in\cS(\bbR^d)$ we define
 \begin{equation}\label{eqn41.333}
   A_{\omega, \pm}
   u(x)=\Psi_{\omega,\delta',R'}(x)\int_{\bbR^d}e^{ix\cdot
     \xi}\chi\left(\pm \frac{\xi\cdot\omega}{(1+|x|^2)^{\frac
       12}}\right)\cF(\Psi_{\omega,\delta,R}u)(\xi)d\xi.
 \end{equation}
 The operator depends on the choices made in its definition, in
 particular the positive constants, $\delta,\delta',R,R',\epsilon,$
 but we suppress this dependence in the notation.  Note that while
 $\cF(\Psi_{\omega,\delta,R}u)(\xi)$ is generally only a tempered
 distribution, it is rapidly decreasing, which makes $A_{\omega, \pm}
 u$ well defined, see~\cite[Remark 7]{EpMaWG3}. Assuming that
 $\WF_{3\Sc}(u)\restrictedto_{\Int\mf}\subset R^{-}_{k^2},$ the formal
 solution $u$ is outgoing along $\ff_j$ if there exist
 $0<\delta'<\delta,$ and $R'>R>0$ so that
\begin{equation}
  (1+r^2)^{\frac{l}{2}}A_{\theta_j, +} u\in  L^2(\bbR^d)\text{ for an }l>-\frac 12.
\end{equation}
This condition implies that
$\WF_{3\Sc}(u)\restrictedto_{\ff_j}\subset\{\tau<0\}.$ A formal
solution is outgoing if these conditions hold for all
$\theta_j\in\cC,$ along with
$\WF_{3\Sc}(u)\restrictedto_{\Int\mf}\subset R^{-}_{k^2}.$ As a
practical matter, one needs only to verify that~\eqref{eqn37.101}
holds for an $l>-\frac 12.$ For reasonably explicit solutions, these
conditions can be checked by simple computations. Several examples, in
the setting of $2d$-wave-guide networks, are given in~\cite{EpMaWG3}.
There is an analogous definition for an incoming solution, with $+$
replaced by $-.$

We take note of the following simple consequence of these results, and
the analysis in~\cite{VasyAsterisque}.
\begin{proposition}\label{prop1}
  Suppose that $u\in\cS'(\bbR^d)$ is a solution to $(\Delta-k^2(1+q))u=0$ in
  $B_R^c(0).$ If $u\in L^2(B_R^c(0))$ then $u$ vanishes in this set. 
\end{proposition}
\begin{proof}
  As $u\in L^2(B_R^c(0)),$ the condition in~\eqref{eqn37.101} is
  satisfied for both $+$ and $-,$ which then implies {\Bl by the
    propagation of singularities as noted above} \Bk that
  $\WF_{3\Sc}(u)=\emptyset.$ Therefore $u$ is Schwartz in a
  neighborhood of infinity. The proof in~\cite{VasyAsterisque} that
  such a solution vanishes faster than any exponential is local at
  infinity. Moreover, the Carleman estimate {\Bl
    in~\cite{VasyAsterisque}} \Bk showing that a
  super-exponentially vanishing solution is actually zero in a
  neighborhood of infinity is also local at infinity. Hence $u$
  vanishes in an open set. The classical unique continuation theorems
  for second order elliptic operators then imply that $u$ vanishes in
  $B_R^c(0).$
\end{proof}

\section{Acoustic wave-guides}\label{acous_wg}

 In this section we recall the basic hypotheses
 from~\cite{VasyExpDecay}, and consider  operators
 $H:\cS(X;E)\to\cS(X;E),$ where $E\to X$ is a smooth vector bundle.
  
  The conditions we require on $H$ are:
  \begin{enumerate}
    \item For $j=1,\dots,m,$ the difference $g-g_0$ along $\ff_j$ is a
      section of $T^*\ff_j\otimes T^*\ff_j.$
  \item $H$ is self adjoint operator w.r.t $d\nu_{g}.$
    \item
      $H=\Delta_{g}\otimes\Id_E+V,$ where $\Delta_g$ is the metric
      Laplace operator,
           and $V$ is a first order, 3-scattering operator, which is
      formally self adjoint w.r.t. $d\nu_{g}$ and vanishes to
      infinite order on $\mf.$
      \item The indicial operators, $\hN_0(V)(\tau),$
      do not depend on $\tau.$
  \end{enumerate}

A basic example of an operator covered by these assumptions is a
model in acoustics where the sound speed and density, $c(x)$ and
$\rho(x),$ vary. So far as we know, radiation conditions for this type
of operator do not appear in the literature. This physical situation
is  modeled by the operator
$$\tH_{c,\rho}u=c(x)^2\rho(x)\nabla\cdot \rho(x)^{-1}\nabla u.$$ For
an acoustic wave-guide it is assumed that $d\rho$ and $dc$ vanish
outside a compact set, union with cylindrical neighborhoods of the
rays, $\cR.$ Moreover we assume that $\rho(x), c(x)$ are independent
of the channel radial variables outside a compact set.  A time
harmonic acoustic wave has the form $U(x,t)=e^{-i\omega t}u(x),$ which
therefore satisfies the equation
\begin{equation}
  \tH_{c,\rho}u+\omega^2u=0
\end{equation}

The papers~\cite{DebevierePravica1, DebevierePravica2} contain proofs
of the limiting absorption principle and the absence of point spectrum
embedded in the continuous spectrum, which lies in $(-\infty,0],$ for
  $\tH_{c,\rho}$ in the bi-infinite and perturbations of the
  bi-infinite case. Our interest is in a \emph{network} of
  wave-guides, which are modeled as a collection of semi-infinite
  wave-guides outside a compact set. As in the model case,
  $H=\Delta-k^2(1+q(x),$ considered in Section~\ref{wg_geom}, the
  radial sets for the operators $\tH_{c,\rho}$ are determined by the
  wave-guide modes for the bi-infinite wave-guides defined in the
  channel ends.

Suppose that $x=(x_1,x'),$ with $c, \rho$ independent of $x_1$ and
positive constants outside a compact set. A wave-guide mode for such
an operator takes the form $u(x_1,x')=e^{i\tau x_1}v(x')$ where $v$ is
an $L^2$-solution to
\begin{equation}
  \rho\nabla_{x'}\rho^{-1}\cdot\nabla_{x'} v+\frac{\omega^2}{c^2(x')}v=\tau^2v.
\end{equation}
If we let $v=\sqrt{\rho}w,$ then this equation becomes
\begin{equation}
  \sqrt{\rho}\,\nabla_{x'}\rho^{-1}\cdot\nabla_{x'}\sqrt{\rho} w+\frac{\omega^2}{c^2(x')}w=\tau^2w,
\end{equation}
which is a perturbation of $\Delta_{x'}+\omega^2c_0^{-2},$ self
adjoint w.r.t. $dx'.$ Here $c_0$ is the positive constant value assumed
by $c(x')$ outside a compact set. We can rewrite this as
\begin{equation}
   \sqrt{\rho}\,\nabla_{x'}\rho^{-1}\cdot\nabla_{x'}\sqrt{\rho}
   w+q(x')w=\left(\tau^2-\frac{\omega^2}{c_0^2}\right)w,
\end{equation}
with $q=\omega^2(c^{-2}(x')-c_0^{-2}),$ a compactly supported
potential. If $q$ is non-positive, then the spectrum of the operator
on the left is $(-\infty,0]$ and we can show that there are no
  $L^2$-eigenvalues embedded in the continuous spectrum, and therefore
  no wave-guide modes. If $q(x')>0$ for $x'$ in a set of positive
  measure, then the operator may have a finite number of positive
  eigenvalues, $\{(E^2_j,w_j):\: j=1,\dots, N\}$ such that $w_j\in
  L^2(\bbR^{d-1})$ and
\begin{equation}
   \sqrt{\rho}\,\nabla_{x'}\rho^{-1}\cdot\nabla_{x'}\sqrt{\rho}
   w_j+q(x')w_j=E_j^2w_j.
\end{equation}
In this case the frequencies
\begin{equation}
  \tau^{\pm}_j=\pm\sqrt{E_j^2+\frac{\omega^2}{c_0^2}},
\end{equation}
contribute to the radial set of this operator along the front face.

Returning to the general case of an acoustic wave-guide network, the equation we
would like to analyze is
\begin{equation}
  \tH_{c,\rho}u+\omega^2u=0.
\end{equation}
With $u=\sqrt{\rho}w,$ we rewrite this as
\begin{equation}
  H_{\rho,q}w+k_0^2w=0,
\end{equation}
where $k_0^2=\frac{\omega^2}{c^2_0},$ and 
\begin{equation}
  H_{\rho,q}w= \sqrt{\rho}\,\nabla\cdot \rho(x)^{-1}\nabla \sqrt{\rho}w+q(x)w,
\end{equation}
with
\begin{equation}
  q(x)=\omega^2\left(\frac{1}{c^2(x)}-\frac{1}{c^2_0}\right).
\end{equation}
The operator $H_{\rho,q}$ is self adjoint with respect to $dx.$

This equation fits into the framework outlined above, with
reference Laplacian $-\Delta_{g_0}.$ The operators $H_{\rho,q}$
and $-\Delta_{g_0}$ have the same second order part and therefore
\begin{equation}
  H_{\rho,q}-\Delta_{g_0}=Q_{\rho}
\end{equation}
is a first order operator that is formally self adjoint with respect to
$dx.$ It is not difficult to see that
$Q_{\rho}\in\Diff^1_{3\Sc}(\obbRd_{\WG}),$  does not involve derivatives
normal to the front faces of $\obbRd_{\WG},$ hence $\hN_0(Q_{\rho})(\tau)$
does not depend on $\tau$ near to the front faces in $\obbRd_{\WG}.$ 

On the $\mf$ the perturbation terms vanish and $\rho(x)$ is constant. Hence,
the characteristic variety over $\mf$ is as before,
\begin{equation}
  \Sigma_{k_0^2}\restrictedto_{\Int\mf}=\{(\omega;\tau,\mu):\:\omega\in\Int\mf\text{
    and }\tau^2+|\mu|^2=k_0^2\};
\end{equation}
with the radial sets given by
$R^{\pm}_{k_0^2}\restrictedto_{\Int\mf}=\{(\omega;\pm k_0,0)\}.$ We let
$(x_1,x')$ denote coordinates in which $\theta_j=(1,0,\dots,0).$ Let
\begin{equation}
 \frac{1}{x_1},\quad Y= x'
\end{equation}
be projective coordinates for $\ff_j.$ In these coordinates, the indicial operator
over $\ff_j$ is
\begin{equation}
  \hN_0(H_{\rho,q})(\ff_j;\tau)=\sqrt{\rho(Y)}\nabla_{Y}\cdot\rho^{-1}(Y)\nabla_Y\sqrt{\rho(Y)}
  +q(Y)+k_0^2-\tau^2.
\end{equation}
The continuous spectrum of the operator
$\sqrt{\rho(Y)}\nabla_{Y}\cdot\rho^{-1}(Y)\nabla_Y\sqrt{\rho(Y)}+q(Y)$ is $(-\infty,0].$ As
  noted above, there may be a finite number of positive
  eigenvalues
  $\{E^2_{j,l}:l=1,\dots,N_j\}.$ From this we conclude
  that the characteristic variety over $\ff_j$ is
  \begin{equation}
    \Sigma_{k_0^2}\restrictedto_{\ff_j}=\{|\tau|\leq k_0\}\cup\left\{\pm
    \sqrt{k_0^2+E_{j,l}^2}:\:l=1,\dots, N_j\right\}.
  \end{equation}
  As before the radial set over $\ff_j$ is
  \begin{equation}
     R^{\pm}_{k_0^2}\restrictedto_{\ff_j}=\{\tau=\pm k_0\}\cup\left\{\pm
    \sqrt{k_0^2+E_{j,l}^2}:\:l=1,\dots, N_j\right\}.
  \end{equation}
  
  In this wave-guide setting, the results in~\cite{VasyExpDecay} apply
  prove the following the following radiation condition.
  \begin{theorem}\label{thm2}
    If the metric $g,$  and the operator $H_{\rho,q}$ satisfy
    the hypotheses above, then, for $k_0\in (0,\infty),$ a solution $u$ to
    $(H_{\rho,q}+k_0^2)u=0,$ for which $\WF^{*,l}_{3\Sc}(u)\cap R^-_{k^2_0}=\emptyset,$ (or
    $\WF^{*,l}_{3\Sc}(u)\cap R^+_{k^2_0}=\emptyset$), for an $l>
    -\frac 12,$ has
    $\WF_{3\Sc}(u)\subset
    R^+_{k^2_0}$ (resp. $\WF_{3\Sc}(u)\subset
    R^-_{k^2_0}$) and therefore  $u\equiv 0.$
  \end{theorem}
  \begin{remark}
    The outgoing condition on $\mf$ can be checked by evaluating the
    asymptotic decay rate of $(\pa_r-ik_0)u(r\omega);$ it is checked
    near the front faces by applying operators like those
    in~\eqref{eqn41.333}. In~\cite{VasyAsterisque,VasyExpDecay} the
    Mourre estimate, at any negative eigenvalue, is proved for
    operators like $H_{\rho,q}$. The absence of point spectrum
    embedded in the continuum is also established. Using standard
    arguments, this gives a proof of the limiting absorption principle,
    along with the fact that the limiting absorption solution to
    \begin{equation}
      (H_{\rho,q}+k_0^2)u=f\in (1+r^2)^{-\frac{l}{2}}L^2(\bbR^d),\, l>\frac 12,
    \end{equation}
    is outgoing.
  \end{remark}
  \begin{proof}
    {\Bl See the paragraph of
      \eqref{eqn37.101-pre}-\eqref{eqn37.101} for the proof that $\WF_{3\Sc}^{m,l}(u)\subset
    R^+_{k^2_0}$ (resp. $\WF ^{m,l}_{3\Sc}(u)\subset
    R^-_{k^2_0}$) for $l<-1/2$. By a result of Isozaki
  \cite[Lemma~4.5]{Isozaki94}, as presented in
  \cite[Proposition~17.8]{VasyAsterisque}, which strongly uses
    that $u$ satisfies the homogeneous equation,  for such a
  $u,$ $\WF_{3\Sc}^{m,l}(u)=\emptyset$ for  $l\in(-1/2,0).$ Hence
  the radial point estimate of \cite{VasyAsterisque} strengthens
  this to all $l$. At this point Proposition~\ref{prop1} is applicable and shows
  that $u$ vanishes near infinity. The classical unique continuation theorems
  for second order elliptic operators then imply that $u$ vanishes.}
  \end{proof}

  As a corollary of the proof in~\cite{VasyAsterisque,VasyExpDecay} of
  the uniqueness theorem we also have:
\begin{corollary}
    If $u\in L^2(B_r^c(0))$ is a solution to
    $H_{\rho,q}u+k_0^2u=0,$ for an $\omega>0,$ then $u\equiv 0.$
    Hence such an $L^2$-solution to $\tH_{c,\rho}u+\omega^2u=0$ is also
    identically zero.
\end{corollary}
The proof of the corollary is essentially identical to the proof of Proposition~\ref{prop1}.
This result in the bi-infinite case follows from the analysis 
  in~\cite{DebevierePravica1, DebevierePravica2}.

\section{Time Harmonic Maxwell Equation}\label{s.maxwell}
A considerably more interesting case is that of electromagnetic
wave-guide networks, which are modeled by Maxwell's
equations on $\bbR^3.$ In the
time harmonic case, with time dependence $e^{-i\omega t},$ the
Maxwell equations take the form
\begin{equation}\label{eqn7}
  d\xi=i\omega\mu(x)\eta,\quad d^*\eta=-i\omega\epsilon(x)\xi,
\end{equation}
where $\xi$ is a 1-form and $\eta$ a 2-form. Here $\epsilon(x)$ is the
scalar permittivity and $\mu(x)$ the scalar permeability of the medium
at this frequency.  An immediate consequence of~\eqref{eqn7} are the
divergence equations
\begin{equation}\label{eqn13}
  d[\mu(x)\eta]=d^*[\epsilon(x)\xi]=0.
\end{equation}
Outside of a compact set, $\epsilon(x), \mu(x)$ are only non-constant
in the wave-guide channels, which are cylindrical neighborhoods of
finitely many rays, $\cR=\{\ell_1,\dots,\ell_m\}.$ Sufficiently far out
we assume that they are independent of the distance along the ray.
That is, near to $\ell_j,$ outside of large ball,
\begin{equation}
    \epsilon(x)=\epsilon(\pi_j(x)),\quad\mu(x)=\mu(\pi_j(x)).
\end{equation}
We sometime denote these pairs of functions as $\{(\epsilon_j,\mu_j)\}.$

The radiation conditions for such problems are essentially the same as
in the previous cases, however establishing uniqueness is considerably
more complicated. As we see below, the indicial operators that arise
in this setting are not of ``product type,'' and so the proof of the
Mourre estimate requires a additional assumption, which is guaranteed
hold under certain conditions on the
$\kappa_j(x)=\epsilon_j(x)\mu_j(x),$ see Proposition~\ref{prop4}. Even
with these hypotheses, we are not able to rule out the possibility
that, for a discrete set, $\{\omega_j\},$ equation~\eqref{eqn7} has
non-trivial Schwartz class, solutions.  While $L^2$-eigenfunctions at
positive energies do not arise for bi-infinite wave-guides, it remains
a tantalizing question whether Maxwell wave-guide networks can exhibit
behavior so utterly different from what arises in the scalar case.

  The uniqueness results, as stated in~\cite{VasyExpDecay}, do not
  apply directly to first order systems, but a solution to this
  equation also solves a second order elliptic PDEs. Indeed, there are
  many possible choices. We now give a construction for a second order
  equation where the form-degrees separate. As a first step we rescale setting
\begin{equation}\label{eqn14.40}
  \xi=\sqrt{\mu}\txi\text{ and }\eta=\sqrt{\epsilon}\teta,
\end{equation}
and let $\kappa(x)=\sqrt{\epsilon(x)\mu(x)}.$  In the sequel we assume that
$\kappa(x)=\kappa_0>0$ in the complement of the support of
$d\kappa.$ In the electromagnetic setting,
$1/\kappa(x)$ is the light speed. If the complement of the support of $d\kappa$
is the ``physical'' vacuum, then special relativity requires $\kappa(x)\geq \kappa_0.$

With the new unknowns we rewrite the equations
as
\begin{equation}\label{eqn4}
D^{\epsilon,\mu} \left(\begin{matrix}\txi\\\teta\end{matrix}\right)= \left(\begin{matrix} 0 &\frac{i}{\sqrt{\epsilon}}d^*\sqrt{\epsilon}\\
    \frac{-i}{\sqrt{\mu}}d\sqrt{\mu}&0\end{matrix}\right)
    \left(\begin{matrix}\txi\\\teta\end{matrix}\right)=
      \omega\kappa\left(\begin{matrix}\txi\\\teta\end{matrix}\right).
\end{equation}
Throughout this paper $d^*$ denotes the adjoint of $d$ w.r.t. to the
Euclidean metric. We define the fiber-metric on forms to be the
Euclidean metric, but let $d\nu_{w}(x)=w(x)dx$ denote a weighted Lebesgue
measure; the $L^2$-norms on 1- and 2-forms are
\begin{equation}\label{eqn17}
  \|\txi\|^2+\|\teta\|^2=\int_{\bbR^3}[|\txi(x)|^2+|\teta(x)|^2]d\nu_{\kappa}(x),
\end{equation}
where $|\cdot|^2$ are the Euclidean fiber norms. The operator
$D^{\epsilon,\mu}$ is a formally self adjoint operator with respect to this
inner product.


In the rescaled unknowns  the divergence equations become:
\begin{equation}
  d^*(\sqrt{\epsilon}\kappa\txi)=0=d(\sqrt{\mu}\kappa\teta).
\end{equation}
To include the divergence terms in the operator we define a map from 1-forms to 0-forms
\begin{equation}
\txi\mapsto  \frac{i}{\kappa\sqrt{\epsilon}}d^*(\sqrt{\epsilon}\kappa\txi),
\end{equation}
with adjoint w.r.t. $d\nu_{\kappa}$
\begin{equation}
  \varphi\mapsto -i\sqrt{\epsilon}d\left(\frac{\varphi}{\sqrt{\epsilon}}\right),
\end{equation}
and a map from 2-forms to 3-forms
\begin{equation}
  \teta\mapsto -\frac{i}{\sqrt{\mu}\kappa}d(\sqrt{\mu}\kappa\teta),
\end{equation}
with adjoint w.r.t. $d\nu_{\kappa}$
$$\psi dx\mapsto i\sqrt{\mu}d^*\left(\frac{\psi dx}{\sqrt{\mu}}\right).$$
Putting together the various pieces we define a first order operator acting on sections
of the full exterior algebra,
$E=\Lambda^0\oplus\Lambda^1\oplus\Lambda^2\oplus\Lambda^3,$ by setting
\begin{equation}\label{eqn71.336}
  \fL^{\epsilon,\mu}=\frac{1}{\kappa}\left[\begin{matrix}
      0&\frac{i}{\kappa\sqrt{\epsilon}}d^*\kappa\sqrt{\epsilon}&0&0\\
    -i\sqrt{\epsilon}d\frac{1}{\sqrt{\epsilon}}&0&\frac{i}{\sqrt{\epsilon}}d^*\sqrt{\epsilon}&0\\
      0&-\frac{i}{\sqrt{\mu}}d\sqrt{\mu}&0&i\sqrt{\mu}d^*\frac{1}{\sqrt{\mu}}\\
      0&0&-\frac{i}{\kappa\sqrt{\mu}}d\kappa\sqrt{\mu}&0
      \end{matrix}\right].
\end{equation}
This operator is formally self adjoint on sections,
$\Phi\in\cC^{\infty}_c(\bbR^3;E)$ w.r.t. the Hermitian inner product
\begin{equation}
  \int_{\bbR^3}|\Phi|^2d\nu_{\kappa^2},
\end{equation}
where $|\cdot|^2$ is the Euclidean fiber norm. We can rewrite  Maxwell's
equations in the form
\begin{equation}\label{eqn73.335}
  \fL^{\epsilon,\mu}\left[  \begin{matrix}
      0\\\txi\\\teta\\0  \end{matrix}\right]=\omega\left[  \begin{matrix}
      0\\\txi\\\teta\\0  \end{matrix}\right],
\end{equation}
which implies that
\begin{equation}\label{eqn24.30}
 \kappa^2 \fL^{\epsilon,\mu}  \fL^{\epsilon,\mu}\left[  \begin{matrix}
      0\\\txi\\\teta\\0  \end{matrix}\right]=\omega^2\kappa^2\left[  \begin{matrix}
      0\\\txi\\\teta\\0  \end{matrix}\right].
\end{equation}

The entries of $\fL^{\epsilon,\mu}$  define both a complex and,
by duality, a co-complex: let $i\fd$ denote the lower diagonal of
$\fL^{\epsilon,\mu}$ and $-i\fd^*$ its adjoint
(w.r.t. $d\nu_{\kappa^2}$), the upper diagonal, then
\begin{equation}\label{fd_cmplx}
  \begin{split}
    \fd_j:\cC^{\infty}(\bbR^3;\Lambda^{j})\longrightarrow\cC^{\infty}(\bbR^3;\Lambda^{j+1})
    &\text{ and }
    \fd_j^*:\cC^{\infty}(\bbR^3;\Lambda^{j})\longrightarrow\cC^{\infty}(\bbR^3;\Lambda^{j-1})\\
    \fd_{j+1}\fd_j=&0=\fd_{j-1}^*\fd_j^*.
    \end{split}
\end{equation}
 So acting on $j$-forms:
$$(i\fd^*-i\fd)_j^2=(\fd+\fd^*)^2_j=\fd_{j-1}\fd^*_{j}+\fd_{j+1}^*\fd_j.$$
 Multiplying by $\kappa^{-2},$ we can therefore rewrite~\eqref{eqn24.30} as
\begin{equation}\label{eqn24.45}
  (\fd+\fd^*)^2(\txi+\teta)=\omega^2(\txi+\teta).
\end{equation}
With these choices the form-degrees in the second order operator $(\fd+\fd^*)^2$ are uncoupled.

We let $\cD_j^{\epsilon,\mu}$ denote the degree $j$ part of $(\fd+\fd^*)^2.$ A calculation shows that
\begin{equation}
  \begin{split}
    \cD_1^{\epsilon,\mu}&=\frac{1}{\sqrt{\epsilon}\kappa}d^*\frac{1}{\mu}d\sqrt{\mu}+\frac{\sqrt{\epsilon}}{\kappa}
    d\frac{1}{\epsilon\kappa^2}d^*\sqrt{\epsilon}\kappa\\
     \cD_2^{\epsilon,\mu}&=\frac{1}{\sqrt{\mu}\kappa}d\frac{1}{\epsilon}d^*\sqrt{\epsilon}+\frac{\sqrt{\mu}}{\kappa}
    d^*\frac{1}{\mu\kappa^2}d\sqrt{\mu}\kappa.
  \end{split}
\end{equation}
If $\star$ denotes the Euclidean Hodge star-operator, then a simple calculation shows that
\begin{equation}\label{eqn78.333}
  \cD_2^{\epsilon,\mu}=\star \cD_1^{\mu,\epsilon}\star.
\end{equation}

The operator $H^{\epsilon,\mu}= \kappa^2
\fL^{\epsilon,\mu}\fL^{\epsilon,\mu}$ is self adjoint with respect
to the standard Riemannian measure,
\begin{equation}\label{eqn27.20}
  \begin{split}
    \int_{\bbR^3}\langle  H^{\epsilon,\mu}\Phi,\Psi\rangle dx&=
    \int_{\bbR^3}\langle  \fL^{\epsilon,\mu}\Phi, \fL^{\epsilon,\mu}\Psi\rangle \kappa^2 dx\\
       &=\int_{\bbR^3}\langle \Phi,  H^{\epsilon,\mu}\Psi\rangle dx.
    \end{split}
\end{equation}
 Note that the first line in~\eqref{eqn27.20} shows that $ \langle
 H^{\epsilon,\mu}\Phi,\Phi\rangle\geq 0$. The second order part is
 just $F_{E}=\Delta_{g_0}\otimes \Id_E,$ where $g_0$ is the flat
 metric, which is obviously self adjoint w.r.t. $dx.$ The difference
\begin{equation}
 Q^{\epsilon,\mu}= \kappa^2 \fL^{\epsilon,\mu}\fL^{\epsilon,\mu}-F_{E}
\end{equation}
is a 3-scattering, first order operator, which is evidently also self
adjoint. We define the operator
\begin{equation}\label{eqn25}
  H^{\epsilon,\mu}=
  F_{E}+Q^{\epsilon,\mu}.
\end{equation}

Recalling that $H^{\epsilon,\mu}$ preserves form degree, a pair
$(\txi,\teta)$ satisfying~\eqref{eqn4} also satisfies
\begin{equation}\label{eqn27.2}
 \kappa^2 \left( \begin{matrix}
    \cD_1^{\epsilon,\mu}&0\\0&\cD_2^{\epsilon,\mu}
  \end{matrix}\right)
  \left[ \begin{matrix} \txi\\\teta  \end{matrix}\right]=
  \omega^2\kappa^2\left[ \begin{matrix}  \txi\\\teta \end{matrix}\right].
\end{equation}
Let $\kappa_0$ denote the constant value of $\kappa(x)$ away from the
channels, and rewrite this equation as an eigenvalue equation for the
operator $\cH^{\epsilon,\mu}_{\omega},$ which is the ``Hamiltonian'' in this
context:
\begin{equation}\label{eqn27.1}
 \cH^{\epsilon,\mu}_{\omega}\left[ \begin{matrix}
     \txi\\\teta \end{matrix}\right]\overset{d}{=} \kappa^2
 \left( \begin{matrix} \cD_1^{\epsilon,\mu}&0\\0&\cD_2^{\epsilon,\mu}
  \end{matrix}\right)
  \left[ \begin{matrix} \txi\\\teta  \end{matrix}\right]
  +\omega^2(\kappa_0^2-\kappa^2) \left[ \begin{matrix} \txi\\\teta  \end{matrix}\right]=
  \omega^2\kappa_0^2\left[ \begin{matrix}  \txi\\\teta \end{matrix}\right].
\end{equation}
The $\omega^2(\kappa_0^2-\kappa^2(x))$-term is now a potential with
support near infinity within the channels, and independent of the
`radial' variables $\{\langle x,\theta_j\rangle\}.$ In many physical applications this potential
is non-positive.
\begin{remark}
  In~\eqref{eqn27.2} we give a generalized eigenvalue problem implied
  by Maxwell's equation, with an operator that is independent of
  $\omega,$ whereas in~\eqref{eqn27.1}, we formulate it as a standard
  eigenvalue problem, but in terms of an operator that also depends on
  $\omega.$ We generally use the latter formulation as it fits more
  closely with earlier work.
\end{remark}

Focusing on the 1-form part of $H^{\epsilon,\mu},$ we compute the
explicit form of the lower order terms of $\kappa^2
\cD_1^{\epsilon,\mu}.$ If we let
$$\ellk(x)=\log\kappa(x),$$
then a tedious calculation using the identities in Appendix~\ref{app1}
shows that the operator $\kappa^2\cD_1^{\epsilon,\mu}$ is given by
\begin{multline}\label{eqn32}
  \kappa^2 \cD_1^{\epsilon,\mu}\txi=\Delta_1\txi-
   2[ d\ellk  d^*\txi+d(d\ellk\lrcorner\txi)]+\\ \left[d\log(\sqrt{\epsilon}\kappa)\right]
  d\log(\sqrt{\epsilon}\kappa)\lrcorner\txi+\left[
  d\log\sqrt{\mu}\right]d\log\sqrt{\mu}\lrcorner\txi+\frac{1}{2}\left[\frac{d^*d\mu}
  {\mu}+\frac{3}{2}\frac{|d\mu|^2}{\mu^2}\right]\txi.
\end{multline}
If $v_{\alpha}$ is the vector field dual to the 1-form, $\alpha,$
w.r.t.  the Euclidean metric on $\bbR^3,$ then
$\alpha\lrcorner\beta\overset{d}{=}i_{v_{\alpha}}\beta.$ See
Appendix~\ref{app1}.

The first order perturbation terms are
$$Q_1^{\epsilon,\mu}\txi=- 2[ d\ellk  d^*\txi+d(d\ellk\lrcorner\txi)],$$
and the zeroth order terms are
\begin{multline}
  Q_0^{\epsilon,\mu}\txi=\\
  \left[d\log(\sqrt{\epsilon}\kappa)\right]
d\log(\sqrt{\epsilon}\kappa)\lrcorner\txi+[d\log\sqrt{\mu}]d\log\sqrt{\mu}\lrcorner\txi
+\frac{1}{2}\left[\frac{d^*d\mu}{\mu}+\frac{3}{2}\frac{|d\mu|^2}{\mu^2}\right]\txi.
\end{multline}
These define self adjoint
operators in $\Diff^1_{3\Sc}(\obbRd_{\WG})$ and
$\Diff^0_{3\Sc}(\obbRd_{\WG})$ respectively. The coefficients of these
operators are supported in $\supp(|d\epsilon|+|d\mu|),$ and therefore
vanish to infinite order at the free face, $\mf,$ and also in a
neighborhood of $\cup_j\mf\cap\ff_j.$

\section{Radiation Conditions for Maxwell's Equations in Wave-guides}\label{Max_rad_cond}

We begin by recalling that a solution to the free-space Maxwell equations,
\begin{equation}\label{eqn89.300}
  d\xi=i\omega\mu_0\eta,\quad d^*\eta=-i\omega\epsilon_0\xi,
\end{equation}
is also a solution to
\begin{equation}
  \Delta\xi=k_0^2\xi,\quad\Delta\eta=k_0^2\eta,
\end{equation}
where $k_0=\omega\sqrt{\epsilon_0\mu_0}.$ As the Laplacian on forms
acts as the free-space Laplacian on the components, its
scattering wave front set lies in the scattering characteristic variety of
$\Delta_0-k_0^2,$ where
\begin{equation}
  \Sigma_{\Delta_0-k_0^2}=\{(\omega;\tau,\mu):\tau^2+|\mu|^2=k_0^2\}.
\end{equation}
The radial sets are the subsets of $\Sigma_{\Delta_0-k_0^2}$ given by
\begin{equation}
 R^{0,pm}_{k_0^2}(k_0)=\{(\omega;\pm k_0, 0)\}.
\end{equation}

A free-space solution to the Maxwell equation, $(\xi,\eta),$ satisfies
the outgoing radiation condition if
\begin{equation}\label{eqn93.300}
  \WF^{*,-\frac 12}_{\Sc}(\xi,\eta)\cap R^{0,+}_{k_0^2}=\emptyset.
\end{equation}
This means that there
is an  operator $A\in\Psi^{0,0}_{\Sc}(\bbR^3)$ with normal operator
$\hN_0(A)$ non-vanishing along $R^{0,+}_{k_0^2}$ such that
\begin{equation}
  A\left(\begin{matrix}\xi\\ \eta\end{matrix}\right)\in H^{*,l}_{\Sc}(\bbR^3)
\end{equation}
for an $l>-\frac 12.$ From the analysis in~\cite{Melrose94} the
condition in~\eqref{eqn93.300} implies that
\begin{equation}
  \WF_{\Sc}(\xi,\eta)\subset R^{0,-}_{k_0^2},
\end{equation}
and the coefficients of $(\xi,\eta)$ satisfy the classical Sommerfeld radiation
condition: $(\pa_r-ik_0)u(r\omega)=O(|r|^{-2}).$

It is well known that the components of a 1-form solution,
$\xi,$ to the Helmholtz equation satisfy this radiation condition if
and only if $\xi$ satisfies the radiation condition for
the vector Helmholtz equation:
\begin{equation}\label{VHRadCond}
  i_{\hx}d\xi-d^*\xi\, \frac{x\cdot dx}{|x|}-i\omega\kappa_0\xi=O(|x|^{-2}),
\end{equation}
where $\hx=x\cdot\pa_x/|x|.$ The 2-form solution is outgoing if and
only if $\star \eta$ satisfies this condition. If $(\xi,\eta)$
satisfy~\eqref{eqn89.300}, then $d^*\xi=0$ and this condition can be
rewritten as
\begin{equation}\label{MaxRadCond}
  \sqrt{\mu_0} i_{\hx}\eta-\sqrt{\epsilon_0}\xi=O(|x|^{-2}),
\end{equation}
see~\cite{ColtonKress}.

In order to obtain radiation conditions for the
Maxwell equation with wave-guides~\eqref{eqn4}, we have introduced a
second order elliptic system,~\eqref{eqn27.2}, such that any solution
of the former is a solution to the latter. This allows us to use a
variant of the standard microlocal approach to radiation conditions
from~\cite{Melrose94,VasyAsterisque} to obtain radiation conditions in
the present case. Solutions $(\txi,\teta),$ to~\eqref{eqn4}, are also
``divergence-free'' $\fd_1^*\txi=\fd_2\teta=0.$ In this paper we give
radiation conditions for the second order systems without analyzing
the refinements like those in~\eqref{VHRadCond} and~\eqref{MaxRadCond} implied by
these additional first order equations.

The microlocal approach to radiation conditions requires that we
identify the characteristic varieties and radial sets, in the
3-scattering sense, for the operators,
$\kappa^2\cD_i^{\epsilon,\mu}+\omega^2(\kappa_0^2-\kappa^2),$ $i=1,2,$
that make up $\cH^{\epsilon,\mu}_{\omega}.$ Along the free face, $\mf,$ of
$\pa\obbRd_{\WG},$ these agree with the characteristic variety and
radial sets for $\Delta_0-k_0^2,$
\begin{equation}
  \Sigma_{\cH^{\epsilon,\mu}_{\omega},\mf}=\{(\theta;\tau,\mu):
  \theta\in\mf,\,\tau^2+|\mu|^2=k_0^2\},
\end{equation}
where $k_0=\omega\kappa_0,$ and
\begin{equation}
  R^{\pm}_{k^2_0,\mf}=\{(\theta;\tau,\mu):\theta\in\mf,\,\tau=\pm k_0,\,\mu=0\}.
\end{equation}

\subsection{The Indicial Operators of the Hamiltonian Over $\ff$}\label{nrm.symb}
 Fixing a channel, $\ell_j,$ there are coordinates so that
 $\ell_j=\{(x_1,0,0):\: x_1>r_j\}.$ The coefficients of
 $\cH^{\epsilon,\mu}_{\omega}$ in this channel are then
 $\epsilon_j(x_2,x_3),\,\mu_j(x_2,x_3);$ we set
 $\kappa_j(x_2,x_3)=\sqrt{\epsilon_j\mu_j(x_2,x_3)}.$ Near $\ff_j,$
 but away from $\ff_j\cap\mf,$ we can use $(1/x_1,x_2,x_3)$ as
 coordinates for $\overline{\bbR^3}_{\WG}.$ To define the
 characteristic variety and radial sets on the front faces we need to
 calculate the indicial operators of $\cH^{\epsilon,\mu}_{\omega}$ over $\ff_j,\,
 j=1,\dots,m.$ For each front face, this reduces to the analysis of
 the bi-infinite wave-guide  operator defined by  $\cH^{\epsilon_j,\mu_j}_{\omega}.$

We note that a Maxwell wave-guide mode for the $j$th channel is a
solution of
\begin{equation}
  \begin{split}
  e^{-i\tau x_1}d[e^{i\tau x_1}\xi(x_2,x_3)]&=i\omega\mu_j(x_2,x_3)\eta(x_2,x_3)\\
  e^{-i\tau x_1}d^*[e^{i\tau x_1}\eta(x_2,x_3)]&=-i\omega\epsilon_j(x_2,x_3)\xi(x_2,x_3),
  \end{split}
\end{equation}
with $(\xi,\eta)\in L^2(\bbR^2;\Lambda^1\oplus\Lambda^2).$  After rescaling,
with
\begin{equation}
  \cxi=\sqrt{\mu_j(x_2,x_3)}\,\xi,\quad \ceta=\sqrt{\epsilon_j(x_2,x_3)}\,\eta,
\end{equation}
such a solution
also satisfies
\begin{equation}
\cH^{\epsilon_j,\mu_j}_{\omega}\left[e^{i\tau x_1}
    \left(
  \begin{matrix}
    \cxi(x_2,x_2)\\\ceta(x_2,x_3)
  \end{matrix}\right)\right]=\omega^2\kappa_0^2e^{i\tau x_1}\left(
  \begin{matrix}
    \cxi(x_2,x_2)\\\ceta(x_2,x_3)
  \end{matrix}\right).
\end{equation}

 As the operators $\cH^{\epsilon_j,\mu_j}_{\omega}$ preserve form degree, we
 can focus on the 1-form part.  The analysis of the 2-form part is
 essentially identical. In the $(x_1,x_2,x_3)$-coordinates, let the
 coefficients of $\cxi=\cxi_1dx_1+\cxi_2dx_2+\cxi_3dx_3,$ be
 independent of $x_1,$ then the indicial operators of $ \kappa^2_j
 \cD_1^{\epsilon_j,\mu_j}$ in this channel are the operators acting on
 $\cS(\bbR^2;\Lambda^1)$ defined by
  \begin{equation}
    \hN_{\ff_j,0}(\kappa^2_j\cD_1^{\epsilon_j,\mu_j})(\tau)\cxi=e^{-i\tau
      x_1}\kappa^2_j\cD_1^{\epsilon_j,\mu_j}[e^{i\tau x_1}\cxi].
  \end{equation}
We can rewrite this as
\begin{equation}
  \hN_{\ff_j,0}(\kappa^2_j\cD_1^{\epsilon_j,\mu_j})(\tau)\cxi=
  \kappa_j^2[\fd_{0}(\tau)\fd_1^*(\tau)+\fd_2^*(\tau)\fd_1(\tau)]\cxi,
\end{equation}
where
\begin{equation}
  \fd_j(\tau)=e^{-i\tau x_1}\fd_j(\tau)e^{i\tau x_1},\quad
  \fd^*_j(\tau)=e^{-i\tau x_1}\fd^*_j(\tau)e^{i\tau x_1}
\end{equation}
are the indicial operators for the first order operators defining the
complex in~\eqref{fd_cmplx}. If $\fd_j\theta=\alpha
d[\beta\theta],$ then the indicial operator takes the form
\begin{equation}
  \fd_j(\tau)=\alpha d_{\tau}\beta,\text{ where }d_{\tau}\cxi=e^{-i\tau x_1}d[e^{i\tau x_1}\cxi],
\end{equation}
with a similar formula for the adjoints.

At a wave-guide frequency, $\tau,$ let
\begin{equation}
u_1= \omega^{-2}e^{-i\tau x_1}\fd_0\fd^*_1[ e^{i\tau x_1}\cxi]\text{ and }u_2=
\omega^{-2}e^{-i\tau x_1}\fd_2^*\fd_1[ e^{i\tau x_1}\cxi];
\end{equation}
clearly $\cxi=u_1+u_2.$ Using the fact that
$\fd^2(\tau)=[\fd^*(\tau)]^2=0,$ we see that $\fd_1(\tau)u_1=0$
(\emph{curl-free}), and $\fd_1^*(\tau)u_2,$ (\emph{divergence-free})
and therefore
\begin{equation}\label{eqn111.601}
  \fd_0(\tau)\fd^*_1(\tau) u_1=\omega^2u_1\text{ and }\fd_2^*(\tau)\fd_1(\tau) u_2=\omega^2u_2.
\end{equation}
The forms $u_1, u_2$ are orthogonal on $\bbR^2$ with respect to the
measure $\kappa_j^2dx'.$ Only divergence-free modes arise in the study
of Maxwell's equations.
The case of divergence-free solutions has been carefully analyzed
in~\cite{JolyPoirier}, under weak regularity assumptions on the
coefficients $(\epsilon,\mu).$

  The indicial operator of $\Delta_1$ is
  $[-(\pa_{x_2}^2+\pa_{x_3}^2)+\tau^2],$ and
  the first order terms at $\ff_j$ are given by
  \begin{equation}\label{eqn34.40}
    \begin{split}
  &\hN_{\ff_j,0}(Q_1^{\epsilon_j,\mu_j})(\tau)\left(
  \begin{matrix}
    \cxi_1\\\cxi_2\\\cxi_3
  \end{matrix}
  \right)=2\left(
  \begin{matrix}
    0&-i\tau\ellk_2 &-i\tau\ellk_3\\
    i\tau\ellk_2&-\pa_{x_2}[\ellk_2]&\ellk_2\pa_{x_3}-\pa_{x_2}\ellk_3\\
    i\tau\ellk_3&\ellk_3\pa_{x_2}-\pa_{x_3}\ellk_2&-\pa_{x_3}[\ellk_3]
  \end{matrix}
  \right)\left(
  \begin{matrix}
    \cxi_1\\\cxi_2\\\cxi_3
  \end{matrix}
  \right)=\\
  &2\left[\tau\left(
  \begin{matrix}
    0&-i\ellk_2 &-i\ellk_3\\
    i\ellk_2&0&0\\
    i\ellk_3&0&0
  \end{matrix}
  \right)+\left(
  \begin{matrix}
    0&0 &0\\
    0&-\pa_{x_2}[\ellk_2]&\ellk_2\pa_{x_3}-\pa_{x_2}\ellk_3\\
   0&\ellk_3\pa_{x_2}-\pa_{x_3}\ellk_2&-\pa_{x_3}[\ellk_3]
  \end{matrix}
  \right)\right]\left(
  \begin{matrix}
    \cxi_1\\\cxi_2\\\cxi_3
  \end{matrix}
  \right),
  \end{split}
\end{equation}
where $d\ellk=\ellk_2dx_2+\ellk_3dx_3.$ This is easily seen to be a self adjoint
operator.  If we write $\cxi=\cxi_{1}dx_1+\cuxi,$ then we easily show
that the linear term in $\hN_{\ff_j,0}(Q_1^{\epsilon_j,\mu_j})(\tau)$ can
be expressed as
\begin{equation}\label{lin-tau-term}
  2i\tau(\cxi_1 d\ellk-d\ellk\lrcorner \cuxi\, dx_1),
\end{equation}
which will prove useful.

If we set
\begin{equation}
  \begin{split}
  \alpha=\alpha_2dx_2+\alpha_3dx_3=d\log\sqrt{\epsilon_j}\kappa_j,&\quad
  \beta=\beta_2dx_2+\beta_3dx_3=d\log\sqrt{\mu_j}
  \text{  and }\\
  &m=\frac{1}{2}\left[\frac{d^*d\mu_j}{\mu_j}+\frac{3}{2}\frac{|d\mu_j|^2}{\mu^2_j}\right],
  \end{split}
\end{equation}
then the zeroth order term is
\begin{equation}\label{eqn36.40}
  \hN_{\ff_j,0}(Q_0^{\epsilon_j,\mu_j})(\tau)\left(
  \begin{matrix}
    \cxi_1\\\cxi_2\\\cxi_3
  \end{matrix}
  \right)=\left(
  \begin{matrix} m&0&0\\
    0& \alpha_2^2+\beta_2^2+m&\alpha_2\alpha_3+\beta_2\beta_3\\
    0&\alpha_2\alpha_3+\beta_2\beta_3&\alpha_3^2+\beta_3^2+m
\end{matrix}\right)
    \end{equation}
As noted in~\eqref{eqn78.333}, the degree 2 operator is $\kappa^2\cD_2^{\epsilon_j,\mu_j}=\star
\kappa^2\cD_1^{\mu,\epsilon}\star,$ so its analysis is essentially
identical to the degree 1 case.

To simplify the notation we write
$\tQ_i^{\epsilon_j,\mu_j}(\tau)=\hN_{\ff_j,0}(Q_i^{\epsilon_j,\mu_j})(\tau),$
for $i=0,1,$ so that
\begin{multline}
  \hN_{\ff_j,0}(\kappa^2\cD^{\epsilon_j,\mu_j}_1)(\tau)
    \cxi(x_2,x_2)=\\
  \left[\Delta^{(2)}\otimes\Id_{\Lambda^1\bbR^3}+\tQ_1^{\epsilon_j,\mu_j}(\tau)+\tQ_0^{\epsilon_j,\mu_j}
    +\tau^2\right]
    \cxi(x_2,x_2).
\end{multline}
Here $\Delta^{(2)}$ is the flat, 2-dimensional Laplacian. This operator is self adjoint
with respect to $dx_2dx_3.$ A wave-guide mode in the $j$th channel  is then
defined by an $L^2(\bbR^2)$-solution to
\begin{multline}\label{eqn48.61}
  L_{\tau,j}\cxi(x_2,x_2)\overset{d}{=}
  \left[\Delta^{(2)}\otimes\Id_{\Lambda^1\bbR^3}+\tQ_1^{\epsilon_j,\mu_j}(\tau)+
    \tQ_0^{\epsilon_j,\mu_j}+\omega^2(\kappa_0^2-\kappa_j^2)\right]
  \cxi(x_2,x_2)\\=(\omega^2\kappa_0^2-\tau^2) \cxi(x_2,x_2).
\end{multline}
 Note that the operator $\overline{L_{\tau,j}}=L_{-\tau,j},$ which
 implies that the set of wave-guide frequencies is invariant under
 $\tau\mapsto -\tau.$\footnote{We would like to thank Manas Rachh for
 suggesting this should be true.} These operators are analytic
 families of relatively compact perturbations of
 $\Delta^{(2)}\otimes\Id_{\Lambda^1\bbR^3}.$ They might have a finite
 number of non-positive $L^2$-eigenvalues, which are parameterized
 by analytic functions. This is analyzed in  the following section.

 As noted earlier, a similar, though more detailed analysis, of wave-guide modes is done
 in~\cite{JolyPoirier}. In this paper a slightly different family of
 operators is considered. Instead of studying standard eigenvalue
 problems for $L_{\tau,j},$ they study the generalized eigenvalue
 problems for the operators
 \begin{equation}
  \left[\Delta^{(2)}\otimes\Id_{\Lambda^1\bbR^3}+\tQ_1^{\epsilon_j,\mu_j}(\tau)+
    \tQ_0^{\epsilon_j,\mu_j}+\tau^2\right]
  \cxi(x_2,x_2)=\hat{\lambda}_l(\tau)\kappa_j^2 \cxi(x_2,x_2). 
 \end{equation}
 A wave-guide mode at frequency $\omega$ occurs if
 $\hat{\lambda}_l(\tau)=\omega^2.$ Of course both methods find the
 same wave-guide modes. Our approach seems better adapted to proving
 uniqueness theorems for wave-guide networks, whereas the Joly-Poirier
 approach is more incisive for studying the dependence of wave-guide
 frequencies on $\tau,$ and their dispersion relations. A very useful
 consequence of their analysis (Section 4) is that divergence-free
 modes can only exist for
\begin{equation}\label{jp_ineqal1}
  \tau<\omega\kappa_j^{\max}\text{ where }\kappa_j^{\max}=\max\{\kappa_j(x'):\:x'\in\bbR^2\}.
\end{equation}
Below we show that this is also true for curl-free modes.

\begin{remark}
  The first order term
  $$\hN_{\ff_j,0}(Q_1^{\epsilon_j,\mu_j})(\tau)=\tQ^{\epsilon_j,\mu_j}_{11}+
  \tau\tQ^{\epsilon_j,\mu_j}_{10},$$
  where $\tQ^{\epsilon_j,\mu_j}_{11}$ is a
  first order operator and $\tQ^{\epsilon_j,\mu_j}_{10}$ is zeroth order.
  Because of the term $\tau\tQ^{\epsilon_j,\mu_j}_{10},$ these operators do
  not have the product structure required in earlier versions of
  Vasy's theorem. In particular the proof of the Mourre estimate needs
  to be modified.  In Section~\ref{ss.mourre} we introduce a Mourre
  estimate for $N$-body Schr\"odinger operators that allows terms like
  $\tau\tQ^{\epsilon_j,\mu_j}_{10}$ in the indicial operators at the front
  faces. A simplified proof for the wave-guide case is given in
  Appendix~\ref{MourreEstMaxWG}. A proof of the general result,
  following the argument in~\cite{VasyAsterisque,VasyExpDecay}, is
  given in Appendix~\ref{MourreEstGenCase}.  The Mourre estimate is
  the essential step needed to establish radiation conditions for
  Maxwell wave-guides.
\end{remark}

\subsection{Wave-guide Modes}
As we show below, the $L^2$-eigenvalues of $L_{\tau,j}$ are
non-positive. From standard perturbation theory it follows that the
negative eigenvalues are parameterized by analytic functions
$\{\lambda_l(\tau)\}.$ The eigenvalues that define energies of
wave-guide modes at wave number $\omega^2\kappa_0^2$ occur when
$\lambda_l(\tau)=\omega^2\kappa_0^2-\tau^2.$ This, in turn, implies
that $|\tau|$ lies in a fixed interval $[\omega\kappa_0,\tau_{\max}].$
As one expects on physical grounds, and is in fact true,
$\tau_{\max}\leq \omega\kappa_j^{\max},$ see~\eqref{jp_ineqal1}
and~\eqref{jp_ineqal_cf}. As shown above, every wave guide mode for
the bi-infinite wave-guide operators $\cD^{\epsilon_j,\mu_j}_1,$
$\cD^{\epsilon_j,\mu_j}_2$ is the orthogonal sum of a divergence-free
mode, and a curl-free mode.
\begin{enumerate}
  \item[{\bf Curl-free modes}]
At a wave-guide frequency, $\tau,$ let $e^{i\tau x_1}\cxi_{\tau}$ be a
closed, 1-form wave-guide mode, and set $\psi=\fd_1^*(\tau)\cxi_{\tau}.$ This
function satisfies the scalar equation
$\fd_1^*(\tau)\fd_0(\tau)\psi=\omega^2\psi,$ which is equivalent to
\begin{equation}\label{eqn112.405}
  \frac{1}{\sqrt{\epsilon_j}}
  d^*_2\epsilon_j\left[d_2\frac{\psi}{\sqrt{\epsilon_j}}\right]+\omega^2(\kappa_0^2-\kappa_j^2)\psi=
    (\omega^2\kappa_0^2-\tau^2)\psi.
\end{equation}
Here and in the sequel $d_2$ indicates the exterior differential in
2-dimensions, and $d_2^*$ its adjoint w.r.t. $dx'.$ The operator on
the left is self adjoint with respect to $dx'=dx_2dx_3.$ The equation
\begin{equation}\label{eqn112.505}
  \frac{1}{\sqrt{\epsilon_j}}
  d^*_2\epsilon_j\left[d_2\frac{\psi}{\sqrt{\epsilon_j}}\right]+
  \omega^2(\kappa_0^2-\kappa_j^2)\psi=
  \lambda\psi
\end{equation}
can only have $L^2$-solutions for $\lambda\leq 0;$ hence the solutions to
\begin{equation}
  \lambda=\kappa_0^2\omega^2-\tau^2
\end{equation}
are real with $|\tau|\geq \omega\kappa_0.$ Setting
\begin{equation}
  \cxi_{\tau}=\fd_0(\tau)\psi,
\end{equation}
gives an $\fd(\tau)$-closed 1-form that satisfies
\begin{equation}
  \fd_0(\tau)\fd_1^*(\tau)\cxi_{\tau}=\omega^2\cxi_{\tau},
\end{equation}
and therefore defines a closed 1-form wave-guide mode.

The equation satisfied by the divergence of closed, 1-form modes
only depends on $\tau$ through the $\tau^2$-factor on the r.h.s. If
$\{\lambda_{jl}^{\cf},\psi_{jl}^{\cf}\}$ are the non-positive eigenvalues,
and corresponding eigenfunctions of the operator on the left
in~\eqref{eqn112.405}, then the $\tau$s that arise are given by
  \begin{equation}
    \tau^{\cf}_{jl}=\pm\sqrt{\omega^2\kappa_0^2-\lambda_{jl}^{\cf}},\,
    l=1,\dots,N_j^{\cf}.
  \end{equation}
The eigenfunctions, $\{\psi^{\cf}_{jl}\}$ are orthogonal w.r.t. the
measure $dx'.$  Integrating by parts we easily show that
\begin{equation}\label{jp_ineqal_cf}
  \tau^2\int|\psi|^2dx'\leq \omega^2\int\kappa_j^2|\psi|^2dx',
\end{equation}
which implies that~\eqref{jp_ineqal1} also holds for curl-free modes. 


\item[{\bf Divergence-free modes}]
  Setting $\cxi=\cxi_1dx_1+\cuxi;$ the form $e^{i\tau
  x_1}\cxi(x_2,x_3)$ is divergence free if
\begin{equation}\label{eqn113.400}
  i\tau\kappa_j\sqrt{\epsilon_j}\cxi_1=d_2^*[\kappa_j\sqrt{\epsilon_j}\cuxi].
\end{equation}
Using this relation, we can show that $\cuxi$-part of a
divergence-free wave-guide mode satisfies an eigenvalue equation that
is independent of $\tau;$ we get the $2\times 2$ system,
\begin{equation}\label{eqn115.600}
\cM\cuxi=  \sqrt{\mu_j}d_2^*\frac{1}{\mu_j}d_2[\sqrt{\mu_j}\cuxi]+
  \frac{1}{\sqrt{\mu_j}}d_2\frac{1}{\epsilon_j}d_2^*[\kappa_j\sqrt{\epsilon_j}\cuxi]=
  (\omega^2\kappa_j^2-\tau^2)\cuxi.
\end{equation}
The first term on the l.h.s. of~\eqref{eqn115.600} is self adjoint
w.r.t. $dx',$ the second term w.r.t. $\kappa_j^2dx'.$ As yet we have
not found a self adjoint equation, independent of $\tau,$ satisfied by
some part of a divergence-free wave-guide mode.

We rewrite equation~\eqref{eqn115.600} as
\begin{equation}\label{eqn114.400}
  \cM\cuxi+\omega^2(\kappa_0^2-\kappa_j^2)\cuxi=\lambda\cuxi.
\end{equation}
For $\lambda\notin [0,\infty),$ let
$R_{\lambda}=(\Delta^{(2)}-\lambda)^{-1};$ we can
rewrite~\eqref{eqn114.400} as
\begin{equation}\label{eqn115.400}
  (\Id+R_{\lambda}P)\left(
  \begin{matrix}
    \cxi_2\\\cxi_3
  \end{matrix}
\right)=0,
\end{equation}
where $P$ is a first order operator with compactly supported
coefficients. It is immediate from standard elliptic estimates that
$R_{\lambda}P$ is a family of compact operators on
$L^2(\bbR^2;\Lambda^1),$ analytic in the complement of $[0,\infty),$
  which tends to zero in norm as $|\lambda|\to \infty$ in the
  complement of any proper cone containing $[0,\infty).$ Using analytic
  Fredholm theory it follows that there is at most a finite set of
  eigenvalues, $\{\lambda^{\df}_{jl}:\: l=1,\dots,N_j\},$ in the
  complement of such a cone, for which~\eqref{eqn115.400} has a non-trivial
  $L^2$-solution. This shows that the $\tau$s giving rise to
  divergence-free wave-guide modes are contained in the set
\begin{equation}\label{eqn118.601}
  \{\pm\sqrt{\omega^2\kappa_0^2-\lambda^{\df}_{jl}}:\: l=1,\dots,N^{\df}_j\}.
\end{equation}

Note that if $\cuxi$ is a solution to~\eqref{eqn114.400}, such that
\begin{equation}\label{eqn119.0.601}
  \cxi=\cxi_1dx_1+\cuxi,\text{ with
  }\cxi_1=\frac{d_2^*[\kappa_j\sqrt{\epsilon_j}\cuxi]}{i\tau\kappa_j\sqrt{\epsilon_j}},
\end{equation}
is a divergence-free wave-guide mode at frequency $\tau,$ then the
$dx_1$ component of~\eqref{eqn111.601} reads
\begin{equation}\label{eqn119.601}
  \sqrt{\mu_j}d_2^*\frac{1}{\mu_j}d_2[\sqrt{\mu_j}\cxi_1]-\omega^2\kappa_j^2\cxi_1
  =i\tau\sqrt{\mu_j}
  d_2^*\left(\frac{\cuxi}{\sqrt{\mu_j}}\right).
\end{equation}
Moreover, if this relation hold, then $\cxi$ defined
in~\eqref{eqn119.0.601} gives a divergence-free wave-guide mode.

Using~\eqref{eqn113.400} we see that
\begin{equation}
  d_2^*\left(\frac{\cuxi}{\sqrt{\mu_j}}\right)=i\frac{\tau\cxi_1}{\sqrt{\mu_j}}+
  \frac{d\kappa_j^2}{\kappa_j^2}\lrcorner\left(\frac{\cuxi}{\sqrt{\mu_j}}\right).
\end{equation}
Substituting this into~\eqref{eqn119.601} leads to the relation
\begin{equation}
 \sqrt{\mu_j}d_2^*\frac{1}{\mu_j}d_2[\sqrt{\mu_j}\cxi_1]-\omega^2\kappa_j^2\cxi_1
  =-\tau^2\cxi_1+\frac{i\tau d\kappa_j^2}{\kappa_j^2}\lrcorner\cuxi.  
\end{equation}
Finally, replacing $\cxi_1$ with $\frac{d_2^*
  w}{i\tau\kappa_j\sqrt{\epsilon_j}},$ where
$w=\kappa_j\sqrt{\epsilon_j}\cuxi,$ this equation can be rewritten
\begin{equation}\label{eqn123.601}
  d_2^*\left[\frac{1}{\mu_j}d_2\frac{1}{\epsilon_j}d_2^*w+\left(\frac{\tau^2}{\kappa_j^2}
   - \omega^2\right)w\right]=0,
\end{equation}
which is a consequence of~\eqref{eqn115.600}.

Starting from a solution $\cuxi$
to~\eqref{eqn114.400}, we choose $\tau$ so that
\begin{equation}
  \lambda=\omega^2\kappa_0^2-\tau^2.
\end{equation}
With either choice of sign for $\tau,$ equation~\eqref{eqn123.601} is
a consequence of~\eqref{eqn115.600}. Using~\eqref{eqn113.400} to
define $\cxi_1,$ we reverse the steps to see that~\eqref{eqn119.601}
holds, which shows that $\cxi,$ defined by~\eqref{eqn119.0.601}, is a
wave-guide mode.  It is divergence-free by construction. We conclude
that any $L^2$-solution to~\eqref{eqn114.400} defines a generalized
divergence free wave-guide.  A priori it is possible the $\Im\tau\neq
0;$ in fact such modes are constructed in~\cite{LaiJiang2018}.  If
$\tau$ is real, then this provides an $L^2$-eigenfunction
for $\hN_{\ff_j,0}(\kappa^2\cD^{\epsilon_j,\mu_j})(\tau)-\omega^2\kappa_j^2.$
We let $\{\lambda_{jl}^{\df}\}\subset (-\infty,0]$ denote the subset
  of non-positive real energies that arise in this way and $\{\tau^{\df}_{jl}\}$
  the positive real solutions to
  \begin{equation}
  \lambda^{\df}_{jl}=\omega^2\kappa_0^2-[\tau^{\df}_{jl}]^2.
\end{equation}
These all correspond to wave-guide modes, and every divergence-free wave-guide mode appears
in this list.
\end{enumerate}


This proves the following proposition.
\begin{proposition}\label{propNew2}
  For each $j$ there is a finite number of non-positive values
  $\{\lambda_{jl}:\: l=1,\dots, N_j\},$ which arise as the
  non-positive point spectrum of the operators in~\eqref{eqn112.405}
  and~\eqref{eqn114.400}. The indicial operator
  $\hN_{\ff_j,0}(\kappa^2\cD^{\epsilon_j,\mu_j})(\tau)-\omega^2\kappa_j^2$
  is invertible for $\omega\kappa_0<|\tau|<\omega\kappa_j^{\max},$
  except for
  \begin{equation}
    \pm\tau^1_{jl}=\pm\sqrt{\omega^2\kappa_0^2-\lambda_{jl}},\, l=1,\dots, N_j.
  \end{equation}
  It is invertible for $\tau\geq \omega\kappa^{\max}_j.$
\end{proposition}

\begin{remark}
  In the wave-guide literature the parameter values
  $\{\tau_{jk}^{1}\}$ for which non-trivial $L^2$-solutions
  to~\eqref{eqn48.61} exist are called \emph{propagation constants},
  as $\omega/\tau^{1}_{jk}$ is the phase velocity down the
  wave-guide of the corresponding mode. For consistency with the
  preceding discussion we continue to call them wave-guide
  frequencies. We let
  \begin{equation}
    \cT_1=\bigcup_{j=1}^m\{\pm \tau^{1}_{jk}:\:k=1,\dots,N_j^{1}\}
  \end{equation}
  denote the totality of wave-guide frequencies for 1-forms.
\end{remark}

Proposition~\ref{propNew2} suffices to describe the radial sets for
the operators $\cD^{\epsilon_j,\mu_j}_1$ and
$\cD^{\epsilon_j,\mu_j}_2$. To prove the Mourre estimate for
$\kappa^2\cD_1^{\epsilon,\mu}-\lambda$ for other positive energies
$\lambda,$ we need the analogous result for $L_{\tau,j}+\tau^2-\lambda.$ A
circumstance that requires special consideration arises if the
operator $L_{\tau,j}$ has a branch of negative eigenvalues
$\lambda_0(\tau)$ defined for $\tau$ in
$(\sqrt{\lambda},\sqrt{\lambda}+\delta),$ or
$(\sqrt{\lambda}-\delta,\sqrt{\lambda}),$ for a $\delta>0,$ such that
  \begin{equation}\label{eqn106.224}
    \lim_{\tau\to [\sqrt{\lambda}]^{\pm}}\lambda_0(\tau)=0.
  \end{equation}
   In this case, 0 may or may not be in the point spectrum of
   $L_{\sqrt{\lambda},j}.$ If it is not, but~\eqref{eqn106.224} holds,
   then we say that the family $L_{\tau,j}$ \emph{has a resonance at
   $\sqrt{\lambda}$}.  We define having a resonance at
   $-\sqrt{\lambda}$ in the same way.

   \begin{proposition}\label{prop2.224}
     Assume that $0\notin\sigma_{\pp}(L_{\pm\sqrt{\lambda},j}).$ For
     each $j$ there is an integer $N^{1}_j(\lambda),$ so that there
     are $N^{1}_j(\lambda)$ positive values
     $\{\tau^1_{jk}(\lambda):\:k=1\cdots,N^1_j(\lambda)\}$ and
      eigenpairs
     $\{(\pm\tau^{1}_{jk}(\lambda),\cxi^{\pm}_{jk}(\lambda)):\:k=1,\dots,
     N^{1\,\pm}_j(\lambda)\},$ such that
  \begin{equation}\label{eqn49.60}
     \left[
     L_{\tau^{1}_{jk}(\lambda),j}+(\tau^{1}_{jk}(\lambda))^2\right]
   \cxi^{\pm}_{jk}(\lambda)=\lambda \cxi^{\pm}_{jk}(\lambda),
  \end{equation}
  and these forms span the $L^2$-solutions
  to~\eqref{eqn49.60}.  The values of
  $\tau$ for which such a solution exists have $|\tau|$ in an interval
  $(\sqrt{\lambda},\tau_{\max}].$
   \end{proposition}
   \begin{remark}
     The notations $N^1_j, \tau^1_{jk},\cxi^{\pm}_{jk},$ etc. refer to
     the special case $\lambda=\omega^2\kappa_0^2,$  which is covered
     in Proposition~\ref{propNew2}.
   \end{remark}
\begin{remark}

  In the proof we show that
  \begin{equation}\label{eqn108.225}
    N_{\max}(\lambda)=\max\{N^{1}_j(\lambda):\: j=1,\dots,m\},
  \end{equation}
  the maximum number of $L^2$-solutions to~\eqref{eqn49.60} for any channel, is
  finite. This follows from the fact that the number of non-positive
  eigenvalues of the operators $L_{\tau,j},$ for
  $|\tau|\in[\sqrt{\lambda},\tau_{\max}]$ is uniformly bounded and the
  analytic dependence of the negative eigenvalues on $\tau.$
\end{remark}
\begin{proof}
  Since $L_{-\tau,j}=\overline{L_{\tau,j}},$ we know that
  $\sigma_{\pp}(L_{\tau,j})=\sigma_{\pp}(L_{-\tau,j}),$ and so it
  suffices to consider $\tau>0.$ That all $\tau^{1}_{jk}(\lambda)\geq
  \sqrt{\lambda}$ follows from the fact that the perturbation terms,
  $\tQ_1^{\epsilon_j,\mu_j}(\tau)+
  \tQ_0^{\epsilon_j,\mu_j}+\omega^2(\kappa_0^2-\kappa_j^2),$ are
  compactly supported. Therefore an $L^2$-solution to~\eqref{eqn49.60}
  in the complement of ball can only exist if $(\lambda-\tau^2)\leq
  0.$

  We rewrite
  $$\tQ_1^{\epsilon_j,\mu_j}(\tau)=\tau\tQ_{10}^{\epsilon_j,\mu_j}+\tQ_{11}^{\epsilon_j,\mu_j},\text{ and }
   \hQ_0^{\epsilon_j,\mu_j}= \tQ_0^{\epsilon_j,\mu_j}+\omega^2(\kappa_0^2-\kappa_j^2).$$
  For a variable $\tau,$ we express the condition in~\eqref{eqn49.60}
  in the weak form
  \begin{multline}
  \int_{\bbR^2}  \left[\left|d\cxi\right|^2+  \left|d^* \cxi\right|^2+ \left\langle
          [\tQ_{11}^{\epsilon_j,\mu_j}+\hQ_{0}^{\epsilon_j,\mu_j}]
          \cxi, \cxi\right\rangle\right]dx= \\
  \int_{\bbR^2}\langle
 [\lambda-\tau^2-\tau\tQ^{\epsilon_j,\mu_j}_{10}]\cxi, \cxi\rangle dx.
  \end{multline}
  
 The left hand side is a compactly supported perturbation of the Dirichlet
 form, and so is bounded below by $-M\| \cxi\|^2,$ for a finite $M.$
 On the other hand, the quadratic dependence of
 $\lambda-\tau^2-\tau\tQ^{\epsilon_j,\mu_j}_{10}$ allows
 the r.h.s. to become an arbitrarily large, negative multiple of
 $\|\cxi\|^2,$ as $|\tau|\to\infty.$ Hence there is a $\tau_{\max}$
 as stated in the proposition. We say that the value of $\tau$ is
 \emph{allowed} if $|\tau|\in[\sqrt{\lambda},\tau_{\max}].$
   
For each $\tau$ in the allowed intervals, the operator $L_{\tau,j}$ has
finitely many negative eigenvalues $\{\lambda_l(\tau)\}.$ A simple
comparison argument using the min-max principle shows that there is an
integer $J$ so that, for $|\tau|\leq \tau_{\max},$ there are at most $J$
such eigenvalues.  The proposition follows once we show that the curve
$\{(\tau,\lambda-\tau^2):|\tau|\leq \tau_{\max}\}$
intersects these finitely many curves, $\{(\tau,\lambda_l(\tau))\}$ in
a finite number of points $\{\tau^{1}_{jl}(\lambda):\:l=1,\dots,
N^{1}_j\}.$

 First assume that $L_{\tau,j}$ does not have resonances at
 $\sqrt{\lambda}.$ In this case we only need to consider intersections between
 strictly negative eigenvalue branches defined for
 $|\tau|\in(\lambda,\tau_{\max}]$ and the curve
   $$\cC_p=\{(\tau,\lambda-\tau^2): \omega\kappa_0<|\tau|<
   \tau_{\max}\}.$$ In particular such an eigenvalue is less than $0$
   and therefore we can use analytic perturbation theory,
   see~\cite{Kato}, to conclude that there are analytic functions
   $(\lambda(\tau),u(\tau)),$ such that
 \begin{equation}
   L_{\tau}u(\tau)=\lambda(\tau)u(\tau)\text{ with }\|u(\tau)\|^2_{L^2}=1.
 \end{equation}
 A standard calculation shows that
 \begin{equation}\label{eqn110.223}
   \dot{\lambda}(\tau)=\langle \dot{L}_{\tau}u(\tau),u(\tau)\rangle=
   \langle \tQ^{\epsilon_j,\mu_j}_{10}u(\tau),u(\tau)\rangle
   =-4\Im\langle \cxi_{1\tau},d\ellk\lrcorner\cuxi_{\tau}\rangle,
 \end{equation}
 and therefore,  we have the
 estimate
 \begin{equation}\label{eqn138.400}
   \begin{split}
     |\dot{\lambda}(\tau)|<&4\|d\ellk\|_{L^{\infty}}\|\cxi_{1\tau}\|(\|\cxi_{2\tau}\|
     +\|\cxi_{3\tau}\|)\\
     &<2\sqrt{2}\|d\ellk\|_{L^{\infty}}.
     \end{split}
 \end{equation}
 
 This shows that it is impossible for a branch to satisfy
 $\lambda_l(\tau)=\lambda-\tau^2.$ A simple compactness
 argument, combined with the bound on the maximum number of
 eigenvalues, for any allowable $\tau,$ and the fact that eigenvalue
 branches can only appear, or disappear at $\lambda=0,$ and therefore
 have analytic representations,
 finishes the proof in the non-resonance case.

 Now suppose that $L_{\tau,j}$ has a resonance at $\sqrt{\lambda};$
 for definiteness assume that there are analytic functions
 $(\lambda_0(\tau),v_0(\tau))$ defined in
 $(\sqrt{\lambda},\sqrt{\lambda}+\epsilon)$ for an $\epsilon>0,$ with
 \begin{equation}\label{eqn112.223}
 L_{\tau,j}v_0(\tau)=\lambda_0(\tau)v_0(\tau),   
 \end{equation}
 and
 \begin{equation}
   \lim_{\tau\to[\omega\kappa_0]^+}\lambda_0(\tau)=0\text{ and }\|v_0(\tau)\|_{L^2}=1.
 \end{equation}
 
 The sequence $\{v_0(\tau)\}$ must converge weakly to zero as
 $\tau\to[\sqrt{\lambda}]^+,$ otherwise the (non-zero) limit would be
 an $L^2$-eigenfunction of $L_{\tau,j}$ with eigenvalue equal to zero,
 which we have excluded. The Rellich theorem, elliptic regularity,
 and~\eqref{eqn112.223} imply that, for any $R>0,$
 $v_0(\tau)\restrictedto_{B_R(0)}$ converges strongly to zero. Since
 $\tQ^{\epsilon_j,\mu_j}_{10}$ is compactly
 supported,~\eqref{eqn110.223} implies that
 \begin{equation}
   \lim_{\tau\to[\omega\kappa_0]^+}\dot{\lambda}_0(\tau)=0.
 \end{equation}
 Using the mean value theorem, we easily conclude that there cannot be
 an infinite sequence $\tau_l\to[\sqrt{\lambda}]^+,$ where
 \begin{equation}\label{eqn116.234}
   \tau_l^2+\lambda_0(\tau_l)-\lambda=0,
 \end{equation}
 see Remark~\ref{rmk5.333}. Hence there is a smallest
 $\tau_1>\sqrt{\lambda}$ where
 \begin{equation}
    \lambda_0(\tau_1)=\lambda-\tau_1^2<0.
 \end{equation}

\end{proof}
\begin{remark}
  The proposition shows that there for each $j$ and
  $0<\lambda,$ there is a non-negative number
  $\gamma_j(\lambda)$ so that
  \begin{equation}
    \gamma_j(\lambda)=\min\{[\pm\tau_{jk}^{1}(\lambda)]^2-\lambda:\: k=1,\dots,N_j^{1}\}.
  \end{equation}
  If this number is positive, then we say that $\gamma_j(\lambda)$ is the
  \emph{gap at the bottom of the continuous spectrum} of the family
  $L_{\tau,j}$ at the eigenvalue $\lambda.$ Suppose that $L_{\tau,j}$
  has a resonance at $\pm\sqrt{\lambda};$ let $\tau_{\delta}$ denote
  the solutions
  \begin{equation}\label{eqn118.223}
    -\lambda_0(\tau_{\delta})=\lambda+\delta-\tau^2_{\delta},
  \end{equation}
  for small $\delta.$ It is clear that $\lim_{\delta\to
    0}\lambda_0(\tau_{\delta})=0.$ So that the family of operators
  $L_{\tau,j}$ does not have a uniform sized gap at the bottom of the
  continuous spectrum at  $\tau=\sqrt{\lambda+\delta},$ as
  $\delta\to 0.$ A similar situation can also arise at
  $\tau=-\sqrt{\lambda}.$

  For our simplified proof of the Mourre inequality for wave-guides we
  assume that, for some $\delta_0>0,$ the ``family'' gap:
  \begin{equation}\label{eqn119.224}
    \Gamma(\lambda,\delta_0)=\inf\{\gamma_j(\lambda+\alpha):\:
    |\alpha|\leq\delta_0, j=1,\dots,m\}
  \end{equation}
  is positive. This will be the case, for some $\delta_0>0,$ if and
  only if the families $\{L_{\tau,j}:\: j=1,\dots,m\}$ do not have
  have  families of eigenvalues $\lambda_0(\tau)$ converging to $0$ as
  $\tau\to\pm\sqrt{\lambda}.$ 
\end{remark}
\begin{remark}\label{rmk5.333}
  If  $0\in\sigma_{\pp}(L_{\pm\sqrt{\lambda},j}),$ then
  there may or may not be a branch of negative eigenvalues
  $\lambda(\tau),$ for $\pm\tau\in
  (\sqrt{\lambda},\sqrt{\lambda}+\delta),$ converging to $0$ as
  $\tau\to\sqrt{\lambda}.$  If no such such branch exists then the
  results of Proposition~\ref{prop2.224} clearly remain true. If such
  a branch exists, then it still follows from~\eqref{eqn110.223} that
  the limits
  \begin{equation}
    \lim_{\tau\to\pm\omega\kappa_0^{\pm}}\dot{\lambda}(\tau)=d
  \end{equation}
  exist. If there were infinitely many roots $\{\tau_l\}$ (for
  definiteness greater than $\sqrt{\lambda}$) to the equation
  in~\eqref{eqn116.234}, then the mean value theorem implies that
  there is an infinite  sequence $\tau_{l+1}<\tau'_{l}<\tau_{l}$ so
  that
  \begin{equation}
    \dot{\lambda}(\tau'_l)=-2\tau_l'.
  \end{equation}
  This would imply that
  \begin{equation}
    \lim_{\tau\to \omega\kappa_0^+}\dot{\lambda}(\tau)=-2\sqrt{\lambda}.
  \end{equation}
  If $\sqrt{2}\|d\ellk\|_{L^{\infty}}<\sqrt{\lambda},$
  then~\eqref{eqn138.400} shows that this is not possible, and again
  the conclusion of the theorem holds in this case. Hence the
  conclusion holds for sufficiently large energies.

  Of particular interest is the case $\lambda=\omega^2\kappa_0^2.$
  Even if $\sqrt{2}\|d\ellk\|_{L^{\infty}}>\omega\kappa_0,$
  Proposition~\ref{propNew2} implies that there cannot be infinitely many
  solutions to $\lambda(\tau)=\omega^2\kappa_0^2-\tau^2.$
\end{remark}

  To establish radiation conditions and prove uniqueness results
  for~\eqref{eqn27.2} and thereby also Maxwell's equations, we need to
  establish a Mourre estimate for this operator at the energy
  $\lambda=\omega^2\kappa_0^2,$ see~\eqref{eqn27.1}.  More generally
  to prove the Mourre estimate for this operator, at energy $\lambda,$
  it is essential that for each eigenvalue function $\lambda_l(\tau)$
  of the indicial operator, $L_{\tau,j},$ such that there is a
  solution at \emph{phase velocity}, $0<\tau_0,$ to
  \begin{equation}\label{eqn138.500}
    \lambda_l(\tau_0)=\lambda-\tau_0^2,
  \end{equation}
  for which it is also true that the \emph{group velocity},
  \begin{equation}\label{eqn139.500}
    \dot{\lambda}_l(\tau_0)+2\tau_0> 0.
  \end{equation}
  In physical terms, this is requirement that the sign of the group
  velocity and phase velocity agree.

  To prove this seems to require some hypothesis on the wave-number
  $\kappa(x).$  We can prove such a result under two rather different
  hypotheses. 
  \begin{proposition}\label{prop4} Let $L_{\tau,j}$ be the operator defined
    in~\eqref{eqn48.61}.
    \begin{enumerate}
    \item If $\sqrt{2}\|d\log\kappa_j\|_{L^{\infty}}<\sqrt{\lambda}$ then at any
    solution $\tau_0$ to~\eqref{eqn138.500} the inequality
    in~\eqref{eqn139.500} also holds.
      \item Let $\kappa_j^{\max}=\max\{\kappa_j(x'):\:x'\in\bbR^2\}.$
        \begin{equation}\label{cntrst_ratio_est}
          1<\frac{\kappa_j^{\max}}{\kappa_0}<\sqrt{2},
          \end{equation}
then for any $\tau_0$ such that~\eqref{eqn138.500} holds with
        $\lambda=\omega^2\kappa_0^2$ the estimate
        in~\eqref{eqn139.500} also holds.
    \end{enumerate}
  \end{proposition}
  \begin{remark}
    The second part of the proposition gives a condition for the group
    velocity and phase velocity to have the same sign that is
    independent of $\omega.$ The same result holds for bi-infinite
    wave-guides with piecewise constant $\kappa(x').$ The proof gives
    an explicit bound:
    \begin{equation}\label{eqn154.700}
      \dot{\lambda}_l(\tau_0)+2\tau_0>\omega\kappa_0\left[2-\left(\frac{\kappa^{\max}_j}{\kappa_0}
        \right)^2\right]\left[1+\left(\frac{\kappa^{\max}_j}{\kappa_0}\right)^2\right].
    \end{equation}
  \end{remark}
  \begin{proof}
    The proof of 1. is straightforward: From~\eqref{eqn138.400} we
    know that $|\dot{\lambda}_l(\tau)|<2\sqrt{2}\|d\ellk\|_{L^{\infty}};$ on
    the other hand, for~\eqref{eqn138.500} to hold, $\tau_0$ must be
    greater than $\sqrt{\lambda}.$
    Therefore~\eqref{eqn139.500} holds if
  \begin{equation}
    \sqrt{2}\|d\log\kappa_j\|_{L^{\infty}}\leq\sqrt{\lambda}.
  \end{equation}

  The proof of 2. is more involved.  Let $\cxi_{\tau}$ denote an analytic family
  of normalized solutions to
  $$L_{\tau,j}\cxi_{\tau}=\lambda(\tau)\cxi_{\tau},$$
  with $\lambda(\tau_0)=\omega^2\kappa_0^2-\tau_0^2.$ To simplify the notation we omit
  the $j$ and $\tau$ subscripts. 
  Integrating by parts, the PDE implies that
  \begin{equation}\label{eqn153}
    \|\frac{1}{\sqrt{\mu}}d_{\tau}\sqrt{\mu}\cxi\|^2+
    \|\frac{1}{\epsilon\sqrt{\mu}}d^*_{\tau}\epsilon\sqrt{\mu}\cxi\|^2=
    \omega^2\|\kappa\cxi\|^2.
  \end{equation}
  As before, let
  $\cxi=\cxi_1dx_1+\cuxi,$ and again use formula~\eqref{eqn110.223},
  which states
  \begin{equation}
    2q=\dot{\lambda}(\tau_0)=-4\Im\langle \cxi_{1},d_2\log\kappa\lrcorner\cuxi\rangle.
  \end{equation}
  We need to show that
  \begin{equation}
    g=2(\tau_0+q)>0.
  \end{equation}

  Letting $G=\kappa^{-2}_0-\kappa^{-2}\geq 0,$ and using the identity
  \begin{equation}
 \kappa^{-2}d_2\log\kappa=-\frac{1}{2}d_2\kappa^{-2},
  \end{equation}
  we write
  \begin{equation}
    q=-2\Im\langle
    \sqrt{\mu}\cxi_1,\kappa^{-2}d_2\log\kappa\lrcorner\epsilon\sqrt{\mu}\cuxi)=
    -\Im\langle\sqrt{\mu}\cxi_1,d_2 G\lrcorner\epsilon\sqrt{\mu}\cuxi\rangle.
  \end{equation}
  Using the fact that $d_2^*(f\alpha)=fd^*\alpha-df\lrcorner\alpha,$ and
  integrating by parts we let $q=q_1+q_2,$ where
  \begin{equation}
 q_1=\Im\langle d_2(\sqrt{\mu}\cxi_1),G\epsilon\sqrt{\mu}\cuxi\rangle,\quad
 q_2=-\Im\langle\sqrt{\mu}\cxi_1,G d_2^*(\epsilon\sqrt{\mu}\cuxi)\rangle.
  \end{equation}

  We easily show that for a 1-form $\theta=\theta_1dx_1+\cutheta$
  \begin{equation}
    \|\mu^{-\frac 12}d_{\tau}\theta\|^2=\|\mu^{-\frac
      12}d_{2}\cutheta\|^2+
    \|\mu^{-\frac 12}(d_{2}\theta_1-i\tau_0\cutheta)\|^2.
  \end{equation}
  Applying this with $\theta=\sqrt{\mu}\cxi,$ and using~\eqref{eqn153}
  shows that
  \begin{equation}
    \|\mu^{-\frac
      12}(d_{2}\sqrt{\mu}\cxi_1-i\tau_0\sqrt{\mu}\cuxi)\|^2\leq
    \omega^2\|\kappa\cxi\|^2.
  \end{equation}
  Using this identity and the Cauchy-Schwarz inequality  shows that
  \begin{equation}
    \begin{split}
    q_1&=-\Im\langle\mu^{-\frac 12}
    (d_{2}\sqrt{\mu}\cxi_1-i\tau_0\sqrt{\mu}\cuxi)
    ,G\kappa^2\cuxi\rangle+
    \tau_0\|G^{\frac 12}\kappa\cuxi\|^2\\ &\geq
    -\omega\left(\frac{\kappa_1^2}{\kappa_0^2}-1\right)^{\frac 12}
    \|\kappa\cxi\|\|G^{\frac 12}\kappa\cuxi\|+\tau_0\|G^{\frac
      12}\kappa\cuxi\|^2.
    \end{split}
  \end{equation}

  To estimate $q_2$ we use the identity
  $d_{\tau}^*\theta=d_2^*\cutheta-i\tau\theta_1$ to rewrite
  \begin{equation}
    q_2=-\Im\left\langle\kappa^2
    G\cxi_1,\frac{1}{\epsilon\sqrt{\mu}}d_{\tau}^*(\epsilon\sqrt{\mu}\cxi)\right\rangle+
    \tau_0\|\kappa G^{\frac 12}\cxi_1\|^2.
  \end{equation}
  As~\eqref{eqn153} implies that
  $\|(\epsilon\sqrt{\mu})^{-1}d_{\tau}^*(\epsilon\sqrt{\mu}\cxi)\|\leq \omega\|\kappa\cxi\|,$
    we see that
    \begin{equation}
      q_2\geq
      -\omega\left(\frac{\kappa_1^2}{\kappa_0^2}-1\right)^{\frac
        12}\|\kappa\cxi\|
      \|G^{\frac 12}\kappa\cxi_1\|+\tau_0\|\kappa G^{\frac 12}\cxi_1\|^2.
    \end{equation}

    Let $\kappa_1=\kappa_j^{\max}$; combining these estimates gives the lower bound:
    \begin{equation}
            g\geq 2[\tau_0+\tau_0\|\kappa G^{\frac 12}\cxi\|^2-
        \omega\kappa_1(\|G^{\frac 12}\kappa\cxi_1\|+\|G^{\frac 12}\kappa\cuxi\|)].
    \end{equation}
   as $\omega\kappa_0\leq\tau_0,$ this estimate
    implies that
    \begin{equation}\label{eqn169}
      g\geq 2\omega\kappa_0\left[1-\sqrt{2}\frac{\kappa_1}{\kappa_0}
        \left(\frac{\kappa_1^2}{\kappa_0^2}-1\right)^{\frac
          12}\|\kappa G^{\frac 12}\cxi\|+\|\kappa G^{\frac
          12}\cxi\|^2\right].
    \end{equation}
    The polynomial $1-\lambda
    x+x^2$ is strictly positive if  $\lambda<2,$ which is the
    case in \eqref{eqn169} if $\frac{\kappa_1}{\kappa_0}<\sqrt{2}.$ In
    this case the minimum value is $1-\frac{\lambda^2}{4}.$
 
  \end{proof}

  \begin{remark}
    If the wave-guide mode, $\cxi_{\tau_0}$ used in the proof happens
    to be either divergence-free or curl-free, then we can improve
    the estimate in~\eqref{eqn169} to
    \begin{equation}
      g\geq 2\omega\kappa_0\left[1-\frac{\kappa_1}{\kappa_0}
        \left(\frac{\kappa_1^2}{\kappa_0^2}-1\right)^{\frac
          12}\|\kappa G^{\frac 12}\cxi\|+\|\kappa G^{\frac
          12}\cxi\|^2\right].
    \end{equation}
    This quantity is positive if
    \begin{equation}
      \frac{\kappa_1}{\kappa_0}
      <\sqrt{\frac{1+\sqrt{17}}{2}},
    \end{equation}
   which is larger than 1.6. Even though wave-guide modes can always be split into a
   divergence-free part and a curl-free part, it is by no means
   obvious that an analytic analytic family of eigenfunctions
   $\{\cxi_{\tau}\}$ satisfying
   \begin{equation}
     L_{\tau,j}\cxi_{\tau}=\lambda(\tau)\cxi(\tau),
   \end{equation}
   with $\lambda(\tau_0)=\omega^2\kappa_0^2-\tau_0^2$ can be split in
   this way. So we cannot generally replace the $\sqrt{2}$
   in~\eqref{cntrst_ratio_est} with this larger number.
  \end{remark}

\begin{remark}\label{rmk8.500}
  Let $c>0,$ and define
  \begin{equation}\label{eqn5.89.93}
    \lambda_M(c)=\inf\{\mu:\text{ if~\eqref{eqn138.500} holds
        for a }\lambda>\mu,\text{ then } \dot{\lambda}_l(\tau_0)+2\tau_0\geq c\}.
  \end{equation}
  We let $\IM(c)=(\lambda_M(c),\infty)$ be the \emph{Mourre interval};
  we will be able to prove a Mourre estimate for energies in $\IM.$ It
  follows from Proposition~\ref{prop4} that, for some $c>0,$
  $\lambda_M(c)$ is finite. It remains a very interesting question
  whether or not $\lambda_M(c)\to 0,$ as $c\to 0.$
\end{remark}

\subsection{Characteristic Variety and Radial Sets for $\cH^{\epsilon,\mu}_{\omega}$}\label{tau_symm}

As $\kappa^2\cD_2^{\mu,\epsilon}=
\star\kappa^2\cD_1^{\epsilon,\mu}\star,$ there is no obvious reason
why the wave-guide frequencies for 1-forms should be the same as for
2-forms. To define the characteristic variety and radial sets for
$\cH^{\epsilon,\mu}_{\omega},$ which has both a 1-form and a 2-form part, we
therefore need to take the union of the $\tau$s that arise separately
for the 1-form and 2-form parts.  The analysis of the 2-form operator
is essentially the same as the 1-form case.  Hence, under assumptions
like those made in Proposition~\ref{prop2.224}, there is a finite set
of $\tau$s,
 $$\cT_2=\bigcup_{j=1}^m\{\pm\tau^{2}_{jk}:\:k=1,\dots,N_j^{2}\}.$$
with $|\tau|\in[\omega\kappa_0,\tau_{\max}],$ for which 2-form
wave-guide modes exist at
$$k_0^2=\omega^2\kappa_0^2.$$

  Divergence-free  1-form wave-guide
  modes automatically define 2-form modes as well. These are the
  wave-guide modes of interest for Maxwell's equations. Similarly, if
  we look at \emph{closed} 2-form wave-guide modes, then they automatically
  define Maxwell wave-guide modes. Thus the wave-guide 
  frequencies for  Maxwell equations are a subset of $\cT_1\cap
  \cT_2.$

Using Proposition~\ref{propNew2} we can describe the characteristic variety and
radial sets for the 1-form operator,
$$\kappa^2\cD_1^{\epsilon,\mu}+\omega^2(\kappa_0^2-\kappa^2)-\omega^2\kappa_0^2=
\kappa^2\cD_1^{\epsilon,\mu}-\omega^2\kappa^2,$$
over the front faces. The indicial operators
$\hN_{\ff_j,0}(\kappa^2\cD_1^{\epsilon,\mu}-\omega^2\kappa^2)(\tau)$
are invertible except for $\tau^2\leq \omega^2\kappa_0^2,$ and
$\{\pm\tau_{jk}^{1}:\: k=1,\dots,N^{1}_j\}.$ Thus the
characteristic variety, for the 1-form part, over $\ff_j$ is
\begin{equation}
  \Sigma^1_{k_0^2,\ff_j}=\{|\tau|\leq
  \omega\kappa_0\}\cup\{\pm\tau^{1}_{jk}:\: k=1,\dots, N^{1}_j\},
\end{equation}
and {\Bl the incoming (top line, $+$) and outgoing (bottom line, $-$)}
\Bk radial sets are
\begin{equation}
  R^{1,\pm}_{k_0^2,\ff_j}=
  \{\tau=
\pm  \omega\kappa_0\}\cup\{\pm\tau^{1}_{jk}:\:k=1,\dots,N^{1}_j\}.
\end{equation}
There are analogous definitions of $\Sigma^2_{k_0^2,\ff_j},
R^{2,\pm}_{k_0^2,\ff_j}$ for the 2-form operator,
$\kappa^2\cD_2^{\epsilon,\mu}-\omega^2\kappa^2,$ with
$\{\pm\tau^2_{jk}\}$ replacing  $\{\pm\tau^1_{jk}\}.$

\subsection{Uniqueness and Radiation Conditions for $\cH^{\epsilon,\mu}_{\omega}$}
Before stating the uniqueness result we briefly recall its proof.  For
this argument we need to know that $\cH^{\epsilon,\mu}_{\omega}$
satisfies a Mourre estimate at $\lambda=\omega^2\kappa_0^2,$ which
will be the case if either of the conditions in
Proposition~\ref{prop4} are satisfied.  Let $\cE$ denote the
frequencies $\omega,$ for which $\cH^{\epsilon,\mu}_{\omega}$ has an
$L^2$-eigenfunction at $\omega^2\kappa_0^2.$ Using Mourre estimates,
we show that $\cE$ is a discrete set, which we call the set of
\emph{exceptional frequencies}.  The propagation of singularities
results from~\cite{VasyAsterisque}, as extended in
Appendix~\ref{prop_est}, show that the 3-scattering wave-front sets of
these eigenfunctions are empty and therefore they actually belong to
Schwartz class. Unfortunately, the standard argument from the scalar
case showing that such eigenvectors must be super-exponentially
decaying and therefore, in fact, zero, is not readily applicable in
this case.

 The radiation conditions for the second order operators considered above are:
    \begin{theorem}\label{RadCondMax}    If
    $\cH^{\epsilon,\mu}_{\omega}$ satisfies a Mourre estimate at
    $k_0^2=\omega^2\kappa^2_0,$ and $\omega\notin\cE,$ 
    then    a solution $\txi,$ $\teta$ resp. to
    \begin{equation}\label{eqn131.337}
      \kappa^2\cD^{\epsilon,\mu}_1\txi=\omega^2\kappa^2\txi,\quad
      \kappa^2\cD^{\epsilon,\mu}_2\teta=\omega^2\kappa^2\teta\text{ resp. }
    \end{equation}
    for which the components satisfy the classical Sommerfeld
    radiation condition along $\mf$ and such that
    \begin{equation}\label{eqn132.337}
      \WF_{3\Sc}^{*,l}(\txi)\cap R^{1,-}_{k_0^2,\ff_j}=\emptyset,\quad
       \WF_{3\Sc}^{*,l}(\teta)\cap R^{2,-}_{k_0^2,\ff_j}=\emptyset\text{ resp.},
    \end{equation}
    for $j=1,\dots, m,$ and an $l>-\frac 12,$ is identically zero.

    If $\omega\in\cE,$ then such solutions belong to finite
    dimensional spaces of $L^2$-eigenfunctions. The set $\cE$ is
    discrete.
  \end{theorem}
  \begin{remark}
    Proposition~\ref{prop4} shows that these hypotheses are satisfied if either
    \begin{equation}
      \sqrt{2}\|d_2\log\kappa_j\|_{L^{\infty}}<\omega\kappa_0,\text{
        or }
   \frac{\kappa_j^{\max}}{\kappa_0}<\sqrt{2},
    \end{equation}
    for $j=1,\dots,m.$ It seems likely that these are not sharp and
    that uniqueness will hold much more generally.
    
    Using standard elliptic estimates one can show that if
    $2\|x\cdot\nabla_x\log\kappa\|_{L^{\infty}}<1,$ then the set $\cE$
    is bounded and therefore finite. There is a similar statement for
    incoming solutions.
  \end{remark}
  As a corollary we have radiation conditions for Maxwell's equations.
  \begin{corollary} Assume that $\cH^{\epsilon,\mu}_{\omega}$ satisfies a Mourre estimate at
    $\omega^2\kappa^2_0,$ and $\omega\notin\cE;$ let $(\xi,\eta)$ be a
    solution to Maxwell's equations~\eqref{eqn7}. If the pair
    $(\xi,\eta)$ satisfies the Sommerfeld radiation condition over
    $\mf,$ with $k_0=\omega\kappa_0,$ and the conditions
    in~\eqref{eqn132.337}, then $(\xi,\eta)=(0,0).$ If $\omega\in\cE,$
    then such a solution belongs to a finite dimensional space of
    Schwartz class eigenfunctions.
  \end{corollary}
  \begin{proof}
    The corollary follows from the theorem because, with the rescaling
    in~\eqref{eqn14.40}, a solution to~\eqref{eqn7} automatically gives
    solutions to the equations in~\eqref{eqn131.337}.
  \end{proof}

  \noindent
   A detailed proof of Theorem~\ref{RadCondMax} is given in
   Section~\ref{RadCondPf}. For solutions of Maxwell's equations there
   is a more refined description of the wave-front set given by the
   Silver-M\"uller radiation conditions, see~\eqref{MaxRadCond}
   and~\cite[Theorem 4.4]{ColtonKress}. These conditions will hold in
   the present case for directions along the $\mf$. In the channel
   directions the principal terms in the asymptotic expansion are
   outgoing wave-guide modes. The lower order terms probably satisfy
   the Silver-M\"uller conditions.

   \subsection{Proof of Theorem~\ref{RadCondMax}}\label{RadCondPf}
The proof of Theorem~\ref{RadCondMax} uses several standard
consequences of the Mourre estimate,~\eqref{eq:Mourre-0}, as well as
results from Chapter 17 of~\cite{VasyAsterisque}. The needed
modifications to the proofs of these results are described in
Appendix~\ref{prop_est}.  Once these things are taken care of, there
is still a subtlety that needs to be dealt with connected to the fact
that  both the operators $\cH^{\epsilon,\mu}_{\omega}$ and the eigenvalue of
interest, $\omega^2\kappa_0^2,$ depend on $\omega.$

Let $u$ be a solution to
$\cH^{\epsilon,\mu}_{\omega}u=\omega^2\kappa_0^2u,$ which satisfies the
hypotheses of Theorem~\ref{RadCondMax}.  The hypotheses of the theorem
imply that the hypothesis (17.41) in Proposition 17.8
of~\cite{VasyAsterisque} holds and that, in the notation
of~\cite{VasyAsterisque}, $(H-\omega^2\kappa_0^2)u=0.$ As noted above, we cannot
conclude that $u=0,$ but the argument in~\cite{VasyAsterisque} does imply
that $\WF_{3\Sc}(u)=\emptyset,$ and therefore
$u\in\cS(\overline{\bbR^3}_{\WG}),$ and hence is an
$L^2$-eigenvector.

The question now is whether
$(\cH^{\epsilon,\mu}_{\omega}-\omega^2\kappa_0^2)$ can have
non-trivial $L^2$-solutions. We have assumed that
$\cH^{\epsilon,\mu}_{\omega}$ satisfies a Mourre estimate at
$\omega^2\kappa_0^2,$ either because $\omega^2\kappa_0^2\in\cI_M,$ or
because the contrast ratios are small enough. It is a standard
consequence of this fact that there is at most a discrete set
$\{\lambda_k\}$ near to $\omega^2\kappa_0^2$ for which
$(\cH^{\epsilon,\mu}_{\omega}-\lambda_k)$ has non-trivial
$L^2$-solutions; each eigenspace is finite dimensional,
see~\cite{CFKS}. We now show that the set of exceptional frequencies,
$\cE,$ for which $\cH^{\epsilon,\mu}_{\omega}$ can have an
$L^2$-eigenfunction at $\lambda=\omega^2\kappa_0^2,$ is discrete.

We consider the 1-form part of $\cH^{\epsilon,\mu}_{\omega};$ the
analysis of the 2-form part is essentially identical.  The Mourre estimate for
$$\cH^{\epsilon,\mu}_{1,\omega}=\kappa^2
\cD^{\epsilon,\mu}_1-\omega^2(\kappa^2-\kappa_0^2)$$ localized at
energy $\omega^2\kappa_0^2 $ can be thought of as a positive
commutator estimate for $P(\omega)=\kappa^2
\cD^{\epsilon,\mu}_1-\omega^2\kappa^2$ at energy 0. With $A$ defined
in~\eqref{eqn134.338}, there is a $c,\delta>0$ so that if
$\phi\in\cC_c^{\infty}((-\delta,\delta)),$ then
$$i\phi(P(\omega))([A,P(\omega)]\phi(P(\omega))\geq
c\phi(P(\omega))^2-K,$$
where $K$ is compact or indeed finite rank.

Note that this bound is locally uniform in the parameter $\omega:$
if it holds for some $\omega,$ then replacing $\omega$ with a  nearby
$\omega'$ changes the LHS-$c\phi(P(\omega))^2$  by a bounded operator
of norm $\leq C|\omega-\omega'|.$  Hence
$$i\phi(P(\omega'))([A,P(\omega')]\phi(P(\omega'))\geq
c\phi(P(\omega'))^2-C|\omega-\omega'|-K.$$
If $|\omega-\omega'|$ is
sufficiently small, then multiplying through by $\psi(P(\omega'))$
from left and right, with $\psi$ of slightly smaller support than
$\phi,$ gives a lower bound with $K$ replaced by
$\psi(P(\omega')K\psi(P(\omega')),$ and $c$ replaced by $c/2.$ That
is, for $\phi$ with small enough support near zero, and
$|\omega-\omega'|$ sufficiently small, we have the estimates
\begin{equation}\label{eqn138.336}
  i\phi(P(\omega'))([A,P(\omega')]\phi(P(\omega'))\geq
  \frac{c}{2}\phi(P(\omega'))^2-\phi(P(\omega'))K\phi(P(\omega')).
  \end{equation}

Now suppose that $\cE$ has an accumulation point, with a sequence of
distinct frequencies $\{\omega_j\}\subset\cE$ converging to
$\omega^*,$ and $\cH^{\epsilon,\mu}_{\omega^*}$ satisfies a
Mourre estimate at $[\omega^*\kappa_0]^2$ We consider the family of
estimates like~\eqref{eqn138.336} for $|\omega-\omega^*|$ sufficiently
small. Let $u_j$ be a normalized $L^2$ eigenfunction of
$\cH^{\epsilon,\mu}_{1,\omega_j}$ with energy $\omega_j^2\kappa_0^2.$
These are also eigenvectors of $P(\omega_j)$ with energy 0, so in
particular can be written as $u_j=\phi(P(\omega_j))u_j $ where $\phi$
equals 1 near to 0.  Hence the estimates in~\eqref{eqn138.336} imply
that, for $|\omega^*-\omega_j|$ sufficiently small,
$$i\langle [A,P(\omega_j)]u_j,u_j\rangle\geq
\frac{c}{2}\|u_j\|^2-\langle Ku_j,u_j\rangle.$$ By the virial theorem,
the left hand side is 0 because $u_j$ is Schwartz and
$P(\omega_j)u_j=0.$ On the other hand
$$0=\langle (P(\omega_j)-P(\omega_k))u_j,u_k\rangle=\langle
\kappa^2(\omega_k^2-\omega_j^2)u_j,u_k\rangle.$$
Since $\omega_j^2\neq\omega_k^2,$ we see that
$$\langle u_j,u_k\rangle_{\kappa^2}=\langle \kappa^2 u_j,u_k\rangle=0.$$

As $\kappa$ is bounded above and below by positive constants, the
$\kappa^2$ and standard inner products are equivalent, and therefore
\begin{equation}\label{eqn139.336}
0=i\langle [A,P(\omega_j)]u_j,u_j\rangle\geq
c'\|u_j\|_{\kappa^2}^2-\langle
\kappa^{-2}Ku_j,u_j\rangle_{\kappa^2}.
\end{equation}
The $\kappa^2$-orthonormality of the $\{u_j\}$ implies that $\{u_j\}$
converges weakly to zero, in the $\kappa^2$-sense.  The compactness of
$K$ implies that $Ku_j\to 0$ strongly. However
\begin{equation}
  \liminf_{j\to\infty}\|u_j\|_{\kappa^2}^2=m>0,
\end{equation}
which contradicts~\eqref{eqn139.336} and shows that $\cE$ cannot have accumulation points.



\section{A Generalized Mourre Estimate}\label{ss.mourre}

In Section~\ref{Max_rad_cond} we extended the radiation conditions to
Maxwell wave-guides by replacing the first order Maxwell system with
the second order system~\eqref{eqn27.2}. As observed
in~\eqref{eqn34.40}, the first order part of the indicial operator
along $\ff$ of the perturbation has a term that depends linearly on
$\tau.$ This is something explicitly excluded in prior analyses and
requires an extension of the Mourre estimate. Thus far, we cannot
prove this estimate in full generality, but require a smallness
hypothesis on this term. We do prove this result in the full
generality of the $N$-body problem, as it requires only slightly more
effort than would be required in the 3-body case. We let $X$ denote a
closed manifold with a metric, $g,$ that is a perturbation of a scattering
metric, $g_0,$ which is of the form given in~\eqref{metric0}. In this
section $s$ denotes a defining function for $\pa X.$

Here we extend the Mourre estimate of \cite[Appendix]{VasyExpDecay},
proved there in a geometric setting; the extension is that {\em we do
  not assume the product structure of the indicial operators.} The
notation in this section and Appendix~\ref{MourreEstGenCase} follows
that in~\cite{VasyExpDecay}. {\Bl The Hamiltonians considered in the
  actual $N$-body case, where the underlying space is the
  compactification of $\bbR^d$, has collision planes that are
  orthocomplements $\{X_a:\ a\in\cA\}$ of the linear subspaces
  $\{X^a:\: a\in\cA\},$ indexed by a set $\cA.$ The subsystem
  Hamiltonians ``live'' on the $\{X^a\}$, that is they act on function
  spaces on these geometric spaces. By convention, the partial order
  is based on the inclusions of the $X^a$.  For the minimal element,
  $\{0\},$  indexed by $a=0$,  $X_0$ is the whole space $\bbR^d,$
  with $\pa\overline{X_0}=C_0=\pa\overline{\RR^d}$, $X^0=\{0\}.$ 
  The maximal element $\bbR^d$ is indexed by $1,$ so that $X_1=\{0\}$
  is disjoint from the boundary $\pa\overline{\bbR^d}.$ One usually
  assumes that the corresponding interaction vanishes; otherwise it
  would simply change the spectral parameter. It is convenient to
  assume that the set $\{X_a:\ a\in\cA\}$ is closed under interactions. If $H$
  is the Hamiltonian, and $V_a$ are the perturbations of
  $\Delta_{\bbR^d},$ which depend only on variables in $X^a,$ then the
  subsystem Hamiltonians are
  \begin{equation}
    H^a=\Delta_{X^a}+\sum_{X^b\subset X^a}V_b.
  \end{equation}
  Note that, as $X^0=\{0\},$
  \begin{equation}
    H^0=0,
  \end{equation}
  see~\cite{Vasy:Propagation-Many}.
  
The 2-body case corresponds to having only 2 levels, with $X_0$ and
$X_1$ the only collision planes. The maximal totally ordered
subset of $\{X_a:\ a\in\cA\}$, i.e.\ chain, has length 2. The
3-body case occurs if $a\neq b$, and  $a,b\neq 0,1$, implies that $X_a\cap X_b=\{0\}.$
Hence the collision planes at infinity, $C_a=\pa\overline{X_a}$, are
disjoint (for $a\neq 0$). The maximal chain of the $X_a$ has
length 3. Wave guides are examples of  3-body problems with the
additional restrictions that $X=\bbR^2$ or $X=\bbR^3$, and  the
$X_a$, for $a\neq 0,1$, are 1-dimensional. The triviality of
their intersections is automatic. Correspondingly, at $\pa X$, the
hierarchy consists of $\pa X_0=\pa X$, and $\pa X_a$ ($a\neq 0,1$),
which each consist of two antipodal points, the channel ends.} \Bk

\subsection{The Product Case}

In~\cite{VasyExpDecay} it was assumed  that, restricted to a front
face, $g-g_0$ is a section of $T^*X^a\otimes T^*X^a$, i.e.\ depends on
the interaction variables only over $C_a$, and that for each $y_a\in
C_a$, $\hat V_a(y_a,\xi_a) \in\Diff^1(X^a;E)$ is independent of
$\xi_a$.  In other words, the assumption was that all indicial
operators of $H$ are pointwise in $C_a$ of `product-type', i.e.\ have the
form
\begin{equation}\label{eq:H_a-ind-pt-pr}
\hat H_a(\xi_a)=|\xi_a|^2_{y_a}\otimes\Id_E+H^a(y_a),
\end{equation}
where $H^a(y_a)\in\Diff^2(X^a;E)$, does not depend on $y_a,$ which is the variable along
$C_a$. In the wave-guide case, each $C_a$ reduces to a
channel end, which are isolated points, and $\xi_a\leftrightarrow\tau.$

Before discussing the extension, let us recall the Mourre estimate in
this `product' setting.
Let $\Lambda$
denote the thresholds of $H$, i.e.\ the set of the $L^2$ eigenvalues
of all of its subsystem Hamiltonians $H^a$;
this is a closed countable subset of
$\RR$. (Note that $0$ is an eigenvalue of $H^0$, hence it is a threshold
of $H$; if no non-trivial subsystem has an $L^2$-eigenvalue then
$\Lambda=\{0\}$.)
Let
\begin{equation*}
d(\lambda)=\inf\{\lambda-\lambda':\lambda'\leq\lambda,\ \lambda'\in\Lambda\}
\end{equation*}
be the distance of $\lambda$ to the next threshold \emph{below} it. If
$\lambda<\inf\Lambda,$ then $d(\lambda)$ is an arbitrary positive number.

Let $\chi\in\Cinf_c([0,1))$ be supported near $0$, identically $1$ on
  a smaller neighborhood of $0$. Recall that $s$ is a defining function
  for $\pa X;$ let $D_s=-i\pa_s,$ and
  \begin{equation}\label{eqn134.338}
     A=\half(\chi(s) sD_s+(\chi(s)sD_s)^*).
   \end{equation}
   The Mourre estimate at $\lambda$ is the statement that for all
  $\epv>0$ there is $\delta>0$ such that for all
  $\phi\in\CI_c((\lambda-\delta,\lambda+\delta);[0,1])$
\begin{equation}\label{eq:Mourre-00}
\phi(H)i[A,H]\phi(H)\geq 2(d(\lambda)-\epv)\phi(H)^2+K',
\end{equation}
$K'\in\PsiSc^{-\infty,1}(\Xb,\cC)$. Note the $K'$ is a compact
operator. The key point is that, for small enough $\epv>0,$
$d(\lambda)-\epv>0.$

\subsection{The Non-product Case}
We now consider Hamiltonians that do not possess the product structure in
the sense that $\hat V_a(y_a,\xi_a)
\in\Diff^1(X^a;E)$ is {\em not} required to be independent of
$\xi_a$. (We assume that $V_0=0$ as usual.) More explicitly then, in view of $V$ being a first order
differential operator,
$$
\hat V_a(y_a,\xi_a)=\hat
V^0_a(y_a)+\hat V^1_a(y_a)\xi_a,
$$
where $\hat V^0_a(y_a)$ is a first order operator tangent to the fibers of
the front face and $\hat V^1_a(y_a)$ is a multiplication operator.
{\Bl We need to make assumptions on operator
$$
H^a(y_a,\xi_a)=\Delta_{g^a}+\hat V(y_a,\xi_a).
$$ The precise nature of these assumptions depends on whether the
eigenspaces of the operator $H^a(y_a,\xi_a)$ are finite
dimensional. The reason the dimension plays a role is ultimately due
to compactness considerations; if $\phi\in\cC^\infty(\bbR;[0,1])$ with
support in $(\lambda'-\delta,\lambda'+\delta)$, $\phi(\lambda')=1$,
then as $\delta\to 0$, $\phi(H^a)$ converges strongly $E$, where $E$
is the spectral projection for $\lambda'.$ Thus for $K$ compact,
$(\phi(H^a)-E)K\to 0$ in norm. On the range of $E$ one can make the
assumption we make in the finite dimensional case, but the pairings
between the range of $E$ and its orthocomplement are only easy to
control if $E$ is finite rank. Ultimately this again relies on
compactness considerations. In the general infinite rank case we need
to make an extra smallness assumption. Weaker hypotheses suffice in the finite rank
case, as well as in the infinite rank case with an additional
compactness property.

We start with the `ideal' case.

\begin{definition}\label{IdMourreConst}
 Suppose $\lambda\in\RR$. We say that the {\em ideal Mourre constant} at the level of an eigenvalue
$\lambda'=\lambda-|\xi_a|^2$ of $H^a(y_a,\xi_a)$ with eigenprojector $E$ is $M^a_{(y_a,\xi_a)}(\lambda)$
\begin{equation}\label{eq:ideal-Mourre-constant}
M^a_{(y_a,\xi_a)}(\lambda)=2
(\lambda-\lambda')+C^a_{(y_a,\xi_a)}
\end{equation}
where $C^a_{(y_a,\xi_a)}(\lambda)$ is the largest constant (i.e.\ the
$\sup$ of all constants) such that
\begin{equation}\label{eq:ideal-Mourre-constant-2}
\lambda-\lambda'=|\xi_a|^2\Rightarrow
|\xi_a|E\Big[\frac{\xi_a}{|\xi_a|}\pa_{\xi_a},H^a\Big]E\geq C^a_{(y_a,\xi_a)}(\lambda) E.
\end{equation}
In case $\lambda'=\lambda-|\xi_a|^2$ is not an eigenvalue of
$H^a(y_a,\xi_a)$, we set $C^a_{(y_a,\xi_a)}(\lambda),$ and
$M^a_{(y_a,\xi_a)}(\lambda)$ to be $+\infty$.

We say that the ideal Mourre constant at the level of a subsystem
$a\in\cA$ at energy $\lambda$ is
\begin{equation}\label{Ma_def}
M^a(\lambda)=\inf_{(y_a,\xi_a)}M^a_{(y_a,\xi_a)}(\lambda),
\end{equation}
understood to be $+\infty$ if $\lambda-|\xi_a|^2\notin \sigma_{\pp}(H^a(y_a,\xi_a))$ for all
$(y_a,\xi_a)$.
  \end{definition}

  The slightly artificial phrasing of
  \eqref{eq:ideal-Mourre-constant}-\eqref{eq:ideal-Mourre-constant-2}
  becomes clearer by replacing $H^a$
  with
  $$
  \hat H_a(y_a,\xi_a)=|\xi_a|^2_{g_a}+H^a(y_a,\xi_a);
  $$
  then under the assumption $|\xi_a|^2=\lambda-\lambda'$ the right
  hand side of \eqref{eq:ideal-Mourre-constant} becomes the largest
  constant $M^a_{(y_a,\xi_a)}(\lambda)$
  \begin{equation}\label{eq:Mourre-hyp-H_a}
|\xi_a|E\Big[\frac{\xi_a}{|\xi_a|}\pa_{\xi_a},\hat H_a(y_a,\xi_a)\Big]E\geq M^a_{(y_a,\xi_a)}(\lambda)E.
\end{equation}

If $H^a$ is actually independent of $\xi_a$, as in the
product setting, then $[\frac{\xi_a}{|\xi_a|}\pa_{\xi_a},H^a]=0$,
so
$$
M^a_{(y_a,\xi_a)}(\lambda)=2(\lambda-\lambda'),
$$
understood to be $+\infty$ if $\lambda'=\lambda-|\xi_a|^2$ is not an
eigenvalue of $H^a$, and correspondingly
\begin{equation}
M^a(\lambda) =2(\lambda-\lambda'),
\end{equation}
with the same convention.

To give a more spectral theoretic meaning to the assumption, for a
point $(y_a^0,\xi_a^0)$ suppose that
$|\xi_a^0|^2=\lambda-\lambda'(y_a^0,\xi_a^0)$, $E(y_a^0,\xi_a^0)$ is
the orthogonal projection to the $\lambda'(y_a^0,\xi_a^0)$-eigenspace,
and suppose that there are smooth functions
$\lambda'=\lambda'(y_a,\xi_a)$ and $E=E(y_a,\xi_a)$ taking these
values at $(y_a^0,\xi_a^0)$. Then for any
$\psi(y_a,\xi_a)$ in the $\lambda'(y_a,\xi_a)$
eigenspace,
\begin{equation*}\begin{aligned}
\Big\langle
\psi,\Big[\frac{\xi_a}{|\xi_a|}\pa_{\xi_a},H^a\Big]\psi \Big\rangle&=\Big\langle
\psi,\frac{\xi_a}{|\xi_a|}\pa_{\xi_a}H^a\psi \Big\rangle-\Big\langle\psi,
H^a\frac{\xi_a}{|\xi_a|}\pa_{\xi_a}\psi \Big\rangle\\
&=\Big\langle
\psi,\frac{\xi_a}{|\xi_a|}\pa_{\xi_a}\lambda'\psi \Big\rangle-\Big\langle\lambda'\psi,\frac{\xi_a}{|\xi_a|}\pa_{\xi_a}\psi \Big\rangle\\
&=\Big\langle \psi, \frac{\xi_a}{|\xi_a|}(\pa_{\xi_a}\lambda')\psi \Big\rangle.
    \end{aligned}\end{equation*}
Note that, in the interesting cases for the waveguide setting,  $a\neq
0,1$,  $|\xi_a|=|\tau|$, and
$\frac{\xi_a}{|\xi_a|}\pa_{\xi_a}=\sign(\tau)\pa_\tau$. The ideal
Mourre constant is the largest $M^a(\lambda)$ so that
$$
\tau^2=\lambda-\lambda'\Rightarrow 2|\tau|^2+\tau\pa_\tau\lambda'\geq M^a(\lambda).
$$
This is implied by the group velocity condition in \eqref{eqn139.500}
with
$$
M^a(\lambda)=\inf
\sqrt{\lambda-\lambda_l(\tau_0)}(\dot\lambda_l(\tau_0)+2\tau_0).
$$
The ideal Mourre constant plays essentially the same role as the
constant `$c$' in the definition of Mourre interval given in
Remark~\ref{rmk8.500}.

Next, we explain the smallness conditions that may be needed when
eigenspace is finite dimensional.

\begin{definition} We say that $\hat V_a$ is $\tilde\delta_{1,a}$-small if the zeroth order,
`non-product', part of $\hat V_a$, namely $\hat V^1_a(y_a)\xi_a$ satisfies
\begin{equation}\label{eq:hat-V-1-est}
  |\hat V^1_a|\leq \tilde\delta_{1,a};
\end{equation}
notice that for $a=0$, $\tilde\delta_{1,0}=0$.
\end{definition}
} \Bk


{\Bl Returning now to the general $N$-body case, in fact, we do not
  always need smallness estimates; indeed in the $3$-body setting we
  never need them. To explain this, we introduce further notation.

  \begin{definition}\label{def:microlocal-Mourre-loss}
We say that the effective microlocal Mourre loss at energy $\lambda$
at $(y_a,\xi_a)$, $\delta^a(y_a,\xi_a;\lambda)$, is given as follows: let
$\lambda'=\lambda-|\xi_a|^2$  with spectral
projector $E$ of $H^a(y_a,\xi_a)$ to $\lambda'$.
\begin{enumerate}
  \item
  If $E$ is finite rank (including if $\lambda'$ is not an eigenvalue) we set $\delta^a(y_a,\xi_a;\lambda)=0$,
  \item
$E$ is infinite rank and
\begin{equation}\label{eq:Mourre-strongly-adapted-compact}
E\Big[\frac{\xi_a}{|\xi_a|}\pa_{\xi_a},H^a\Big]E
\end{equation}
is compact on $L^2(X^a)$, we set $\delta^a(y_a,\xi_a;\lambda)=0$.
\item
  Otherwise $\delta^a(y_a,\xi_a;\lambda)$ is the inf of constants so that
\begin{equation}\label{eq:Mourre-strongly-adapted}
\Big\||\xi_a|E\Big[\frac{\xi_a}{|\xi_a|}\pa_{\xi_a},H^a\Big]E\Big\|\leq  \delta^a(y_a,\xi_a;\lambda).
\end{equation}
\end{enumerate}

We say that the effective Mourre loss at the level of a subsystem
$a\in\cA$ at energy $\lambda$ is
$$
\delta^a(\lambda)=\sup_{(y_a,\xi_a)}\delta^a(y_a,\xi_a;\lambda),
$$
understood to be $0$ if $\lambda-|\xi_a|^2\notin \sigma_{\pp}(H^a(y_a,\xi_a))$ for all
$(y_a,\xi_a)$.
    \end{definition}

If $a=0$ or a 2-body subsystem (so it has no non-trivial subsystems)
then $\delta^a(\lambda)=0$:
Supposing that $\hat V^1_a$ is negative order and symbolic (of 
order $\alpha<0$), then
$$
\psi(H^a) \hat V^1_a\psi(H^a)\in\Psisc^{-\infty,\alpha}(\overline{X^a})
$$
is compact on $L^2(X^a)$.

In general, if the
`non-product' part of $\hat V_a$, namely $\hat V^1_a(y_a)\xi_a$, satisfies
\begin{equation}
  |\hat V^1_a|\leq \tilde\delta_{1,a},
\end{equation}
then as
$$
\Big[\frac{\xi_a}{|\xi_a|}\pa_{\xi_a},\hat H_a(y_a,\xi_a)\Big]=\Big[\frac{\xi_a}{|\xi_a|}\pa_{\xi_a},\hat V^1_a(y_a)\xi_a\Big]=\frac{\xi_a}{|\xi_a|}\hat V^1_a(y_a),
$$
we immediately see that we can take any $\delta^a(y_a,\xi_a;\lambda)$ with
$$
\sqrt{\lambda-\lambda'}\tilde\delta_{1,a}\leq
\delta^a(y_a,\xi_a;\lambda),\ \lambda'=\lambda-|\xi_a|^2,
$$
and thus we may take
\begin{equation}\label{eq:delta^a-def-small-V}
\delta^a(\lambda)=\sup_{(y_a,\xi_a):\ \lambda-|\xi_a|^2\in\sigma_{\pp}(H^a)}|\xi_a|\tilde\delta_{1,a}.
\end{equation}
In the waveguide case, the condition of
\eqref{eq:Mourre-strongly-adapted} would amount to
$$
|\tau\pa_\tau\lambda'|\leq \delta^a(\lambda),
$$
but as all eigenspaces are finite dimensional, it is not needed.

We define $d(\lambda)$ as
above, with $\Lambda$ no longer necessarily discrete since the
subsystem eigenvalues can vary continuously with $(y_a,\xi_a)$. Recall that
a chain of collision planes, or collision planes at infinity, is a
totally ordered subset of them; we work with the latter, i.e.\
$\{C_a:\ a\in\cA\setminus\{1\}\}$. Let $\cI_a$ be the collection of chains whose elements all contain
$C_a$, and $\cI'_a$ the subset of $\cI_a$ consisting of chains {\em
  not} having $C_a$ itself as an element, so $\cI_0=\{\{C_0\}\}$,
$\cI'_0=\emptyset$. For a chain $S$ consisting of $C_{a_1}\subsetneq
C_{a_2}\subsetneq\ldots\subsetneq C_{a_k}$ let
  $$
\delta_S(\lambda)=\sum_{j=1}^k\delta^{a_j}(\lambda),
$$
and let
$$
\delta_a(\lambda)=\sup_{S\in\cI_a}\delta_S(\lambda),\ \delta'_a(\lambda)=\sup_{S\in\cI'_a}\delta_S(\lambda).
$$
Notice that in the three-body setting we simply have
$$
\delta'_1(\lambda)=\sup_{a\neq 0,1}\delta_{a}(\lambda)=0.
$$

We let
$$
M_a(\lambda)=\inf_{X^b\subset X^a} M^b(\lambda),\ M'_a(\lambda)=\inf _{X^b\subsetneq X^a} M^b(\lambda).
$$
Thus, if $X^c\subsetneq X^b$
\begin{equation}\label{eq:subsystem-Mourre-M-ordering}
M_b(\lambda)\leq M'_b(\lambda)\leq M_c(\lambda)\leq M'_c(\lambda),
\end{equation}
and
$$
M'_b(\lambda)=\inf_{X^c\subsetneq X^b}M_c(\lambda)
$$
  } \Bk

The general
Mourre-type estimate, in our non-product setting, is as follows:
\begin{theorem}\label{Mourre+}
For all $\epv>0$ there is $\delta>0$ such that for all smooth functions
$\phi\in\CI_c((\lambda-\delta,\lambda+\delta);[0,1]),$ we have
\begin{equation}\label{eq:Mourre-0}
\phi(H)i[A,H]\phi(H)\geq {\Bl(M'_1(\lambda)-\delta'_1(\lambda)-\epv)} \Bk \phi(H)^2+K',
\end{equation}
where $K'\in\PsiSc^{-\infty,1}(\Xb,\cC)$ is a compact operator.
\end{theorem}

{\Bl
\begin{remark} 
Notice that if $\lambda\notin\Lambda$ and all the
$\{\tilde\delta_{1,a}\}$ are sufficiently small, then
$M'_1(\lambda)-\delta'_1(\lambda)>0$. Given any
$\{\tilde\delta_{1,a}\}$, as  $C^a_{(y_a,\xi_a)}(\lambda)\lesssim
\sqrt{\lambda-\lambda'},$ the same estimate holds for sufficiently large
$\lambda>0$. 
\end{remark}
} \Bk 
\noindent
This theorem is proved in Appendix~\ref{MourreEstGenCase}.
\begin{remark}
 As we see below, in the wave-guide setting the Mourre estimate holds
 at $\lambda=\omega^2\kappa_0^2$ if
 \begin{equation}\label{eqn93.334}
  \sqrt{2} \|d\log\ellk\|_{L^{\infty}}<\omega\kappa_0;
 \end{equation}
it also holds at $\lambda=\omega^2\kappa_0^2,$ for any $\omega,$
if the contrast ratios satisfy~\eqref{cntrst_ratio_est}.
In these cases a simpler proof of the Mourre estimate is available,
which we outline in Appendix~\ref{MourreEstMaxWG}.
\end{remark}

\newpage
\appendix
\centerline{\LARGE{\bf Appendices}}
\section{Proof of the Mourre estimate for Maxwell
  Wave-guides}\label{MourreEstMaxWG} In this appendix we give an
outline of the proof of the Mourre estimate for the special case of a
1-form, Maxwell or vector Helmholtz wave-guide operator
\begin{equation}
  H=\kappa^2 \cD^{\epsilon,\mu}_{1}+\omega^2(\kappa_0^2-\kappa^2),
\end{equation}
 for energies $\lambda\in\IM,$ see Remark~\ref{rmk8.500}. The operator $\kappa^2
 \cD^{\epsilon,\mu}_{1}$ is defined in~\eqref{eqn32}. To simplify the
 proof, we make the additional assumptions:
 \begin{equation}\label{eqn136.224}
   \begin{split}
   &\text{For } j=1,\dots,m,\quad 0\notin\sigma_{pp}(L_{\pm\sqrt{\lambda},j}),\\
&\text{the families gap } \Gamma_0(\lambda,\delta_0)>0,\text{ for some }\delta_0>0.
\end{split}
 \end{equation}
 The second assumption holds if $L_{\pm\sqrt{\lambda},j}$ does not have a resonance at $0.$
  We let $A=\frac i2[x\cdot\pa_x+\pa_x\cdot x],$ {\Bl which agrees with the
  definition in~\eqref{eqn134.338} as $x$ here is the standard
  coordinate on Euclidean space.} \Bk In this setting the Mourre
  estimate is the following:
\begin{theorem} Fix a $c_0>0,$ and 
 assume that~\eqref{eqn136.224} holds.  Given
 $\lambda\in\IM(c_0),$ 
there are positive constants $C_0$ and $\delta<\delta_0,$ so
that if $\phi\in\cC^{\infty}_c(\bbR)$ has $\supp\phi\subset
(\lambda-\delta,\lambda+\delta),$ with $\phi(t)=1$,\hspace{2 pt} for $
t\in (\lambda-\frac{\delta}{2},\lambda+\frac{\delta}{2}),$ then
\begin{equation}\label{eqn128.222}
  i\langle\phi(H)[A,H]\phi(H)\psi,\psi\rangle\geq
  C_0\|\phi(H)\psi\|^2+
  \langle\psi, K\psi\rangle,
\end{equation}
where $K$ is a compact operator.
\end{theorem}
\begin{remark}
  The argument given below also works at $\lambda=\omega^2\kappa_0^2$
  if the contrast ratios satisfy the estimate
  in~\eqref{cntrst_ratio_est}. For simplicity we give the proof under
  assumption $\lambda\in\cI_M,$ but see Remark~\ref{rmk14.700}.
\end{remark}
\begin{proof}
  For the proof we follow the approach taken in~\cite[Chapter
    12]{VasyAsterisque}, which reduces the proof of the global
  estimate~\eqref{eqn128.222} to proving uniform estimates, at
  infinity in $\overline{\bbR}^3_{\WG}$ for the indicial operators,
  $i\hN_{a,0}([A,H])(\tau),$ where $a=\mf$ or $\ff_j,\,j=1,\dots,m.$
  The argument with $a=\mf$ is trivial because $H$ equals the
  Euclidean Laplacian in a neighborhood of $\mf,$ and therefore
  $i[A,H]=2H.$ We focus on the estimates along the front faces.

  Recall that in a neighborhood of $\ff_j$ the coefficients
  $\epsilon_j,\mu_j$ depend only variables orthogonal to the channel
  direction. We fix a value of $j$ and introduce coordinates so that
  the channel is $\{x_1>r,\, x_2=x_3=0\}.$ To simplify the notation we
  also drop the $j$-subscripts. As coordinates near the
  front face we use
  \begin{equation}
    \rho=\frac{1}{x_1},\,(Y_1,Y_2)=(x_2,x_3).
  \end{equation}
We observe that
\begin{equation}\label{eqn158.502}
  \begin{split}
    i\hN_{\ff,0}([A,H])(\tau)&=[\hN_{\ff,0}(H)(\tau),Y\cdot\pa_{Y}]+2\tau^2+
    \tau\tQ^{\epsilon,\mu}_{10}\\
    &=[\hN_{\ff,0}(H)(\tau),Y\cdot\pa_{Y}]+\tau\pa_{\tau}\hN_{\ff,0}(H)(\tau).
    \end{split}
  \end{equation}
  Because of the term $\tau\tQ^{\epsilon,\mu}_{10},$ this differs from
  the usual case.  Another important difference is the fact that the
  point spectrum of $\hN_{\ff,0}(H)(\tau)$ depends on $\tau.$ The
  family gap, defined in~\eqref{eqn119.224}, satisfies
  \begin{equation}
  \Gamma_0=  \Gamma(\lambda,\delta_0)>0;
  \end{equation}
  we assume that $\delta<\delta_0.$

  We let $\hH_{\ff}(\tau)=\hN_{\ff,0}(H)(\tau),$ which is given
  in~\eqref{eqn48.61}:
  \begin{equation}
   \hH_{\ff}(\tau)= \Delta^{(2)}\otimes\Id_{\Lambda^1\bbR^3}+\tQ_1^{\epsilon,\mu}(\tau)+
    \tQ_0^{\epsilon,\mu}+\omega^2(\kappa_0^2-\kappa^2),
  \end{equation}
  where $\Delta^{(2)}=-[\pa_{Y_1}^2+\pa_{Y_2}^2].$ Using the results
  in Appendix C of~\cite{VasyAsterisque} it suffices to show that
  there is a $C>0$ for which there exists a $\delta>0$ so that, if
  $\supp\phi\subset (\lambda-\delta,\lambda+\delta),$ then
  \begin{multline}\label{eqn144.224}
    \phi(\hH_{\ff}(\tau)+\tau^2)
        \left([\hH_{\ff}(\tau),Y\cdot\pa_Y]+2\tau^2+
          \tau\tQ^{\epsilon,\mu}_{10}\right)
        \phi(\hH_{\ff}(\tau)+\tau^2)\\\geq
        C \phi^2(\hH_{\ff}(\tau)+\tau^2)
  \end{multline}
  holds for all $\tau.$ The operator $\phi(\hH_{\ff}(\tau)+\tau^2)$
  depends only on the spectrum of $\hH_{\ff}(\tau)$ in the interval
  $(\lambda-\tau^2-\delta, \lambda-\tau^2+\delta).$ It therefore
  follows that from Proposition~\ref{prop2.224} that, for $\delta<\delta_0,$
  there is a $\tau'_{\max},$ so that
  \begin{equation}
      |\tau|>\tau'_{\max}\Rightarrow \phi(\hH_{\ff}(\tau)+\tau^2)=0,
  \end{equation}
and therefore, for $|\tau|>\tau'_{\max},$~\eqref{eqn144.224} is
immediate. Because the family gap, $\Gamma_0,$ is non-zero, if
$2\delta<\Gamma_0,$ then $\supp\phi(\hH_{\ff}(\tau)+\tau^2)$ lies
either in the continuous spectrum of $\hH_{\ff}(\tau),$ or in the
point spectrum, but not both. The arguments in the 2 cases are
somewhat different. We first consider the continuous spectrum; in this
case $\tau^2<\lambda+\delta.$

Given $\epsilon>0,$ with $\phi(\hH_{\ff}(\tau)+\tau^2)$ supported in
the continuous spectrum, the fact that there is a $\delta>0,$ so that
 \begin{multline}\label{eqn146.224}
    \phi(\hH_{\ff}(\tau)+\tau^2)
        \left([\hH_{\ff}(\tau),Y\cdot\pa_Y]+2\tau^2\right)
        \phi(\hH_{\ff}(\tau)+\tau^2)\\>
        2(\lambda-\frac{\epsilon}{2}) \phi(\hH_{\ff}(\tau)+\tau^2)^2
 \end{multline}
 is the classical 2-body estimate
 from~\cite{FroeseHerbst1,FroeseHerbst2}. Only the term
 \begin{equation}\label{eqn149.334}
   \tau q_0(\tau)= \tau\phi(\hH_{\ff}(\tau)+\tau^2)
        \tQ^{\epsilon,\mu}_{10}
        \phi(\hH_{\ff}(\tau)+\tau^2)
  \end{equation}
requires further consideration. Because the support of
$\phi(\hH_{\ff}(\tau)+\tau^2)$ lies in the continuous spectrum, we
claim that as $\delta\to 0,$ the norm of this operator tends to
zero. We prove this only in the case of the free space Laplacian.

Let $Q'=\tQ^{\epsilon,\mu}_{10};$ this is a compactly supported,
smooth matrix valued function, and let $\hQ'$ denote its Fourier
transform. For simplicity we replace $\phi(t)$ with
$\chi^{\delta}(t)=\chi_{[\lambda-\delta,\lambda+\delta]}(t).$ The
projection operator then takes the form
\begin{equation}
 \chi^{\delta}(\Delta+\tau^2)\psi=\cF^{-1}\left[\chi_{D_{\tau}}(\xi)\hpsi(\xi)\right],
\end{equation}
where $D_{\tau}=B_{\sqrt{(\lambda+\delta-\tau^2)_+}}(0)\setminus
B_{\sqrt{(\lambda-\delta-\tau^2)_+}}(0),$ and $t_+=\max\{t,0\}.$ Using
the convolution theorem we see that
\begin{equation}
  \cF(q_0(\tau)\psi)(\xi)=\chi_{D_{\tau}}(\xi)\int_{\bbR^2}\hQ'(\eta)\chi_{D_{\tau}}(\xi-\eta)\hpsi(\xi-\eta)d\eta.
\end{equation}

The Cauchy-Schwarz inequality shows that
\begin{equation}
  \left|\langle\chi^{\delta}(\Delta+\tau^2)\psi,q_0(\tau)\psi\rangle\right|\leq
  \|\chi^{\delta}(\Delta+\tau^2)\psi\|\|q_0(\tau)\psi\|.
\end{equation}
The Parseval theorem and the Cauchy-Schwarz inequality imply that
\begin{equation}
  \begin{split}
    \|q_0(\tau)\psi\|^2\leq &\|Q'\|_{L^1}^2\int
    \chi_{D_{\tau}}(\xi)
    \left[\int_{\bbR^2}\chi_{D_{\tau}}(\xi-\eta)d\eta\right]\left[\int
     \chi_{D_{\tau}}(\xi-\eta) |\hpsi(\xi-\eta)|^2d\eta\right] d\xi\\
  \leq &
  |D_{\tau}|^2\|Q'\|^2_{L^1}\|\chi^{\delta}(\Delta+\tau^2)\psi\|^2.
  \end{split}
\end{equation}
For bounded $\tau,$ the area  $|D_{\tau}|<m\delta,$ for a fixed
constant, $m,$ and therefore
\begin{equation}
  \left|\langle\chi^{\delta}(\Delta+\tau^2)\psi,\tau q_0(\tau)\psi\rangle\right|\leq
  m\delta|\tau||Q'\|_{L^1} \|\chi^{\delta}(\Delta+\tau^2)\psi\|^2.
\end{equation}
By choosing $\delta$ so that
$$m\delta\sqrt{\lambda+\delta}\|Q'\|_{L^1}<\epsilon,$$ this term can
be absorbed into the right hand side of~\eqref{eqn146.224}, proving
the existence of $\delta$ in these cases.

Recalling that $\delta<\delta_0,$ we are left to consider $|\tau|\in
(\sqrt{\lambda+\delta},\tau'_{\max}),$ with
$(\lambda-\delta-\tau^2,\lambda+\delta-\tau^2)$ intersecting only the
point spectrum of $\hH_{\ff}(\tau).$ Recall that all the point spectrum of
these operators is negative and is locally parameterized by analytic
functions $\{(\lambda_l(\tau),v_l(\tau)):l=1,\dots, N_{\tau}\},$ where
$N_{\tau}<J,$ for $|\tau|\leq |\tau'_{\max}|.$ The maximum number of
eigenvalues in an interval of length $2\delta$ is denoted by
$N^{\delta}_{\max}.$

As before we simplify the
discussion by working with the projections defined by
$\chi^{\delta}(\hH_{\ff}(\tau)+\tau^2).$
A new difficulty that arises in the present case is that the point
spectrum of  $\hH_{\ff}(\tau)$ depends on $\tau.$ If we make the
assumption that the eigenvalues of $\hH_{\ff}(\tau)$ never cross, then
we can choose $\delta>0$ so that each interval
$(\lambda-\delta-\tau^2,\lambda+\delta-\tau^2)$ contains at most one
eigenvalue. It then follows from the virial theorem that
\begin{equation}\label{eqn170.502}
   \chi^{\delta}(\hH_{\ff}(\tau)+\tau^2)
        [\hH_{\ff}(\tau),Y\cdot\pa_Y]
        \chi^{\delta}(\hH_{\ff}(\tau)+\tau^2)=0,
\end{equation}
therefore we are left to show that there is a positive constant $c$ so
that
\begin{equation}\label{eqn156.334}
  2\tau^2\chi^{\delta}(\hH_{\ff}(\tau)+\tau^2)+\tau
  q_0(\tau)>c\chi^{\delta}(\hH_{\ff}(\tau)+\tau^2),
\end{equation}
see~\eqref{eqn149.334}. As $\tau^2>\lambda\in\IM(c_0)$ it follows that 
there is a positive number $c_0$ so that
\begin{equation}\label{eqn202.700}
 \tau(2\tau+\dot{\lambda}(\tau))>c_0.
\end{equation}

We now need to consider what happens if the interval
$(\lambda-\tau^2-\delta,\lambda-\tau^2+\delta)$ contains several
different eigenvalues. We first estimate of the contribution of the
$\tQ^{\epsilon,\mu}_{10}$-term. As above
$\{(\lambda_j(\tau),v_j(\tau)):j=1,\dots,n\}$ are the eigenpairs with
$\lambda_j(\tau)\in(\tau^2-\lambda-\delta,\tau^2-\lambda+\delta);$ we
know that $n<N^{\delta}_{\max}.$ Since we are only considering a
single $\tau$ at a time, the functions $\{v_j(\tau)\}$ are an
orthonormal set, and therefore
\begin{equation}\label{eqn203.700}
  \langle \pa_{\tau}(\hH_{\ff}(\tau)+\tau^2)v_j,v_l\rangle=
  (2\tau+\dot{\lambda_j}(\tau))\delta_{jl}+(\lambda_j-\lambda_l)\langle\dot{v}_j,v_l\rangle.
\end{equation}
From which it follows that there is a constant, $M,$ so that
\begin{equation}
  \chi^{\delta}(\hH_{\ff}(\tau)+\tau^2)
 \tau
 \pa_{\tau}(\hH_{\ff}(\tau)+\tau^2)\chi^{\delta}(\hH_{\ff}(\tau)+\tau^2)>
    \tau[ c_0-\delta M]\chi^{\delta}(\hH_{\ff}(\tau)+\tau^2).
\end{equation}
Decreasing $\delta$ if needed, gives the desired lower bound.

We cannot use the virial theorem to estimate the left hand side
of~\eqref{eqn170.502}, but there is an approximate virial theorem that we
can use instead. We need to bound the
operator
\begin{multline}
  F_{\delta,\tau}=\chi^{\delta}(\hH_{\ff}(\tau)+\tau^2)
        [\hH_{\ff}(\tau),Y\cdot\pa_Y]
        \chi^{\delta}(\hH_{\ff}(\tau)+\tau^2)=\\
        \chi^{\delta}(\hH_{\ff}(\tau)+\tau^2)
        [\hH_{\ff}(\tau)-\lambda_0,Y\cdot\pa_Y]
        \chi^{\delta}(\hH_{\ff}(\tau)+\tau^2).
\end{multline}
Using the eigenpairs from above, the projection operator is given by
\begin{equation}
  \chi^{\delta}(\hH_{\ff}(\tau)+\tau^2)\psi=\sum_{l=1}^n\langle
  \psi,v_l\rangle  v_l,
\end{equation}
and
\begin{equation}
 Y\cdot\pa_Y (\hH_{\ff}(\tau)-\lambda_0)
 (\chi^{\delta}(\hH_{\ff}(\tau)+\tau^2)\psi=
 \sum_{l=1}^n(\lambda_l-\lambda_0)\langle
  \psi,v_l\rangle Y\cdot\pa_{Y} v_l.
\end{equation}

A computation shows that
\begin{equation}
  \langle\chi^{\delta}(\hH_{\ff}(\tau)+\tau^2)\psi,
  F_{\delta,\tau}\psi\rangle=
  -2\sum_{k,l=1}^{n}(\lambda_l-\lambda_0)\langle\psi,v_l\rangle
  \langle\psi,v_k\rangle\langle Y\cdot\pa_Y v_k,v_l\rangle
\end{equation}
Since all of the eigenvalues are less than $-\Gamma_0,$ the
exponential decay estimates for eigenfunctions show that there is
constant $M$ independent of $|\tau|\in [\lambda,\tau_{\max}]$ and the
particular eigenfunction so that
\begin{equation}\label{eqn165.334}
  \|Y\cdot\pa_Y v_k\|_{L^2}\leq M.
\end{equation}
The Cauchy-Schwarz inequality then shows that, with another constant, $M',$
\begin{equation}\label{eqn166.334} 
  |\langle\chi^{\delta}(\hH_{\ff}(\tau)+\tau^2)\psi,
  F_{\delta,\tau}\psi\rangle|\leq M'\delta
  \|\chi^{\delta}(\hH_{\ff}(\tau)+\tau^2)\psi\|_{L^2}^2.
\end{equation}
Decreasing $\delta$ further completes the proof of the theorem.
\end{proof}
\begin{remark}\label{rmk14.700}
  If $\lambda=\omega^2\kappa_0^2$ and the estimates
  in~\eqref{cntrst_ratio_est} hold, then, for small enough $\delta>0,$
  the estimates in~\eqref{eqn202.700} and~\eqref{eqn203.700} continue
  to hold for some $c_0>0.$ This shows that the Mourre estimate holds
  in this case as well.  
\end{remark}
\begin{remark}
  The assumption that $\Gamma_0>0$ allows us to avoid analyzing what
  happens if $\supp\phi(\hH_{\ff}(\tau)+\tau^2)$ contains both
  continuous and point spectrum. It is easy to see that the
  approximate virial theorem,~\eqref{eqn166.334}, extends to the case
  $\Gamma_0=0.$ In particular, the estimate~\eqref{eqn165.334} does
  not require $\Gamma_0>0.$ Using the following lemma, we can extend
  our proof of this estimate to the case that $\Gamma_0=0.$ We leave the details of
  this extension to the reader; a proof of the Mourre estimate in much
  greater generality is provided in
  Appendix~\ref{MourreEstGenCase}.
  \begin{lemma}\label{lem1}
    Let $q$ be a  first order operator whose coefficients are
    supported in $B_R(0)$. There is a constant $C>0,$ so that if
    $v\in L^2(\bbR^2)$ is a solution to
    \begin{equation}\label{eqn168.334}
      (-\Delta+q)v=-\lambda v,
    \end{equation}
    for a $\lambda>0,$ then
    \begin{equation}
      \|r\pa_r v\|^2_{L^2(B_1^c(0))}\leq 4\|v\|^2_{L^2}+C\|\nabla v\|^2_{L^2(B_R(0))}.
    \end{equation}
    The constant, $C,$ is independent of $\lambda.$
  \end{lemma}
  \begin{remark}
    These estimates hold very generally for eigenfunctions of $N$-body
    Schr\"odinger operators, though the proof in the general case is
    considerably more technical.
  \end{remark}
  \begin{proof}
  As $q$ is compactly
 supported, the fact that $\lambda>0$ implies that $v$ and all of its
 derivatives are exponentially decaying. Let
 $\chi\in\cC^{\infty}(\bbR)$ be a monotone increasing function that is
 zero for $r<1/2,$ and $1$ for $r\geq 1.$ Integrating by parts we see
 that
 \begin{equation}
  \int_0^{2\pi} \int_{0}^{\infty}|\chi^{\frac
    12}(r)r\pa_rv(r,\theta)|^2rdrd\theta=
  -\int_0^{2\pi} \int_{0}^{\infty}\bar{v}\pa_r[\chi(r)r^3\pa_rv(r,\theta)]drd\theta.
 \end{equation}
 Expanding the derivative, using the equation~\eqref{eqn168.334} and integrating by parts in
 $\theta$ we see that
 \begin{multline}
  \int_0^{2\pi} \int_{0}^{\infty}|\chi^{\frac
    12}(r)r\pa_rv(r,\theta)|^2rdrd\theta=-\int_{0}^{2\pi}\int_0^{\infty}
  \chi(r)[|v_{\theta}|^2+\lambda r^3|v|^2]drd\theta\\
  -\int_{0}^{2\pi}\int_0^{\infty}[2\chi(r)r^2\bar{v}\pa_rv+\chi_r
    r^3\bar{v}\pa_rv
    +\chi(r)r^3\bar{v}q v]drd\theta.
 \end{multline}
 The first integral on the right is non-positive, and the last two
 terms in the next integral are compactly supported. Applying the
 Cauchy-Schwarz and arithmetic-geometric mean inequality to the
 remaining term gives the conclusion of the lemma.
  \end{proof}
  
\end{remark}

\section{Proof of the Generalized Mourre Estimate}\label{MourreEstGenCase}
In this appendix we prove the Mourre estimate in the setting of
general $N$-body problems, \eqref{eq:Mourre-0}, relying on a now
standard iterative argument for the indicial operators that originated
in this form in \cite{FroeseHerbst2}; see also
\cite{Derezinski-Gerard:Scattering}. We now allow the indicial operators
of the perturbation of the Laplacian to have terms, linear in the
frequency variables, which satisfy an overall smallness condition. The
argument and the notation in this section follows that
in~\cite[Appendix]{VasyExpDecay}.

To prove \eqref{eq:Mourre-0}, one only needs to show that for all
$b\in\cA$, $b\neq 1$, the corresponding indicial operators satisfy the
corresponding inequality, i.e.\ that
\begin{equation}\label{eq:Mourre-1a}
\phi(\hat H_b)i\widehat{[A,H]}_b\phi(\hat H_b)
\geq {\Bl (M'_1(\lambda)-\delta'_1(\lambda)-\epv)} \Bk \phi(\hat H_b)^2,
\end{equation}
without a compact error term. This means that the operators on the
two sides, which are families of operators on $X^b$, depending on
$(y_b,\xi_b)\in\sct_{C_b}\Xb$, satisfy the inequality for all
$(y_b,\xi_b)\in\sct_{C_b}\Xb$.  We actually prove that
\begin{equation}\label{eq:Mourre-1}
\phi(\hat H_b)i\widehat{[A,H]}_b\phi(\hat H_b)
\geq {\Bl (M_b(\lambda) -\delta_b(\lambda)-\epv)} \Bk \phi(\hat H_b)^2;
\end{equation}
since {\Bl $M_b(\lambda)-\delta_b(\lambda)\geq
  M_1'(\lambda)-\delta'_1(\lambda)$} \Bk for all $b\neq 1$, the
desired conclusion follows.  It is convenient to assume that $\phi$ is
identically $1$ near $\lambda$; if \eqref{eq:Mourre-1} holds for such
$\phi$, it holds for any $\phi_0$ with slightly smaller support, as
follows by multiplication by $\phi_0(\hat H_b)$ from the left and
right.

Note that for $b=0$ the estimate \eqref{eq:Mourre-1}
certainly holds: it comes from the Poisson bracket formula
in the scattering calculus, or from a direct computation yielding
$i\widehat{[A,H]}_0=2\Delta_{g_0}$. Hence, if the localizing factor
$\phi(\hat H_0)=\phi(|\xi|^2)$ is supported in
$(\lambda-\delta,\lambda+\delta)$ and $\lambda>0$,
then \eqref{eq:Mourre-1} holds even
with {\Bl $M(\lambda)-\epv$ replaced with $2(\lambda-\delta)$. } \Bk Note that
$\lambda= M_0(\lambda)$, if $\lambda\geq 0$, since $0$ is a threshold of $H$.
On the other hand, if $\lambda<0$, both sides of \eqref{eq:Mourre-1}
vanish for $\phi$ supported near $\lambda$, so the inequality holds
trivially.

In general, we may assume inductively\footnote{In \cite{VasyExpDecay},
  the inclusions were accidentally reversed in the appendix in the
  corresponding paragraph.} that at all clusters $c$ with
$C_c\supsetneq C_b$, i.e.\ $X^b\supsetneq X^c$, \eqref{eq:Mourre-1} has
been proved with $\phi$ replaced by a cutoff $\phit$ and $\epv$ replaced
by $\epv'$, i.e.\ we may assume that for all $\epv'>0$ there exists $\delta'>0$
such that for all $c$ with $C_c\supsetneq C_b$, and for all
$\phit\in\Cinf_c(\RR;[0,1])$ supported in $(\lambda-\delta',\lambda+\delta')$,
\begin{equation}\label{eq:Mourre-1-0}
\phit(\hat H_c)i\widehat{[A,H]}_c\phit(\hat H_c)
\geq {\Bl (M_c(\lambda)-\delta_c(\lambda)-\epv')} \Bk \phit(\hat H_c)^2.
\end{equation}
The operators appearing in~\eqref{eq:Mourre-1-0} are  the
indicial operators of $\phit(\hat H_b)i\widehat{[A,H]}_b\phit(\hat
H_b)$, so, as discussed in
\cite[Proposition~8.2]{Vasy:Propagation-Many}, in view of {\Bl
  $M'_{b}(\lambda)\leq M_c(\lambda)$ and $\delta_b'(\lambda)\geq
  \delta_c(\lambda)$} \Bk for all $c$ as above, \eqref{eq:Mourre-1-0}
implies that
\begin{equation}\label{eq:Mourre-8}
\phit(\hat H_b)i\widehat{[A,H]}_b\phit(\hat H_b)
\geq {\Bl (M'_b(\lambda)-\delta'_b(\lambda)-\epv')} \Bk\phit(\hat H_b)^2+K_b,
\ K_b\in\PsiSc^{-\infty,1}(X^b,\cC^b).
\end{equation}
Recall that this implication relies on a square root construction in
the many-body calculus, which is particularly simple in this case. To
complete the induction step we need to show that, possibly reducing
$\delta,$ this estimate holds without the compact error term.

To that end we first multiply \eqref{eq:Mourre-8}
through by $\phi(H_b)$ from both the left and the
right. We use coordinates $(w_b,w^b)$ on $X_b\oplus X^b,$ with
$(\xi_b,\xi^b)$  the dual coordinates.
Note that $\hat H_b=\hat H_b(y_b,\xi_b)=|\xi_b|^2+H^b$, so
if $\lambda-|\xi_b|^2$ is not an eigenvalue of $H^b$,
then as $\supp\phi\to\{\lambda\}$,
$\phi(H^b+|\xi_b|^2)\to 0$ strongly. Since $K_b$ is compact,
$\phi(H^b+|\xi_b|^2)K_b\to 0$ in norm; in particular it
can be made to have norm smaller than $\epv'-\epv>0$. After multiplication
from both sides by $\phi_1(\hat H_b)$, with $\phi_1$ having even smaller
support, \eqref{eq:Mourre-1}
follows with
$$
{\Bl M'_b(\lambda)\geq M_b(\lambda),\ \delta'_b(\lambda)\leq\delta_b(\lambda)} 
$$
\Bk 
in place of
{\Bl $M_b(\lambda), \delta_b(\lambda)$} \Bk and with $\phi_1$ in place of $\phi$,
with the size of $\supp\phi_1$ a priori depending on $(y_b,\xi_b)$.
The operators
$$
i\phi_1(\hat H_b)\widehat{[A,H]}_b\phi_1(\hat H_b)
$$
are continuous functions of $(y_b,\xi_b),$ with values in bounded operators on $L^2(X^b)$,
so if \eqref{eq:Mourre-1} holds at one value
of $\xi_b$, then it holds nearby {\Bl if $\epv>0$ is increased
  slightly: a priori one obtains a new term $-\epv'' \Id$ in the lower
  bound, but multiplication by a spectral cutoff restores the desired
  form.} \Bk Moreover, for large $|\xi_b|$ both sides
vanish as $\hat H_b=H^b+|\xi_b|^2$, with $H^b$ bounded below, so by a
covering argument
the estimate is in fact uniform
if we slightly increase $\epv>0$. {\em This proves \eqref{eq:Mourre-1} if
  $\lambda-|\xi_b|^2$ is not an eigenvalue of $H^b$.}

In general, if $E=E(y_b,\xi_b)$
denotes the projection on the $L^2$ eigenspace of $H^b$ at $\lambda'=\lambda
-|\xi_b|^2$ (so $\lambda'\leq\lambda$!), the argument just sketched {\em works if we can replace}
$\phi(\hat H_b)$ by $(\Id-E)\phi(\hat H_b)$, in particular, it suffices to
show that
\begin{equation}\label{eq:Mourre-8p}
i\phit(\hat H_b)\widehat{[A,H]}_b\phit(\hat H_b)
\geq {\Bl (M_b(\lambda)-\delta_b(\lambda)-4\epv')} \Bk \phit(\hat H_b)^2+(\Id-E)K'_b(\Id-E),
\end{equation}
$K'_b$ compact on $L^2(X^b)$, for then one can repeat the above
argument to conclude \eqref{eq:Mourre-1}.

To show \eqref{eq:Mourre-8p}, we follow an argument due to B.\ Simon
(as explained
in the paper \cite{FroeseHerbst2} of Froese and Herbst). The key point is
to replace $E$ by a finite rank orthogonal projection $F$, which
will later ensure that an error term is finite rank, hence compact. Thus,
by the compactness of $K_b$ (from \eqref{eq:Mourre-8}),
there is a
finite rank orthogonal projection $F$ with $\Range F\subset\Range E$
(so $F$ commutes with $\hat H_b$)
such that
\begin{equation}\label{eq:F-proj-def}
\|(\Id-E)K_b(\Id-E)-(\Id-F)K_b(\Id-F)\|<\epv'.
\end{equation}
{\Bl If $E$ is finite rank, we take $F=E$; in this case the 
  left hand side of \eqref{eq:F-proj-def} vanishes.} \Bk 
{\Bl Furthermore, if for all $|\xi_b|^2=\lambda-\lambda'$
  $$
  T_1=E|\xi_b|\Big[\frac{\xi_b}{|\xi_b|}\pa_{\xi_b},H^b\Big]E
  $$
  is compact, then we also assume (by possibly updating $F$) that
  \begin{equation}\label{eq:F-proj-def-2}
\|(E-F)T_1\|,\|T_1(E-F)\|< \epv'.
\end{equation}
} \Bk 
Multiplying \eqref{eq:Mourre-8} through by $(\Id-F)\phi(\hat H_b)
=\phi(\hat H_b)-F$ from
left and right and using \eqref{eq:F-proj-def} gives
\begin{equation}\begin{split}\label{eq:Mourre-16}
&i(\phi(\hat H_b)-F)\widehat{[A,H]}_b(\phi(\hat H_b)-F)\\
&\qquad\geq
{\Bl (M'_b(\lambda)-\delta'_b(\lambda)-2\epv')} \Bk (\phi(\hat H_b)^2-F)\\
&\qquad\qquad+\phi(\hat H_b)(\Id-E)K_b
(\Id-E)\phi(\hat H_b).
\end{split}\end{equation}
Moving the terms involving $F$ from the left hand side to the right hand
side yields that
\begin{equation}\begin{split}\label{eq:Mourre-16p}
i\phi(\hat H_b)\widehat{[A,H]}_b\phi(\hat H_b)
&\geq {\Bl  (M'_b(\lambda)-\delta'_b(\lambda)-2\epv')} \Bk (\phi(\hat H_b)^2-F)\\
&\qquad+F\phi(\hat H_b)i\widehat{[A,H]}_b(\phi(\hat H_b)-F)\\
&\qquad +(\phi(\hat H_b)-F)i\widehat{[A,H]}_b\phi(\hat H_b)F
+FTF\\
&\qquad+\phi(\hat H_b)(\Id-E)K_b
(\Id-E)\phi(\hat H_b).
\end{split}\end{equation}
where $T=Ei\widehat{[A,H]}_bE$.
But
\begin{equation*}\begin{split}
&F\phi(\hat H_b)i\widehat{[A,H]}_b(\phi(\hat H_b)-F)=FC+
F\phi(\hat H_b)i\widehat{[A,H]}_b(E-F),\\
&C=F\phi(\hat H_b)i\widehat{[A,H]}_b\phi(\hat H_b)(\Id-E),
\end{split}\end{equation*}
and for any $\epv_1>0$,
\begin{equation*}
FC+C^*F\geq-\epv_1 F-\epv_1^{-1} C^*C,
\end{equation*}
so \eqref{eq:Mourre-16p} yields that
\begin{equation}\begin{split}\label{eq:Mourre-32}
i\phi(\hat H_b)\widehat{[A,H]}_b\phi(\hat H_b)
\geq &{\Bl (M'_b(\lambda) -\delta'_b(\lambda)-2\epv')\phi(\hat H_b)^2} \Bk \\
&-{\Bl (M'_b(\lambda) -\delta'_b(\lambda)-2\epv'+\epv_1)F} \Bk \\
&+FT(E-F)+(E-F)TF+FTF\\
&
+\phi(\hat H_b)(\Id-E)(K_b+K''_b)
(\Id-E)\phi(\hat H_b),
\end{split}\end{equation}
with $K''_b=-\epv_1^{-1}
\widehat{[A,H]}_b F\phi(\hat H_b)^2 \widehat{[A,H]}_b$,
which is compact due to the appearance of $F$.

We next show that
\begin{equation}\label{eq:Mourre-23}
-{\Bl (M_b(\lambda)-2\epv'+\epv_1)F} \Bk 
+FT(E-F)+(E-F)TF+FTF \geq {\Bl -2\epv' E-\delta^b(\lambda)E}, \Bk 
\end{equation}
where $\lambda'=\lambda-|\xi_b|^2$ is the eigenvalue of $H^b$ to which
$E$ projects. {\Bl We remark that it is at this point that the Mourre
constant $M'_b(\lambda)$ is replaced by the (in general) smaller
$M_b(\lambda)$.} \Bk 

{\Bl Now\footnote{Note the missing $-$ signs in the second expression
    for $A$
    in the proof given in the reference \cite{VasyExpDecay}.}, $A=sD_s+A'=-w_b\cdot D_{w_b}-w^b\cdot D_{w^b}+A'$,
$A'\in \DiffSc^1(\Xb,\cC)$ (thus has one order of decay relative
  to $A$),} \Bk hence
\begin{equation*}\begin{split}
T=Ei\widehat{[A,H]}_bE=&{\Bl -} \Bk Ei\widehat{[w_bD_{w_b},H]}_bE{\Bl
  -} \Bk 
Ei\widehat{[w^bD_{w^b},H]}_bE\\
&+E\widehat{A'}_b(\hat H_b-\lambda)E
-E(\hat H_b-\lambda)\widehat{A'}_b E={\Bl -} \Bk Ei\widehat{[w_b D_{w_b},H]}_bE
\end{split}\end{equation*}
since by the virial theorem $iE[w^bD_{w^b},H^b]E=0$ (see the remark
below).
{\Bl Thus we see  that
$$
T=-Ei\widehat{[w_b
  D_{w_b},H]}_bE=E|\xi_b|\Big[\frac{\xi_b}{|\xi_b|}\pa_{\xi_b},\hat
H_b\Big]E\geq M^b(\lambda) E
$$
by the definition of $M^b(\lambda)$, presented
in Definition~\ref{IdMourreConst}.

In the finite rank case, when $E=F$,
$$
FT(E-F)+(E-F)TF+FTF=FTF
$$
and $FTF\geq M^b(\lambda) F$ proves \eqref{eq:Mourre-23} if we
choose $\epv_1<2\epv'$, even with $0$ on the right hand side, since
$M_b(\lambda)\leq M^b(\lambda)$.

In the infinite rank case, if for all $|\xi_b|^2=\lambda-\lambda'$
  $$
  T_1=E|\xi_b|\Big[\frac{\xi_b}{|\xi_b|}\pa_{\xi_b},H^b\Big]E
  $$
  is compact,
  $\|FT(E-F)+(E-F)TF\|<2\epv'$, and
  $$
  FT(E-F)+(E-F)TF=E(FT(E-F)+(E-F)TF)E,
  $$
  so $FT(E-F)+(E-F)TF\geq
  -2\epv' E$. Hence we can finish the argument for \eqref{eq:Mourre-23}  as in the finite rank
  case, with now $-2\epv' E$ needed on the right hand side, but with $\delta^b=0$.} \Bk

{\Bl In the general infinite rank case} \Bk we have
$$
T=T_0+T_1={\Bl -} \Bk Ei\widehat{[w_b D_{w_b},\Delta_{X_b}]}_bE {\Bl -} \Bk Ei\widehat{[w_b
  D_{w_b},H^b]}_bE
$$
where
\begin{equation*}
T_0={\Bl -} \Bk iE\widehat{[w_b D_{w_b},\Delta_{X_b}]}E=2|\xi_b|^2 E=2(\lambda-\lambda')E;
\end{equation*}
here $\lambda'=\lambda-|\xi_b|^2$ is the eigenvalue of $H^b$ to which
$E$ projects.
{\Bl By the definition of $M^b(\lambda)$ we also have
$$
ETE=2|\xi_b|^2
E+|\xi_b|E\Big[\frac{\xi_b}{|\xi_b|}\pa_{\xi_b},H^b\Big]E\geq
M^b(\lambda) E,
$$
and
$$
\Big\||\xi_b|E\Big[\frac{\xi_b}{|\xi_b|}\pa_{\xi_b},H^b\Big]E\Big\|\leq\delta^b(\lambda).
$$
With
$$
T_1=iE[\xi_bD_{\xi_b},\hat H^b]E=E[\xi_b\pa_{\xi_b}, \hat H^b]E,
  $$
  we thus have 
$$
\|T_1\|\leq \delta^b(\lambda).
$$
Thus,
\begin{equation*}\begin{aligned}
    &FT_1(E-F)+(E-F)T_1F\geq -
  \delta^b(\lambda) E
\end{aligned}\end{equation*}
since
\begin{equation*}\begin{aligned}
|\langle FT_1(E-F)\psi,\psi\rangle|&=|\langle
T_1(E-F)\psi,F\psi\rangle|\leq\|T_1\|\|(E-F)\psi\|\|F\psi\|\\
&\leq\|T_1\|\frac{1}{2}(\|(E-F)\psi\|^2+\|F\psi\|^2).
\end{aligned}\end{equation*}
Hence
\begin{equation*}\begin{aligned}
&\langle FT_1(E-F)\psi,\psi\rangle+\langle (E-F)T_1
F\psi,\psi\rangle\\
&\qquad\geq -\|T_1\|(\langle (E-F)\psi,\psi\rangle+\langle
F\psi,\psi\rangle)
=-\|T_1\|\langle E\psi,\psi\rangle.
\end{aligned}\end{equation*}
} \Bk 
We also have $FT_0(E-F)=0=(E-F)T_0F$ since $T_0$ is a scalar multiple
of $E$.
As $M_b(\lambda)\leq M^b(\lambda)$, if we choose $\epv_1<2\epv'$, then it follows that
\begin{equation*}\begin{aligned}
-{\Bl (M_b(\lambda)}&{\Bl -2\epv'+\epv_1)F} \Bk 
+FT(E-F)+(E-F)TF+FTF\\
&\geq {\Bl (-M_b(\lambda)+2\epv'-\epv_1)F +M^b(\lambda)F-
  \delta^b(\lambda) E\geq -\delta^b(\lambda) E};
 \Bk 
\end{aligned}\end{equation*}
 proving \eqref{eq:Mourre-23}.

Thus, by \eqref{eq:Mourre-32},
\begin{equation}\begin{split}\label{eq:Mourre-32p}
i\phi(\hat H_b)\widehat{[A,H]}_b\phi(\hat H_b)
&\geq {\Bl (M'_b(\lambda)-\delta'_b(\lambda)-2\epv') } \Bk \phi(\hat
H_b)^2\\
&\qquad{\Bl +(M_b(\lambda)-M'_b(\lambda)+\delta'_b(\lambda))F-2\epv' E-\delta^b(\lambda)E} \Bk \\
&\qquad+\phi(\hat H_b)(\Id-E)(K_b+K''_b)
(\Id-E)\phi(\hat H_b)\\
&{\Bl \geq (M'_b(\lambda)-\delta'_b(\lambda)-4\epv'-
  \delta^b(\lambda)) (\phi(\hat
  H_b)^2-F)} \Bk \\
&\qquad{\Bl +(M_b(\lambda)-4\epv'-
  \delta^b(\lambda))F}\\
&\qquad+\phi(\hat H_b)(\Id-E)(K_b+K''_b)
(\Id-E)\phi(\hat H_b),
\end{split}\end{equation}
which proves \eqref{eq:Mourre-8p} {\Bl since $\phi(\hat
  H_b)^2=\phi(\hat H_b)^2-F+F,$  $M'_b(\lambda)\geq M_b(\lambda)$
  and $\delta_b(\lambda)=\delta'_b(\lambda)+\delta^b(\lambda)\geq
  \delta^b(\lambda)$.} \Bk Hence the proof of \eqref{eq:Mourre-0} is
complete.

\begin{remark}
The statement $iE[w^bD_{w^b},H^b]E=0$ is formally a consequence of
\begin{equation*}
iE[w^bD_{w^b},H^b]E=iEw^b D_{w^b}(H^b-\lambda-|\xi_b|^2)E
-iE(H^b-\lambda-|\xi_b|^2)w^b D_{w^b}E=0,
\end{equation*}
but this requires justification since $w^b D_{w^b}$ is not bounded
on $L^2(X^b)$. In fact, by elliptic regularity
(namely using $E=E\phi(\hat H_b)$), the only issue is
the lack of decay of $w^b D_{w^b}$ at infinity, but the computation is
justified by replacing $w^b D_{w^b}$ by $\chi_0(|w^b|/C) w^bD_{w^b}$,
$\chi_0\in\Cinf_c(\RR)$ identically $1$ near $0$, and observing that
\begin{equation*}
\phi(\hat H_b)[\chi_0(w^b/C),H]w^bD_{w^b}
\end{equation*}
is uniformly bounded and
goes to $0$ strongly as $C\to\infty$.
\end{remark}

\section{Propagation estimates}\label{prop_est}

{\Bl Propagation estimates in  $N$-body scattering describe the
  propagation of the $N$-body wave front set under generalized broken
  bi-characteristic flow.  The trajectories, called generalized broken
  bi-characteristics(GBBs), are the standard bi-characteristics in the
  free region, away from the proper collision planes $C_a$, which
  project to geodesics on $C_0=\pa X.$ If
  they are normally incident at $C_a$, then they are broken
  bi-characteristics, much as in the reflection of singularities of
  solutions of the wave equation from boundary hypersurfaces or edges,
  which project to broken geodesics. That is,
at their locally unique point in $C_a$ the tangential momentum to
$X_a^*$ is preserved, but the normal component, in $(X^a)^*$, can jump,
subject to the condition that it remains in the characteristic
set.

The definition becomes more complicated for tangentially incident
bi-character-istics, but for $C_a$ which consists of distinct points, as
is the case of wave guides, all the tangential points are in fact
radial. There are no tangential variables along the $C_a$, and the
ambient space $X_a$ only adds the radial direction. Hence, the GBBs in
this case are simply broken bi-characteristics.

With this background, the main propagation estimates of
\cite{VasyAsterisque} in the 3-body case, \cite{Vasy:Propagation-Many}
in the $N$-body case without bound states in the subsystems and
\cite{Vasy:Bound-States} for the general $N$-body case state that the
$N$-body wave front set of $u$, both in the absolute or finite
weighted Sobolev regularity version, is a union of maximally extended
(complete) GBBs. In the special case of the 3-body setting, including
wave guides, the GBBs which are not simply radial points converge (in
the forward, resp.\ backward direction) to either $R^-_{k_0^2,\mf}$ or
$\cup_jR^-_{k^2_0,\ff_j}$, resp.\ $R^+_{k^2_0,\mf}$ or
$\cup_jR^+_{k^2_0,\ff_j}$.

As for the radial sets, much as in the 2-body case (decaying
potentials), there is a distinction between high and low regularity,
which in this context is a measure of the rate of decay. If for some
$\delta>0$ the Sobolev wave front set of $u$, with $-1/2+\delta$
regularity (in the decay sense) is disjoint from say $R^+_{k^2_0,\mf}$
and $\cup_jR^+_{k^2_0,\ff_j}$ then in fact the wave front set is completely
disjoint from $R^+_{k^2_0,\mf}$ and $\cup_jR^+_{k^2_0,\ff_j}$, i.e.\ $u$ is
microlocally Schwartz. This is the content of
\cite[Lemma~17.2]{VasyAsterisque}. It is therefore disjoint from a
neighborhood of these radial sets. We can then use the aforementioned
propagation result along GBBs.

On the other hand, if one does not make an a priori assumption on $u$
at $R^+_{k^2_0,\mf}$ and $\cup_jR^+_{k^2_0,\ff_j}$ beyond being a
tempered distribution and a solution of the $N$-body Helmholtz
equation, then one can propagate the absence of wave front set, of
decay order strictly less than $-1/2$, from a punctured neighborhood of
$R^+_{k^2_0,\mf}$ and $\cup_jR^+_{k^2_0,\ff_j}$ to
  these sets. Of course, $R^+$ may be replaced by
  $R^-$ throughout.

  In particular, in the 3-body setting, assuming that
  $R^+_{k^2_0,\mf}$ and $\cup_jR^+_{k^2_0,\ff_j}$ are disjoint from the
  $-1/2+\delta$, $\delta>0$, wave front set of $u$, one concludes that
  the wave front is contained in $R^-_{k^2_0,\mf}\cup_{j}
  R^-_{k^2_0,\ff_j}$. Furthermore, microlocally, at this locus, $u$ is
  in every weighted Sobolev space of decay order $<-1/2$; this is the
  content of \cite[Lemma~17.7]{VasyAsterisque}.  Further, if one
  assumes that $u$ is in some weighted Sobolev space of order $>-1/2$,
  e.g.\ in $L^2$, one concludes that in fact $\WF_{3\Sc}(u)$ is empty.

  In this section we indicate the changes needed for the proofs of the
  propagation estimates of
  \cite{VasyAsterisque} to accommodate our more general class of
  operators which do not possess a product structure,} \Bk namely
$$
\hat V_a(y_a,\xi_a)=\hat
V^0_a(y_a)+\hat V^1_a(y_a)\xi_a,
$$
with here $y=y_a$ simply, and in the actual wave guide setting (one
dimensional collision plane) there are no $y$ variables, and
$\xi=\tau$. {\Bl We thus have

\begin{theorem}\label{thm:non-product-prop-sing-WG}
Suppose that $H$ is a wave guide 3-body operator, so $C_a$ for $a\neq
0,1$ is discrete. Suppose also that $\lambda=k^2>0$ and that the ideal
Mourre constant is $M'_1(\lambda)>0$. Suppose also that $(H-\lambda)u$
is Schwartz. Then
\begin{enumerate}
\item
  Propagation of
singularities holds, i.e.\ $\WF_{3\Sc}(u)$ is a
union of maximally extended GBBs, and the same holds for all weighted
Sobolev wave front sets, $W^{m,l}_{3\Sc}(u).$
\item
  The high regularity radial estimates hold: if $R^+_{k^2,\mf}$ and
  $\cup_jR^+_{k^2,\ff_j}$ are disjoint from
  $\WF_{3\Sc}^{*,-1/2+\delta}(u)$ for some $\delta>0$, they are also disjoint from
  $\WF_{3\Sc}(u)$. The same holds with $R^+_{k^2,\mf}$ and
  $\cup_jR^+_{k^2,\ff_j}$ replaced by $R^-_{k^2,\mf}$ and
  $\cup_jR^-_{k^2,\ff_j}.$
\item
  The low regularity radial estimates hold: if a punctured
  neighborhood of $R^+_{k^2,\mf}$ and
  $\cup_jR^+_{k^2,\ff_j}$ is disjoint from
  $\WF_{3\Sc}^{*,l}(u)$ for some $l<-1/2$, then so is $R^+_{k^2,\mf}$ and
  $\cup_jR^+_{k^2,\ff_j}$. The same holds for $R^+_{k^2,\mf}$ and
  $\cup_jR^+_{k^2,\ff_j}$ replaced by $R^-_{k^2,\mf}$ and
  $\cup_jR^-_{k^2,\ff_j}$.
\end{enumerate}
  \end{theorem}
}
 \Bk 

  {\Bl As we already noted, we discuss the changes  in the proof of this result relative to
  \cite{VasyAsterisque}; the discussion is from the perspective of the
  general $3$-body problem with this structure, without making an ideal
  Mourre constant assumption, but eventually
  specializing to the wave guide problem when this becomes necessary.}
\Bk 
First, all the elliptic arguments are unchanged, and we already
indicated the changes needed for the Mourre estimate in an even more
general, $N$-body type, setting. The paper \cite{VasyAsterisque}
proceeds with a general construction of potential commutants in
Section~13; the arguments there are unaffected until the proof of
Corollary~13.4. This corollary still holds as stated; the main point is that one needs to observe
that the contributions from\footnote{{\Bl Here, as before, $s$ is
  the boundary defining function denoted by $x$ in
  \cite{VasyAsterisque}.} \Bk } $[s^l,V]\tilde Q$ no longer vanishes, but
as $\tilde Q$ is outside the commutator, the contribution to
$\widehat{[s^l\tilde Q,H]}_{\ff,l+1}(\xi)$ depends on $q$ via
$q(0,0,z,\tau,\mu,\nu)$, so (13.37), estimating $\|R(\xi)\|_{L^2,\Hsc^{1,1}}$, holds
with even $\beta=0$. We also remark that
$V_{\ff}\in\Psisc^{0,1}$ still (since $\xi=\xi_a$ is a parameter), but as we have a spectral localizer
$\psi(H)$ factor, $\psi$ compactly supported, the differential order is
always $-\infty$ in any case, and thus never matters.

Turning to the actual propagation estimates, the potentially affected
part of the normal propagation of
singularities relies on this estimate of
$\|R(\xi)\|_{L^2,\Hsc^{1,1}}$, as well as an argument using that
$\lambda-\tau^2-|\nu|^2_z>0$ is not an eigenvalue of $H_{\ff}$, which
is still true. This allows a shrinking support of spectral cutoff
argument, concerning $\widehat{\tilde\psi(H)}_{\ff}$ around
(14.42). Note in particular that the contribution of $V$ to the
commutator with a regularized weight
$\Lambda_r=s^{-l-1/2}(1+r/s)^{-1}$ is of the same order as that with
$\Delta$, thus these can be treated as  error terms on an equal
basis: in general, for propagation results, making the Hamilton derivative of
the microlocalizer large relative to the microlocalizer itself assures
that the weight commutator terms can be dominated by the
microlocalizer contributions. See the comments above on $q$ being
undifferentiated in the contribution of the new term in the commutator
due to the weight in Corollary~13.4. Hence the normal propagation result, Proposition~14.1, holds
as stated.

The tangential propagation estimate, Proposition~15.2, is proved under
additional conditions in \cite{VasyAsterisque}, which either impose
the lack of bound states in the wave guide, or enforce an independence
of $V_{\ff}$ from $z$, which in the present context would amount to
independence from $(z,\tau,\nu)$, and are thus too strong. However,
Lemma~15.1 remains valid, together with its proof. The more
complicated structure of $V$ would have an effect on Proposition~15.2,
though only if $\lambda-\tau^2-|\nu|^2_z$ is an eigenvalue of
$H_{\ff}$ (the argument otherwise parallels that in the normal case).
The affected part starts around (15.51). This issue is
not due to the weights which can be controlled as in the normal case,
but due to contributions from derivatives in $(z,\tau,\nu)$ of $\hat
H_{\ff}$ in (13.38). Rather than discussing this case further, we
concentrate on our wave guide case, where there are no $(z,\nu)$
variables, thus any point in the tangential part of the characteristic
set $\Sigma_t(\lambda)$ is necessarily in the radial sets, $R^\pm_{k^2,\ff}$,
and the same applies to the negative energy bound states considered in
Section~16.

It remains to discuss Section~17 on the radial sets. Lemma~17.1 is
valid with the same proof, while the high regularity Lemma~17.2 is
affected from around (17.22) on, with a mirrored effect on the low
regularity Lemma~17.7. {\Bl We first indicate how, under smallness assumptions, as in
  \eqref{eq:hat-V-1-est}, one can simply modify the arguments, and later give
  a stronger argument that goes through under our weaker assumptions for the
  Mourre estimate.

  For the simple modification the key
  point being that the argument around
(17.22) already incorporates the weight $s^{-l-1/2}$, so for
completeness also the contribution $[s^{-l-1/2}, \hat
  V^1(y)(s^2D_s)]=\hat V^1(y) [s^{-l-1/2}, s^2D_s]$ needs to be
discussed. We recall first that} \Bk Lemma~17.2 around (17.22) refers to the discussion
around (15.51) (briefly mentioned above), and in our setting of
waveguides, with no $(z,\nu)$ variables, (15.52) still holds as in
(13.38), $\pa_\tau\hat H_{\ff}$ only enters when multiplied by a
$\nu$-derivative of the microlocalizer. Returning to $[s^{-l-1/2}, \hat
  V^1(y)(s^2D_s)]$, this commutator, when multiplied by $s^{l-1/2}$ (this is
the renormalization for the normal operator of that order) is bounded
by {\Bl $(l+1/2)\tilde\delta_{1,a}$ (the subscript $a$ of
$\tilde\delta_{1,a}$ indicates the collision plane, as in
\eqref{eq:hat-V-1-est}).} \Bk
This is then
multiplied by $f$, which is $\chi(\tau)$ in that setting. There is no
helpful term like the one with an arbitrary $M$ in (15.58), 
as we do not have a definite term involving $\chi',$ exactly due to the
radial point. As this $M$ in Section~15 (and 14) is only needed to
control the weights, on which no assumptions are made, (which here
have the correct sign, even with regularization added), this is not an
issue. The conclusion is the desired estimate provided
$2(l+1/2)(\lambda^{1/2}-2\ep)$ (the constant $c_1$ of (17.15) which is
$f^\flat$ at the radial set in $\supp\chi$ in (17.16), and $\ep>0$ is
arbitrary in the construction of $\chi$, and $\lambda$ is the spectral
parameter) is greater than the norm, on $\supp\chi$, of the additional
contribution from above, i.e.\ $(l+1/2)\tilde\delta_{1,a}$, {\Bl i.e.\
  if
  \begin{equation}\label{eq:radial-set-small}
  \tilde\delta_{1,a}<2\lambda^{1/2}.
  \end{equation}

  Note that if $H_a$ has eigenvalues (ignoring the infinite rank and
  compact part of the definition), with the notation of part (3) of
  Definition~\ref{def:microlocal-Mourre-loss}, taking
  $\delta^a(\lambda)$ defined by \eqref{eq:delta^a-def-small-V},
  $\delta^a(\lambda)\geq \sqrt{\lambda}\tilde\delta_{1,a}$, so in
  particular, if $\delta^a(\lambda)<2\lambda$ for all 2-body subsystems
  $a$, then \eqref{eq:radial-set-small} holds. Since
  $M'_1(\lambda)\leq M_0(\lambda)=2\lambda$, this means that if the
  Mourre estimate of Theorem~\ref{Mourre+} holds with a positive
  constant on the right hand side, with $\delta_1'(\lambda)$ defined
  via (3) of Definition~\ref{def:microlocal-Mourre-loss} and Equation
  \eqref{eq:delta^a-def-small-V}, then \eqref{eq:radial-set-small} is
  also satisfied. Thus, in this case the microlocal radiation
  condition
$$
\WF_{3\Sc}(u)\subset R^-_{k^2}\cup
R^-_{k^2}\restrictedto_{\ff_j}
$$
also holds. Moreover if $u$ is a priori in a
weighted Sobolev space with weight strictly larger than $-1/2$ then
$\WF_{3\Sc}(u)=\emptyset$.

Turning to the more significant modification, we need to go back to
the basic commutator computations of \cite{VasyAsterisque} because the
weights $s^l$ there were simple to deal with analytically since they
commuted with the interactions $V$: the lack of this in our more
general setting was already
addressed in the previous paragraph, but only in a perturbative
regime, and a more thorough treatment yields significant benefits. For
this, we need to go back to the basic indicial operator computations
in \cite[Sections~6, 7]{VasyAsterisque}. Concretely, as
\cite[Equation~(7.11)]{VasyAsterisque} indicates, the Schwartz kernel
of the indicial
operator of $A$ at the front face evaluated at parameters
$(z,\tau,\nu)$ with kernel variables\footnote{This identity holds at
  the front face.} $Y=\bar Y-\bar Y',\bar Y$, can be
computed, as
\begin{equation}\label{eq:ind-op-ff}
\hat A_{\ff}(Y,\bar Y;z,\tau,\nu)=\int e^{i(S(-\tau+z\nu)-Z\nu)}A(0,\bar Y,z,S,Y,Z)\,dS\,dZ,
\end{equation}
where $S,Y,Z$ are the scattering front face variables defined near the
diagonal.

In order to avoid confusion, we switch from our notation $s$ to $x,$
used in \cite{VasyAsterisque}, to denote the boundary defining
function.\footnote{In \cite{VasyAsterisque} $s$ denotes $x'/x$.} In
particular $S=\frac{x-x'}{x^2},$ where $x'$ is the pullback of the
boundary defining function from the input, and $x$ from the output
slot, see Equation~(3.7) of \cite{VasyAsterisque}, while $\bar Y=y/x$,
$\bar Y'=y'/x'$ with analogous notation for pullbacks. Thus,
$x'/x=1-xS$, and the Schwartz kernel of the conjugate $x^{-l}Ax^{l}$
is $(x'/x)^l=(1-xS)^{l}$ times the Schwartz kernel of $A$. Since
$x^{-l}[A,x^{l}]=x^{-l}Ax^l-A$, its Schwartz kernel is $(1-xS)^l-1$
times that of $A$, which is $-lxS$ times that of $A$ modulo operators
of two orders decay relative to $A$. Hence by \eqref{eq:ind-op-ff} its
indicial operator is an operator with one additional order of decay
(which factors out the $x$); it is
\begin{equation}\begin{aligned}\label{eq:weight-ind-op-ff}
    &\int e^{i(S(-\tau+z\nu)-Z\nu)}(-lS)A(0,\bar Y,z,S,Y,Z)\,dS\,dZ\\
    &=-il \pa_\tau\int e^{i(S(-\tau+z\nu)-Z\nu)}A(0,\bar Y,z,S,Y,Z)\,dS\,dZ=-il \pa_\tau\hat A_{\ff}(Y,\bar Y;z,\tau,\nu).
  \end{aligned}\end{equation}
In fact, for regularization purposes, with $r>0,$ a small 
regularization parameter, it is useful to point out that the indicial
operator of 
$$
\Big(\frac{x+r}{x}\Big)^l\Big[A,\Big(\frac{x}{x+r}\Big)^l\Big]=\Big(\frac{x+r}{x}\Big)^l A\Big(\frac{x}{x+r}\Big)^l-A
$$
can be similarly computed, uniformly as $r\to 0$, using
$$
\Big(\frac{x+r}{x}\frac{x'}{x'+r}\Big)^l-1=(1-xS)^l\Big(1-\frac{x}{x+r}xS\Big)^{-l}-1,
$$
which is
$$
-lxS\Big(1-\frac{x}{x+r}\Big)=-lxS\frac{r}{x+r}
$$
modulo $x^2$ terms with uniform bounds, to yield
$$
-i l\frac{r}{x+r}\pa_\tau\hat A_{\ff}(Y,\bar Y;z,\tau,\nu).
$$
That is: \ $0\leq \frac{r}{x+r}\leq 1$ times \eqref{eq:weight-ind-op-ff}  (corresponding to commuting
through a fixed weight $x^l$).

Now, in the proof of Lemma~17.2 of \cite{VasyAsterisque} one
introduces
$$
q=x^{-l-1/2}\chi(\tau)\psi_0(x),
$$
where $\psi_0$ localizes near $0$, $\chi=1$ on
$(-\infty,-\lambda^{1/2}+\epsilon)$, supported in
$(-\infty,-\lambda^{1/2}+2\epsilon)$, $\chi'\leq 0$, and takes a
quantization $Q$ of this. Here it is useful to have $Q=x^{-l-1/2}Q_0$,
with $Q_0$ quantizing $\chi(\tau)\psi_0(x)$. While $Q$ is not a
3sc-operator, $Q\psi(H)\in\Psi_{3\Sc}^{-\infty,-l-1/2}$ and
$Q_0\psi(H)\in\Psi_{3\Sc}^{-\infty,0}$. We need to compute
$$
-i[Q\psi(H),H]=-i[x^{-l-1/2},H]Q_0\psi(H)-ix^{-l-1/2}[Q_0\psi(H),H].
$$
By the preceding calculation, the first term on the right hand side
has indicial operator (of decay order $-l+1/2$)
\begin{equation}\label{eq:weight-rad-commutator}
-(l+1/2) (\pa_\tau \hat H_{\ff})\widehat{Q_0\psi(H)}_{\ff,0}=-(l+1/2)
(\pa_\tau \hat H_{\ff})\chi(\tau)\psi(\hat H_{\ff});
\end{equation}
this is precisely extending the first term of Equation~(17.13) of
\cite{VasyAsterisque}, for in that setting $\pa_\tau \hat
H_{\ff}=2\tau$.

Since the indicial operator of the second term is
supported on $\supp\chi'$, thus away from eigenvalues of $\hat
H_{\ff}$, the arguments from the normal propagation of singularities
(including shrinking the support of the spectral cutoff) go through
for its contributions (giving a favorable sign). It remains to
follow \eqref{eq:weight-rad-commutator} through the computations. 
For this we can follow the argument of the proof of our general Mourre
estimate of Section~\ref{MourreEstGenCase}, keeping in mind that
$E\pa_\tau H^a E$ is compact due to the decay of $\hat V^1_a$ for our
2-body subsystems $a$ (so we are always in case (1) or (2) of Definition~\ref{def:microlocal-Mourre-loss}), as long as there exists
$\tilde M^a(\lambda)>0,$ for all subsystem eigenvalues $\lambda'\leq \lambda$ and
corresponding projectors $E$ of $\hat H_a$, so that
$$
E\sign(\tau)(\pa_\tau\hat H_a)(\tau) E\geq \tilde M^a(\lambda) E,\qquad \lambda-\lambda'=\tau^2.
$$ The only difference between $\tilde M^a(\lambda)$ and
$M^a(\lambda)$ is that the latter has an extra factor of $\tau$ in its
definition, i.e.\ $\tau \pa_\tau\hat H_a$. For $\lambda>0$,
$\lambda-\lambda'=\tau^2$, as $\lambda'<0$ and bounded below, we see
that $\tau$ is bounded above and below by positive constants.  Hence,
as long as the ideal Mourre constant in all 2-body subsystems is
positive, it follows that $\tilde M^a(\lambda)>0$.
This proves Theorem~\ref{thm:non-product-prop-sing-WG}.

Thus, even in this
general case, under the hypotheses of
Theorem~\ref{thm:non-product-prop-sing-WG}, if the weak microlocal
radiation condition
$$
\WF_{3\Sc}^{*,-1/2+\delta}(u)\cap(R^+_{k^2}\cup
R^+_{k^2}\restrictedto_{\ff_j})=\emptyset,\ \delta>0,
$$
holds,
then the strong microlocal radiation condition
$$
\WF_{3\Sc}(u)\subset R^-_{k^2}\cup
R^-_{k^2}\restrictedto_{\ff_j}
$$
also
holds, and if $u$ is known to be in a
weighted Sobolev space, with weight strictly larger than $-1/2,$ then
$\WF_{3\Sc}(u)=\emptyset$, i.e.\ $u$ is Schwartz.
} \Bk

\section{Identities from Geometry}\label{app1}
These identities are used to prove~\eqref{eqn32}. If $v_{\alpha}$ is the vector
field dual to the 1-form, $\alpha,$ w.r.t.  the Euclidean metric on $\bbR^3,$ then we let
$\alpha\lrcorner\beta=i_{v_{\alpha}}\beta.$ With this understood we have
\begin{enumerate}
\item 
For $\alpha$ a 1-form, and $\beta$ 1 or 3 form,
\begin{equation}
  \begin{split}
    \alpha\lrcorner \beta&=\star(\alpha\wedge\star\beta)\\
    \alpha\lrcorner \star\beta&=-\star(\alpha\wedge\beta).
  \end{split}
\end{equation}
\item For $\alpha$ a 1-form, and    $\beta$ a 2-form,
\begin{equation}
  \begin{split}
    \alpha\lrcorner \beta&=-\star(\alpha\wedge\star\beta)\\
    \alpha\lrcorner \star\beta&=\star(\alpha\wedge\beta).
  \end{split}
\end{equation}
\item These formul{\ae} imply that for $\alpha$ a 1-, 2- or 3-form and $f$ a function
  \begin{equation}
    d^*(f\alpha)=fd^*\alpha-df\lrcorner\alpha.
  \end{equation}
\item For $\alpha,\beta$ 1-forms,
  \begin{equation}
    \alpha\lrcorner[\alpha\wedge\beta]=[\alpha\lrcorner\alpha]\beta-\alpha[\alpha\lrcorner\beta].
    \end{equation}
\item For a 1-form, $\alpha,$ and  $\beta,$ a $k$-form, Cartan's formula for the Lie derivative is
\begin{equation}
  \cL_{\alpha}\beta=\cL_{v_{\alpha}}\beta=d [i_{v_{\alpha}}\beta]+i_{v_{\alpha}}d\beta.
\end{equation}
Note that for a $\cC^1$-function, $f,$ we have
$$\cL_{f\alpha}\beta=f\cL_{\alpha}\beta+df\wedge i_{v_{\alpha}}\beta.$$
\item For a 1-form, $\beta,$
\begin{equation}
  \star\cL_{\alpha}\star\beta=\cL_{\alpha}\beta-[\beta\,d^*\alpha+\beta\lrcorner d\alpha].
\end{equation}
\end{enumerate}

\section{Acknowledgments}
CLE would like to thank Leslie Greengard, Alex Barnett and Manas Rachh
for very helpful conversations about the material in this paper.

\end{document}